\documentclass[10pt,reqno]{amsart}
\usepackage{pstricks,pst-node,pst-plot,pst-grad}
\usepackage{amssymb,amsfonts,amsmath}
\usepackage{pst-coil}
\usepackage[dvips]{graphicx}
\usepackage[colorlinks]{hyperref}

\newtheorem{theorem}{{\bf Theorem}}[section]

\newtheorem{proposition}[theorem]{{\bf Proposition}}
\newtheorem*{proposition*}{{\bf Proposition}}
\newtheorem{definition}[theorem]{{\bf Definition}}
\newtheorem{lemma}[theorem]{{\bf Lemma}}
\newtheorem{lemma*}{{\bf Lemma}}
\newtheorem{notation}[theorem]{{\bf Notation}}
\newtheorem{convention}[theorem]{{\bf Convention}}

\newtheorem{example}[theorem]{{\bf Example}}
\newtheorem{corollary}[theorem]{{\bf Corollary}}
\newtheorem{remark}[theorem]{{\bf Remark}}

\renewcommand{\H}{\mathsf{H}}

\newcommand{\tst}{\textstyle}
\newcommand{\nn}{\noindent}
\newcommand{\Isom}{\mathrm{Isom}}
\newcommand{\vL}{\vec{L}}
\newcommand{\vF}{\vec{F}}
\newcommand{\vPi}{\vec{\Pi}}
\newcommand{\vS}{\vec{S}}
\newcommand{\A}{\mathbb{A}}
\newcommand{\R}{\mathbb{R}}
\newcommand{\C}{\mathbb{C}}
\newcommand{\Z}{\mathbb{Z}}

\begin{document}

% #############################################
%
%           Title, Authors, etc
%
% #############################################

\title[ Delambre-Gauss formulas in hyperbolic 4-space]%
      { Delambre-Gauss formulas for augmented, right-angled hexagons in hyperbolic 4-space}

\author[S. P. Tan ]{Ser Peow Tan} %
\address{%
        Department of Mathematics \\
        National University of Singapore \\
        119076 \\
        Singapore}
\email{mattansp@nus.edu.sg} %

\author[Y. L. Wong]{Yan Loi Wong} %
\address{%
        Department of Mathematics \\
        National University of Singapore \\
        119076 \\
        Singapore}
\email{matwyl@nus.edu.sg} %

\author[Y. Zhang]{Ying Zhang} %
\address{%
        School of Mathematical Sciences \\
        Soochow University \\
        Suzhou, 215006 \\
        China}
\email{yzhang@suda.edu.cn} %

\dedicatory{To Caroline Series on the occasion of her sixtieth birthday} %

\date{May 27, 2011}

\thanks{Tan was partially supported by the National University
of Singapore academic research grant R-146-000-115-112. Zhang is
supported by the NSFC (China) grant 10871139}

% #############################################
%
%                  Abstract
%

% #############################################

  \begin{abstract}
 We study the geometry of oriented right-angled hexagons in $\H^4$, the hyperbolic $4$-space,
 via Clifford numbers or quaternions.
 We show how to augment alternate sides of such a hexagon so that for the non-augmented sides,
 we can define quaternion half side-lengths whose angular parts  are obtained from half the Euler angles
 associated to a certain orientation-preserving isometry of the Euclidean $3$-space. %
 This generalizes the complex half side-lengths of oriented right-angled hexagons in $\H^3$. %
 We also define appropriate complex half side-lengths for the augmented sides of the hexagon.
 We further explain how to geometrically read off the quaternion half side-lengths for a given oriented,
 augmented, right-angled hexagon in $\H^4$.
 Our main result is a set of generalized Delambre-Gauss formulas for
 oriented, augmented, right-angled hexagons in $\H^4$,
 involving the quaternion half side-lengths and the complex half side-lengths.
 We also show in the appendix how the same method gives
 Delambre-Gauss formulas for oriented right-angled hexagons in $\H^3$,
 from which the well-known sine and cosine laws can be deduced.
 These formulas generalize the classical Delambre-Gauss formulas for spherical/hyperbolic triangles.
% We obtain Delambre-Gauss formulas for
% oriented right-angled hexagons in $\H^3$ and, as the main result, generalized Delambre-Gauss formulas for
% oriented, augmented, right-angled hexagons in $\H^4$, both of which generalize the
% classical Delambre-Gauss formulas for spherical/hyperbolic triangles.
% This is aimed at having applications to the study of representations of surface groups into $\Isom^+(\H^4)$.
 \end{abstract}

 \maketitle

% #############################################
%
%                 Introduction
%
% #############################################

 \vskip 48pt

 \section{\bf Introduction}

 Teichm\"uller theory and the study of right-angled hexagons in the hyperbolic plane $\H^2$ are closely related.
 In particular, in studying the Weil-Petersson metric and the Fenchel-Nielsen coordinates for the Teichm\"uller space,
 a hyperbolic surface is decomposed into pairs of pants with geodesic boundary, which are further decomposed into
 planar right-angled hexagons. The side-lengths of the resulting right-angled hexagons then form part of
 the parameters of the Teichm\"uller space. This can be extended to the study of quasi-Fuchsian space using
 complex Fenchel-Nielsen coordinates in an analogous way,
 using right-angled hexagons in the hyperbolic space $\H^3$ and complex side-lengths instead; %
 see, for example, Tan \cite{tan1994ijm}, Series \cite{series}, and Goldman \cite{goldman}. %
 Parker and Platis \cite{parker-platis} have also made analogous studies on representations of surface groups %
 into the isometry group of complex hyperbolic plane.

 Fundamental to this point of view is that, generically, there is a direct correspondence between equivalence classes of
 irreducible representations $\rho$ of a rank two free group $F_2=\langle X,Y \rangle$ into
 $\Isom^{+}(\H^2)$ (respectively, $\Isom^{+}(\H^3)$) with points in the moduli space of right-angled hexagons in $\H^2$
 (respectively, $\H^3$): the axes of $\rho(X)$, $\rho(Y)$ and $\rho(Y^{-1}X^{-1})$ form the alternate sides of a
 right-angled hexagon, whose (complex) lengths are half the (complex) translation lengths of
 $\rho(X)$, $\rho(Y)$ and $\rho(Y^{-1}X^{-1})$, respectively.

 Things are more complicated for representations of $F_2$ into $\Isom^{+}(\H^4)$,
 where $\H^4$ denotes the hyperbolic 4-space.
 Although one can still define a right-angled hexagon from a generic representation $\rho$ as above, the equivalence
 class of $\rho$ is not determined by the right-angled hexagon, which can be seen by a simple dimension count.
 So there is a non-trivial fibre for the map from the character variety

 \vskip 3pt \centerline{${\rm Hom}(F_2, \Isom^{+}(\H^4))//\Isom^{+}(\H^4)$} \vskip 3pt %

 \nn onto the space of equivalence classes of right-angled hexagons in $\H^4$.
 Thus, the study of this character variety is significantly more difficult and we hope to pursue this in the future. %
 As a preparation for this, this paper is devoted to studying the geometry of right-angled hexagons,
 in particular, to obtaining trigonometric relations such as the law of cosines
 or its equivalent for the geometric configuration determined by $\rho(X)$, $\rho(Y)$ and $\rho(Y^{-1}X^{-1})$. %
%
% \vskip 6pt

 Recall that a convex right-angled hexagon in the hyperbolic plane $\H^2$ is a cyclically ordered sextuple
 $(L_1,\cdots,L_6)$ of complete geodesics in $\H^2$ such that for each $n=1,\,\cdots,6$, the two geodesics
 $L_n$ and $L_{n+1}$ (with indices modulo $6$) intersect perpendicularly (and these are the only intersections %
 among the six geodesics).

 It is well known (see, for example, Beardon's book \cite{beardon1983book}, pp.160--161) that, for a convex
 right-angled hexagon $(L_1,\cdots,L_6)$ in $\H^2$ with side-length $l_n>0$ for the side $L_n$, $n=1,\cdots,6$,
 the Law of Cosines holds: for each $n \; {\rm modulo}\; 6$, %
 \begin{eqnarray}\label{eq:cosineLawH2}
 \cosh l_n = - \cosh l_{n+2} \cosh l_{n+4} + \sinh l_{n+2} \sinh l_{n+4} \cosh l_{n+3}, %
 \end{eqnarray}
 as well as the Law of Sines:
 \begin{eqnarray}\label{eq:sineLawH2}
 \frac{\sinh l_1}{\sinh l_4} = \frac{\sinh l_3}{\sinh l_6} = \frac{\sinh l_5}{\sinh l_2}. %
 \end{eqnarray}

 An oriented right-angled hexagon in the hyperbolic $3$-space $\H^3$ is a cyclically ordered
 sextuple $(\vL_1, \,\cdots,\vL_6)$ of oriented, complete geodesics in $\H^3$
 such that for each $n \; {\rm modulo}\; 6$, the two geodesics $\vL_n$ and $\vL_{n+1}$ intersect perpendicularly.
 % (see Figure ??).
 %\item $\vL_n \neq \pm \vL_{n+2}$, for $n=1, \ldots, 6$ (where $-\vL$ is $\vL$ with the opposite orientation) (is this the right condition or should we omit?).

 Each ordered triple $(\vL_{n-1}, \vL_n, \vL_{n+1})$ defines a complex length
 $\sigma_n$ for the side $\vL_n$, defined up to multiples of $2\pi i$, that is, $\sigma_n \in \C/2\pi i \Z$.
 This is the signed complex translation length of the M\"{o}bius transformation % in $\Isom^+(\H^3)$
 with oriented axis $\vL_n$ which sends $\vL_{n-1}$ to $\vL_{n+1}$
 (see \S \ref{ss:dg4rah} for the precise definition).

 By Fenchel's work (see \cite[Chapter VI]{fenchel1989book}), the complex side-lengths $\{\sigma_n\}_{n=1}^{6}$
 of an oriented right-angled hexagon in $\H^3$ satisfy several trigonometric identities, the most important
 of which are the Law of Cosines: for each $n$ modulo $6$,
 \begin{eqnarray}\label{eq:cosinerule}
 \cosh\sigma_n = \cosh\sigma_{n+2} \cosh\sigma_{n+4} + \sinh\sigma_{n+2} \sinh\sigma_{n+4} \cosh\sigma_{n+3}, %
 \end{eqnarray}
 and the Law of Sines:
 \begin{eqnarray}\label{eq:sinerule}
 \frac{\sinh\sigma_1}{\sinh\sigma_4} = \frac{\sinh\sigma_3}{\sinh\sigma_6} = \frac{\sinh\sigma_5}{\sinh\sigma_2}. %
 \end{eqnarray}
 (Note that the above law of cosines and that of sines for oriented right-angled hexagons in $\H^3$ were known to
 Schilling as early as in 1891 \cite{schilling1891ma}, but the correct treatment of signs and a convincing proof
 seem to be first given by Fenchel in \cite{fenchel1989book}.)
% Note also that the complex side-length for $\vL_n$ is denoted in Fenchel \cite{fenchel1989book} by
% $\delta_n$ which we preserve in this paper for a chosen half side-length modulo $2\pi i$ for $\vL_n$.

 One nice aspect of the law of cosines and the law of sines above for oriented right-angled hexagons in $\H^3$ is that
 they specialize to various well-known identities for relatively simple polygons, convex or self-intersecting, in $\H^2$,
 such as triangles, quadrilaterals with two right angles, pentagons with four right angles,
 and right-angled hexagons; see \cite[\S VI.3]{fenchel1989book} for a complete list of formulas in each case. %
 In fact, this can be seen by observing that each of these planar polygons can be regarded as a subset of
 the sides of a right-angled hexagon in $\H^3$; for example, the three sides of a triangle may be regarded as
 three alternate sides of a right-angled hexagon in $\H^3$. In particular, in the case of a convex, planar
 right-angled hexagon, we may regard it as an oriented right-angled hexagon in $\H^3$ by orienting its side-lines
  consistently. The complex side-lengths are then $l_n + \pi i$, $n=1,\cdots,6$ and formulas
 (\ref{eq:cosinerule}) and (\ref{eq:sinerule}) above simplify to formulas (\ref{eq:cosineLawH2}) and (\ref{eq:sineLawH2}). %

%% (Note that the standard identities for planar
%% right-angled hexagons use real lengths and differ slightly from formulas (\ref{eq:cosinerule}) and
%% (\ref{eq:sinerule}) where the lengths are complex with imaginary parts $0$ or $\pi i$,
%% depending on the choice of orientations for the sides).

%%% These identities specify to the planar cases and give various formulas for a hyperbolic triangle,
%%% a hyperbolic quadrilateral (convex or self intersecting) with two (adjacent or opposite) right angles,
%%% a hyperbolic pentagon (convex or self intersecting) with four right angles,
%%% and a hyperbolic pentagon right-angled hexagon (convex or self intersecting).
%%%%
%%% It is well known that the palnar versions of these formulas play an important role in the study of Riemann surfaces,
%%% Teichm\"uller Theory, Kleinian groups and hyperbolic $3$-manifolds, see for example
%%% (references, e.g., Kerckhoff, Series, Goldman, Tan \cite{tan1994ijm}).
%%%
%%% The identities exist in the literature, dating back to the late 1800's (state some references here?),
%%% where certain ambiguities in the signs were not completely resolved, as pointed out by Fenchel \cite{fenchel1989book}.
%%% A careful proof taking into account the correct signs can be found in \cite{fenchel1989book}.

 On the other hand, recall that, for a spherical triangle in the unit sphere, there are
 the Delambre-Gauss formulas (\ref{eq:gaussspher1})--(\ref{eq:gaussspher4}) (often called Delambre's analogies
 or Gauss formulas separately) for the half side-lengths and half interior angles of the given triangle, from which
 the law of cosines and the law of sines follow easily. Similar formulas hold for hyperbolic triangles
 in the hyperbolic plane. %

 In \S \ref{s:gauss4rah}, we obtain Delambre-Gauss formulas for oriented right-angled hexagons in $\H^3$
 (Theorem \ref{thm:gauss4rah}), from which one may easily obtain (\ref{eq:cosinerule}) and (\ref{eq:sinerule}),
 namely the law of cosines and the law of sines, for oriented right-angled hexagons in $\H^3$. %

 The main purpose of this paper is to obtain analogous formulas for oriented right-angled hexagons in $\H^4$.
 For convenience, we shall use the simple terms {\it lines} and {\it planes}
 to mean respectively complete geodesic lines and totally geodesic $2$-planes in $\H^4$.
% The rest of this paper is motivated by the attempt of trying to obtain analogous formulas for
% oriented right-angled hexagons in $\H^4$.
% We note that a generic, unoriented right-angled hexagon in $\H^3$ or $\H^4$ is determined by three non-adjacent sides,
% since, generically, a pair of complete geodesics have a unique common perpendicular.
% Thus, the generalization of the problem to $\H^4$ is very natural.
 To start with, we observe that for a given oriented right-angled hexagon $(\vL_1, \cdots, \vL_6)$ in $\H^4$, %
 one could similarly define a ``complex length'' for each side $\vL_n$,
 which is invariant under orientation-preserving isometries of $\H^4$.
 The six complex lengths so defined, however, do not determine the right-angled hexagon up to isometries of $\H^4$,
 as the information is incomplete.
 We overcome this problem by considering oriented, \emph{augmented}, right-angled hexagons as follows.
 Starting with an oriented right-angled hexagon $(\vL_1, \ldots, \vL_6)$ in $\H^4$,
 we choose three alternate side-lines, say $\vL_2, \vL_4, \vL_6$, and augment them as oriented line-plane flags
 $\vF_n=(\vL_n, \vPi_n)$, $n=2,4,6$, where $\vF_n$ consists of $\vL_n$ together with a oriented
 plane $\vPi_n$ in $\H^4$ such that $L_n$ is contained in $\Pi_n$, and, crucially, the lines
 $L_{n-1}$ and $L_{n+1}$ are both perpendicular to $\Pi_n$; see Figure \ref{fig:ARAH} for an illustration. %

 %%%%%%%%%%%%%%%%%%%%%%%%%%%%%%%%%%%%%%%%
 %figure 2

 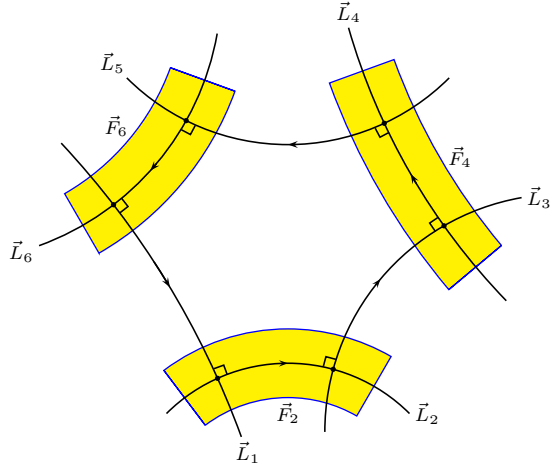
\begin{figure}[h]
 \begin{pspicture}(0,1)(0,7)
 \psset{xunit=13pt}
 \psset{yunit=13pt}
 \psset{runit=13pt}
 \uput{2pt}[270](0,4){\mbox{\scriptsize $\vec{F}_2$}}
 \uput{2pt}[0](4.5,11){\mbox{\scriptsize $\vec{F}_4$}}
 \uput{2pt}[180](-4.5,11.9){\mbox{\scriptsize $\vec{F}_6$}}
 \pscustom[linewidth=0.5pt,linecolor=blue,fillstyle=solid,fillcolor=yellow]{
 \psarcn[linewidth=0.5pt,linecolor=blue](0,0){4}{127}{60}
 \psline[linewidth=0.5pt,linecolor=blue](2,3.464101616)(3,5.196152424)
 \psarc[linewidth=0.5pt,linecolor=blue](0,0){6}{60}{127}
 \psline[linewidth=0.5pt,linecolor=blue](-2.407260093,3.194542040)(-3.610890139,4.791813060)}
 \pscustom[linewidth=0.5pt,linecolor=blue,fillstyle=solid,fillcolor=yellow]{
 \psarcn[linewidth=0.6pt,linecolor=blue](20,20){18}{220}{200}
 \psline[linewidth=0.5pt,linecolor=blue](3.08553283, 13.84363742)(1.20614758,13.15959713)
 \psarc[linewidth=0.6pt,linecolor=blue](20,20){20}{200}{220}
 \psline[linewidth=0.5pt,linecolor=blue](6.21120002,8.42982303)(4.67911114,7.14424781)}
 \pscustom[linewidth=0.5pt,linecolor=blue,fillstyle=solid,fillcolor=yellow]{
 \psarcn[linewidth=0.6pt,linecolor=blue](-10,16){7}{340}{300}
 \psline[linewidth=0.5pt,linecolor=blue](-6.5, 9.937822172)(-5.5, 8.205771364)
 \psarc[linewidth=0.6pt,linecolor=blue](-10,16){9}{300}{340}
 \psline[linewidth=0.5pt,linecolor=blue](-3.422151654, 13.60585900)(-1.542766413, 12.92181871)}
 \psarc[linewidth=0.6pt,linecolor=black](0,0){5}{45}{135}
 \psarc[linewidth=0.6pt,linecolor=black](8,3){6.928203232}{100}{180}
 \psarc[linewidth=0.6pt,linecolor=black](20,20){18.97366596}{196}{224}
 \psarc[linewidth=0.6pt,linecolor=black](0,18){6.633249580}{225}{315}
 \psarc[linewidth=0.6pt,linecolor=black](-10,16){8.062257748}{290}{350}
 \psarc[linewidth=0.6pt,linecolor=black](-25,-5.75){25.16073330}{20}{43}
 \psarc[linewidth=0.5pt,linecolor=black]{<-}(0,0){5}{90}{92}
 \psarc[linewidth=0.5pt,linecolor=black]{<-}(8,3){6.928203232}{140}{142}
 \psarc[linewidth=0.6pt,linecolor=black]{<-}(20,20){18.97366596}{210}{211}
 \psarc[linewidth=0.6pt,linecolor=black]{<-}(0,18){6.633249580}{270}{272}
 \psarc[linewidth=0.6pt,linecolor=black]{<-}(-10,16){8.062257748}{318}{320}
 \psarc[linewidth=0.6pt,linecolor=black]{<-}(-25,-5.75){25.16073330}{31}{33}
 \qdisk(-2.049001550,4.560876303){1pt} %point_61
 \qdisk(-5.076317814,9.615851370){1pt} %point_56
 \qdisk(-2.963067275,12.06533638){1pt} %point_45
 \qdisk(2.801272583,11.98727417){1pt}  %point_34
 \qdisk(4.539072281,9.001831331){1pt}  %point_23
 \qdisk(1.316122624,4.823673003){1pt}  %point_12
 \uput{2pt}[290](-1.35664458,2.855477609){\mbox{\scriptsize $\vec{L}_1$}}
 \uput{2pt}[0](3.5355,3.5355){\mbox{\scriptsize $\vec{L}_2$}}
 \uput{2pt}[0](6.796930137,9.822948256){\mbox{\scriptsize $\vec{L}_3$}}
 \uput{2pt}[90](1.76134168,14.77014888){\mbox{\scriptsize $\vec{L}_4$}}
 \uput{2pt}[130](-4.690415758,13.30958424){\mbox{\scriptsize $\vec{L}_5$}}
 \uput{2pt}[200](-7.242545454,8.423955886){\mbox{\scriptsize $\vec{L}_6$}}
 \rput[bl]{20}(-2.049001550,4.560876303){\psline[linewidth=0.6pt,linecolor=black](0.3,0)(0.3,0.3)(0,0.3)}
 \rput[bl]{-52}(-5.076317814,9.615851370){\psline[linewidth=0.6pt,linecolor=black](0.3,0)(0.3,0.3)(0,0.3)}
 \rput[bl]{-118}(-2.963067275,12.06533638){\psline[linewidth=0.6pt,linecolor=black](0.3,0)(0.3,0.3)(0,0.3)}
 \rput[bl]{-152}(2.801272583,11.98727417){\psline[linewidth=0.6pt,linecolor=black](0.3,0)(0.3,0.3)(0,0.3)}
 \rput[bl]{124}(4.539072281,9.001831331){\psline[linewidth=0.6pt,linecolor=black](0.3,0)(0.3,0.3)(0,0.3)}
 \rput[bl]{75}(1.316122624,4.823673003){\psline[linewidth=0.6pt,linecolor=black](0.3,0)(0.3,0.3)(0,0.3)}
 \end{pspicture}
 \caption{An oriented augmented right-angled hexagon in $\H^4$}\label{fig:ARAH}
 \end{figure}

 %%%%%%%%%%%%%%%%%%%%%%%%%%%%%%%%%%%%%%%%%%%%%%%%%%

 It is always possible to define an oriented, augmented right-angled hexagon from a given oriented
 right-angled hexagon in $\H^4$. In fact, generically, the second condition above ensures that $\Pi_n$
 is well-defined up to orientation, unless $L_{n-1}, L_n$ and $L_{n+1}$ are coplanar, in which case there is
 an ${\mathbb S}^1$ worth (or more) of possible $\vPi_n$'s. We may of course alternatively augment %
 the three odd-indexed side-lines. However, we do not consider a full augmentation of all the side-lines. %

 Consider an oriented, augmented right-angled hexagon $(\vL_1,\vF_2,\vL_3,\vF_4,\vL_5,\vF_6)$. %
 Let $\A_2$ be the Clifford algebra $\mathsf{Cl}_{0,2}$, that is,
 the associative algebra over the reals generated by $e_1, e_2$,
 subject to the relations $e_1e_2+e_2e_1=0$ and $e_1^2=e_2^2=-1$.
 As ungraded algebras, $\A_2$ is isomorphic to the algebra of quaternions.

 For each side-line $\vL_n$, $n=1,3,5$,  the configuration $(\vF_{n-1}, \vL_n, \vF_{n+1})$ allows us to define
 {\it two} quaternion half side-lengths $\delta_n \in \A_2$ modulo period,
 where the two values of $\delta_n$ modulo period differ by $\pi u$ for some $u \in \A_2$ with $u^2=-1$. %
 They can be described by a real translation part and three angles. %
 The real part is uniquely defined to be half the signed translation length of moving the intersection point of
 $L_n$ with $L_{n-1}$ to that of $L_n$ with $L_{n+1}$, measured along the oriented axis $\vL_n$. %
% the same as in the case for complex half-lengths being defined.
 The three angles are determined by the orientation-preserving isometry of the Euclidean $3$-space
 $\R^3 \equiv \R+\R e_1+\R e_2$ induced by parallel translation along $\vL_n$
 which sends the unit tangent vector on the unit sphere ${\mathbb S}^2$ in $\R^3$ determined
 by the oriented flag $\vF_{n-1}$ to that determined by $\vF_{n+1}$;
 explicitly, they are the half Euler angles associated to the above isometry of $\R^3$. %
 On the other hand, for each flag $\vF_n$, $n=2,4,6$, the configuration $(\vL_{n-1}, \vF_n, \vL_{n+1})$
 allows us to define {\it two} $e_2$-complex half side-lengths modulo $2\pi e_2$, denoted by $\delta_n$,
 that is, $\delta_n \in (\R + e_2\R)/2\pi e_2\Z$. %
 Note that the two values of $\delta_n$ modulo $2\pi e_2$ differ by $\pi e_2$.

\medskip
%%%%%%%%%%%%%%%%%%%%%%%%%%%%%%%%%%%%%%%%%%%%%%%%%%%%%%%%%%%%%%%%%%%%%%%%%%%%%%%%%%%%%%%%%%%%%%%%%%%%%%%%%%%%%%%%%%%%%%%%%

 The main result of this paper are the generalized Delambre-Gauss formulas, (\ref{eq:arah1})--(\ref{eq:arah4}) below, %
 for oriented, augmented, right-angled hexagons in $\H^4$.
 We expect that these formulas will prove useful in the study of hyperbolic 4-manifolds
 as well as the representation varieties of surface groups into $\Isom^{+}(\H^4)$. %

 We remark that the asterisk notation $()^*$ appearing in formulas (\ref{eq:arah1})--(\ref{eq:arah4}) %
 denotes the reverse involution of the graded algebra $\A_2$ (see \S \ref{ss:3involutions}): %
 $$ (x_0 + x_1 e_1 + x_2 e_2 + x_{12} e_1e_2)^* := x_0 + x_1 e_1 + x_2 e_2 - x_{12} e_1e_2, $$ %
 with real coefficients $x_0, x_1, x_2, x_{12}$, while the definitions of the hyperbolic functions
 $\cosh$ and $\sinh$ with an $\A_2$-variable will be discussed in \S \ref{ss:cosh-sinh}: %
 in short, we define
 $$ \cosh x :=  \frac{\exp(x) + \exp(-x^*)}{2}, \quad \sinh x :=  \frac{\exp(x) - \exp(-x^*)}{2}. $$

 \begin{theorem}
 [Generalized Delambre-Gauss formulas for  hexagons in $\H^4$]\label{thm:intro-gauss}
 For an oriented, augmented, right-angled hexagon $(\vL_1,\vF_2,\vL_3,\vF_4,\vL_5,\vF_6)$ in $\H^4$ with any choice
 of $\{e_1,e_2\}$-quaternion half side-lengths $\delta_1,\delta_3,\delta_5$ and $e_2$-complex half side-lengths
 $\delta_2, \delta_4, \delta_6$, the following formulas hold: %
 \begin{eqnarray}
 & & \hspace{-50pt} \sinh\delta_1 \cosh\delta_2 \sinh\delta_3 + \cosh\delta_1 \cosh\delta_2 \cosh\delta_3 \nonumber \\ %
 &=& \varepsilon\,(\sinh\delta_4 \cosh\delta_5 \sinh\delta_6 + \cosh\delta_4 \cosh\delta_5 \cosh\delta_6)^*;  \label{eq:arah1} \\ %
 & & \hspace{-50pt} \sinh\delta_1 \cosh\delta_2 \cosh\delta_3 + \cosh\delta_1 \cosh\delta_2 \sinh\delta_3 \nonumber \\ %
 &=& \varepsilon\,(\sinh\delta_4 \sinh\delta_5 \sinh\delta_6 - \cosh\delta_4 \sinh\delta_5 \cosh\delta_6)^*;  \label{eq:arah2} \\ %
 & & \hspace{-50pt} \sinh\delta_1 \sinh\delta_2 \sinh\delta_3 - \cosh\delta_1 \sinh\delta_2 \cosh\delta_3 \nonumber \\ %
 &=& \varepsilon\,(\sinh\delta_4 \cosh\delta_5 \cosh\delta_6 + \cosh\delta_4 \cosh\delta_5 \sinh\delta_6)^*;  \label{eq:arah3} \\ %
 & & \hspace{-50pt} \sinh\delta_1 \sinh\delta_2 \cosh\delta_3 - \cosh\delta_1 \sinh\delta_2 \sinh\delta_3 \nonumber \\ %
 &=& \varepsilon\,(\sinh\delta_4 \sinh\delta_5 \cosh\delta_6 - \cosh\delta_4 \sinh\delta_5 \sinh\delta_6)^*,  \label{eq:arah4} %
 \end{eqnarray}
 with $\varepsilon = 1$ or $-1$, depending on the choices of the six half side-lengths $\{\delta_n\}_{n=1}^{6}$. %
 \end{theorem}

 %%%%%%%%%%%%%%%%%%%%%%%%%%%%%%%%%%%%%%%%%%%%%%%%%%%%%%%%%%%%%%%%%%%%%%%%%%%%%%%%%%%%%%%%%%%%%%%%%%%%%%%%%%%%%%%%%%%%%%

 \begin{remark} {\rm We emphasize that, for a fixed choice of the six half side-lengths,
 it is the same $\varepsilon$ that occurs in the identities.
 For easy memorization, the complicated formulas (\ref{eq:arah1})--(\ref{eq:arah4}) above
 can be conveniently abbreviated as follows:
 \begin{eqnarray*}
 ({\rm scs} + {\rm ccc})_{123} &=& \varepsilon({\rm scs} + {\rm ccc})_{456}^*; \\ %
 ({\rm scc} + {\rm ccs})_{123} &=& \varepsilon({\rm sss} - {\rm csc})_{456}^*; \\ %
 ({\rm sss} - {\rm csc})_{123} &=& \varepsilon({\rm scc} + {\rm ccs})_{456}^*; \\ %
 ({\rm ssc} - {\rm css})_{123} &=& \varepsilon({\rm ssc} - {\rm css})_{456}^*.    %
 \end{eqnarray*}}
 \end{remark}

 \begin{remark} {\rm These formulas rewritten in another form, (\ref{eqn:c1})--(\ref{eqn:c4}), %
 which makes use of the operations $\oplus$ and $\ominus$ in $\A_2$ to be defined later %
 will be given in \S \ref{s:gauss4arah}.} %
 \end{remark}

 %%%%%%%%%%%%%%%%%%%%%%%%%%%%%%%%%%%%%%%%%%%%%%%%%%%%%%%%%%%%%%%%%%%%%%%%%%%%%%%%%%%%%%%%%%%%%%%%%%%%%%%%%%%%%%%%%%%%%%

 \vskip 6pt

 To fully understand the meaning of these identities, and the terms involved, it is useful to understand
 the difficulties involved in attempting this generalization.
%
% To start with, the triple $(\vL_{n-1}, \vL_{n}, \vL_{n+1})$, $n=1, \ldots, 6$ in $\H^4$ only determines
% a ``complex length'' which is invariant under isometries in $\Isom(\H^4)$, that is, up to isometries of $\H^4$,
% this is determined by just a real translation part and an angular invariant $\theta$.
% In fact, the six ``complex lengths'' do not determine the right-angled hexagon in $\H^4$ up to isometries.
% For example, given the complex lengths $\delta_1$ and $\delta_2$, if we place $\vL_6$, $\vL_1$ and $\vL_2$
% into standard position using $\delta_1$, the position of $\vL_3$ is not determined by $\delta_2$ and
% a further ``twist'' invariant is needed to fully determine this.
%
 The main difficulty is related to the fact that the point-wise stabilizer of a complete geodesic in $\Isom^{+}(\H^4)$
 is isomorphic to ${\rm SO}(3)$, which is non-commutative, as opposed to the case of $\H^3$, where the
 point-wise stabilizer of a complete geodesic in $\Isom^{+}(\H^3)$ is isomorphic to ${\rm SO}(2)$ and is commutative.
 To a certain extent, this non-commutativity is reflected by the non-commutativity of $\A_2$, so that the
 representation of elements of $\Isom^{+}(\H^4)$ by Vahlen matrices ($2 \times 2$ matrices with entries in $\A_2$
 satisfying certain
 conditions, following Vahlen \cite{vahlen1902ma} and more recently Ahlfors and his collaborators and
 several others (\cite{ahlfors1984aasfm}--\cite{ahlfors-lounesto1989cvta},\cite{lounesto-latvamaa1980pams},
 \cite{wada1986thesis},\cite{wada1990cvta},\cite{waterman1993advm})), is particularly useful and appropriate.
 This is the approach adopted here. It also has the advantage of shedding some light on the geometry of $\A_2$,
 in relation to $\H^4$ and $\partial \H^4$; indeed some of the results in \S \ref{s:A2} and \S \ref{s:SO3}
 are of independent interest, in particular, the definitions of the functions $\exp$, $\sinh$, $\cosh$ and $\log$ %
 for elements of $\A_2$ and the Euler decomposition of units of $\A_2$ (Proposition \ref{prop:a} and (\ref{eqn:a=e})).
 Nonetheless, this lack of commutativity also means that appropriate
 generalizations of the hyperbolic functions \,sinh\, and \,cosh\, to Clifford numbers (or quaternions)
 need to take careful account of the non-commutativity. We do this in \S \ref{ss:cosh-sinh}
 and it is with respect to these definitions that the identities (\ref{eq:arah1})--(\ref{eq:arah4})
 in Theorem \ref{thm:intro-gauss} should be interpreted.
 The non-commutativity also means that we are unable to obtain analogues of the cosine and sine rules %
 from the Delambre-Gauss formulas in $\H^4$.

 \vskip 6pt

 The rest of this paper is organized as follows. %
 In \S \ref{s:CliffordA_n} we introduce the Clifford algebra $\A_n \equiv \mathsf{Cl}_{0,n}$ and briefly review
 the theory of representing M\"{o}bius transformations of general dimension by $2\times 2$ Vahlen matrices. %
 In \S \ref{s:A2} we focus on the Clifford algebra $\A_2$, and carefully define the exponential function $\exp$,
 the hyperbolic trigonometric functions $\cosh$ and $\sinh$, the logarithmic function $\log$ of an $\A_2$-variable,
 as well as the operations $\oplus$ and $\ominus$ in $\A_2$.
 We also give the Vahlen matrix representations of M\"{o}bius transformations of $\partial \H^4$
 which fix $-1$ and $1$ in terms of the hyperbolic trigonometric functions defined. %
 In \S \ref{s:SO3} we study the geometry of ${\rm SO}(3)$ via the Clifford algebra $\A_2$; in particular,
 we study the Euler decomposition of units in $\A_2$ and explore the algebraic and geometric meaning of
 the associated Euler angles.
 In \S \ref{s:halfdist} we define $\{e_1,e_2\}$-quaternion half side-lengths and the $e_1$- and $e_2$-complex ones
 for configurations $(\vF_{n-1}, \vL_n, \vF_{n+1})$ and $(\vL_{n-1}, \vF_n, \vL_{n+1})$ of %
 oriented lines and flags in $\H^4$. %
 In \S \ref{s:gauss4arah} we prove our main theorem %
 (Theorem \ref{thm:gauss}, rephrasing Theorem \ref{thm:intro-gauss}), the generalized Delambre-Gauss formulas
 involving the $\{e_1,e_2\}$-quaternion and $e_2$-complex half side-lengths for oriented, augmented,
 right-angled hexagons in $\H^4$.
 Finally, in the appendix (\S \ref{s:gauss4rah}) we establish the Delambre-Gauss formulas for
 oriented right-angled hexagons in $\H^3$, %
 whose proof uses the same method used in the proof of Theorem \ref{thm:gauss} %
 but does not require the material in the earlier sections.
 The reader who is familiar with the geometry of $\H^3$
 and would like to understand the geometric idea of the proof of the main theorem
 should first read the appendix.%

 %% and, finally, in \S \ref{s:case} we consider the special case when all the lines and planes of
 %% an augmented right-angled hexagon in $\H^4$ pass through a same point.

 \vskip 6pt

 \nn {\bf Acknowledgements.} The authors would like to thank Bill Goldman, Roger Howe, Sadayoshi Kojima, %
 Fran\c{c}ois Labourie, John Parker, Makoto Sakuma, Caroline Series, %
 Masaaki Wada, Hongyu Wang, Yanlin Yu, and Qing Zhou for helpful conversations. %

%%%%%%%%%%%%%%%%%%%%%%%%%%%%%%%%%%%%%%%%%%%%%%%%%%%%%%%%%%%%%%%%%%%%%%%%%%%%%%%%%%%%%%%%%%%%%%%%%%%%%%%%%%%%%%%%%%%%%
%%%%%%%%%%%%%%%%%%%%%%%%%%%%%%%%%%%%%%%%%%%%%%%%%%%%%%%%%%%%%%%%%%%%%%%%%%%%%%%%%%%%%%%%%%%%%%%%%%%%%%%%%%%%%%%%%%%%%

 \section{\bf The classical Clifford algebras $\A_n$}\label{s:CliffordA_n} %

 \subsection{The classical Clifford algebras $\A_n$}

 Following Ahlfors, we use $\A_n$ to denote the classical Clifford algebra $\mathsf{Cl}_{0,n}$, that is,
 the associative algebra over the reals generated by the elements $e_1, e_2, \cdots, e_n$ subject to the relations
 \begin{eqnarray*}
 && \hspace{-40pt} e_1^2 \,=\, e_2^2 \,=\, \cdots \,=\, e_n^2 \,=\, -1, \\ %
 && \hspace{-40pt} e_i e_j + e_i e_j = 0 \quad \text{for} \,\ i \neq j.
 \end{eqnarray*}
 Thus $\A_n$ is a real algebra of dimension $2^n$ and is a subalgebra of $\A_{n+1}$.

 An element of $\A_n$ can be written uniquely in the form
 \begin{eqnarray}
 a \!\!&=&\!\! a_0 + \sum_{1 \le i_1 \le n} a_{i_1} e_{i_1} %
                   + \sum_{1 \le i_1 < i_2 \le n} a_{i_1i_2} e_{i_1} e_{i_2} %
                   + \sum_{1 \le i_1 < i_2 < i_3 \le n} a_{i_1i_2i_3} e_{i_1} e_{i_2} e_{i_3} + \nonumber \\ %
   & & \hspace{9pt} + \;\cdots \cdots\; + \, a_{12\cdots n} \, e_1 e_2 \cdots e_n %
 \end{eqnarray}
 with real coefficients
 $a_0, \, a_{i_1}, \, a_{i_1i_2}, \, a_{i_1i_2i_3}, \, \cdots, \, a_{i_1i_2 \cdots i_{n-1}}, \, a_{12\cdots n}$.

 Let us write $a^{(p)}$ for the degree $p$ part of $a$, that is, \,$a^{(0)}=a_0$\, and %
 \begin{eqnarray}
 a^{(p)} \; = \sum_{1 \le i_1 < i_2 < \cdots < i_p \le n} %
              a_{i_1 i_2 \cdots i_p} e_{i_1} e_{i_2} \cdots e_{i_p}, \quad\quad p = 1, \cdots, n.
 \end{eqnarray}
 Then
 \begin{eqnarray}
 a = a^{(0)} + a^{(1)} + a^{(2)} + a^{(3)} + \cdots + a^{(n)}.
 \end{eqnarray}

 We call the elements in $\A_n$ of the form $a^{(p)}$ the (degree) $p$-{\it vectors} in $\A_n$;
 in particular, the $0$-vectors are the reals.
 Correspondingly, we have a decomposition of $\A_n$ as the direct sum of its $p$-vector subspaces (with $\A_n^{(0)}=\R$):
 \begin{eqnarray}
 \A_n = \A_n^{(0)} \oplus \A_n^{(1)} \oplus \cdots \oplus \A_n^{(n)}. %
 \end{eqnarray}

 Notice that \,$\A_0 = \R$, \,$\A_1 \cong \C$, and \,$\A_2 \cong {\mathbb H}$, the quaternions. %

 An element of $\A_n$ is called {\it even} if it is the linear combination of even-degree vectors of
 $\A_n$; similarly for {\it odd} elements. The even elements of $\A_n$ form a
 subalgebra $\A_n^+$ of $\A_n$, while the odd elements only form a subspace $\A_n^-$ of $\A_n$.

 We shall conveniently regard $\A_n$ as a Euclidean space with norm
 \begin{eqnarray}
 |a| = \Big(\sum |a_{i_1 i_2 \cdots i_p}|^2\Big)^{1/2}
 \end{eqnarray}
 (where the sum runs over all its coefficients) and the inner product
 \begin{eqnarray}
 \langle a, b \rangle = \sum a_{i_1 i_2 \cdots i_p} b_{i_1 i_2 \cdots i_p}. %
 \end{eqnarray}

 \subsection{Three involutions of $\A_n$}\label{ss:3involutions}
 There are three involutions of $\A_n$.
 The main (or prime) involution $()': \A_n \rightarrow \A_n$, %
 the reverse (or star) involution $()^*: \A_n \rightarrow \A_n$ and %
 the conjugate (or bar) involution \;$\bar{()}\; : \A_n \rightarrow \A_n$ %
 are defined respectively by
 \begin{eqnarray}
      a' \!\!&=&\!\! a^{(0)} - a^{(1)} + a^{(2)} - a^{(3)} + \cdots + (-1)^n a^{(n)}; \\ %
     a^* \!\!&=&\!\! a^{(0)} + a^{(1)} - a^{(2)} - a^{(3)} + \cdots + (-1)^{n(n-1)/2} a^{(n)}; \\ %
 \bar{a} \!\!&=&\!\! a^{(0)} - a^{(1)} - a^{(2)} + a^{(3)} + \cdots + (-1)^{n(n+1)/2} a^{(n)}. %
 \end{eqnarray}
 In particular, for the basis elements, we have
 \begin{eqnarray}
 \hspace{-10pt}(e_{i_1} e_{i_2} \cdots e_{i_p})' \!\!\!&=&\!\!\! (-e_{i_1})(-e_{i_2})\cdots(-e_{i_p})%
 =(-1)^pe_{i_1} e_{i_2} \cdots e_{i_p}; \\ %
 \hspace{-10pt}(e_{i_1} e_{i_2} \cdots e_{i_p})^* \!\!\!&=&\!\!\! e_{i_p} e_{i_{p-1}} \cdots e_{i_1}%
 =(-1)^{p(p-1)/2}e_{i_1} e_{i_2} \cdots e_{i_p}; \\ %
 \hspace{-10pt}\overline{e_{i_1} e_{i_2} \cdots e_{i_p}} \!\!\!&=&\!\!\! ((e_{i_1} e_{i_2} \cdots e_{i_p})')^*%
 =(-1)^{p(p+1)/2}e_{i_1} e_{i_2} \cdots e_{i_p}. %
 \end{eqnarray}

 \begin{remark} {\rm Note that in \cite{lounesto2001lmslns286} the main involution $()'$ is denoted by $()\,\hat{}$
 and the reverse involution $()^*$ is denoted by $()\,\tilde{}$.}
 \end{remark}

 It is easily checked that the composite (in any order) of any two of the three involutions above
 is the remaining one, that is, %
 \begin{eqnarray}
 (a')^* = (a^*)' = \bar{a}, \quad \quad %
 \overline{a^*} = (\bar{a})^* = a', \quad \quad %
 (\bar{a})' = \overline{a'} = a^*.
 \end{eqnarray}
 Note also that the main involution is an algebra isomorphism and the other two involutions are anti-isomorphisms;
 in other words, they are all isomorphisms for addition, and, for multiplication,
 \begin{eqnarray}
 (ab)' = a'b', \quad \quad (ab)^* = b^*a^*, \quad \quad \overline{ab} = \bar{b}\bar{a}.
 \end{eqnarray}

 For the rules of commutativity, we note the following special cases:
 \begin{eqnarray}
 & & a e_i \,=\, e_i a'   \hspace{38pt} \text{if $a$ does not contain $e_i$}; \\ %
 & & a e_i \,=\, -e_i a'  \hspace{30pt} \text{if all terms of $a$ contain $e_i$}. %
 \end{eqnarray}

 \subsection{The group $\A_n^{\times}$ of invertible elements in $\A_n$}

 As usual, an element $a \in \A_n$ is said to be {\it invertible} if there exists an element $b \in \A_n$ such that
 \begin{eqnarray}
  ab \,=\, ba \,=\, 1.
 \end{eqnarray}
 By E. Cartan \cite{cartan}, $\A_n$ is isomorphic to a subalgebra of the matrix algebra ${\rm Mat}(m, \R)$
 for some integer $m$; hence the single equality $ab=1$ implies $ba=1$.
 Such an element $b \in \A_n$, if it exists, is unique and is denoted by $a^{-1}$, as usual.
 The set of all invertible elements of $\A_n$ is a multiplicative group which we denote by $\A_n^{\times}$. %

 As simple examples, it is easy to verify that
 \begin{eqnarray}
 &&\A_0^{\times}=\R^{\times}=\R\backslash\{0\}, \quad
   \A_1^{\times}=\A_1\backslash\{0\}, \quad
   \A_2^{\times}=\A_2\backslash\{0\}, \\ %
 &&\A_3^{\times} = \A_3 \backslash ( (1-e_1e_2e_3)\R \cup (1+e_1e_2e_3)\R ). %
 \end{eqnarray}

 \subsection{The spaces of $1$-vectors and para-vectors of $\A_n$}

 We are interested in the $1$-vector subspace $\A_n^{(1)}$
 and the para-vector subspace $\A_n^{(0,1)}:=\A_n^{(0)} \oplus \A_n^{(1)}$
 which are of (real) dimensions $n$ and $n+1$, respectively.

 It is clear that all the three involutions leave the $1$-vector subspace $\A_n^{(1)}$
 and the para-vector subspace $\A_n^{(0,1)}$ invariant.

 \begin{lemma}
 For $x \in \A_n^{(0,1)}$, we have $x^*=x$, \,$\bar{x}=x'$, and $ x\bar{x} \,=\, \bar{x}x \,=\, |x|^2$, %
 while for $x \in \A_n^{(1)}$, we have $\bar{x}=-x$ and $x^2=-|x|^2$.
 \end{lemma}

 Hence every non-zero para-vector $x \in \A_n^{(0,1)}$ is invertible, with $x^{-1}=\bar{x}/|x|^2$.

 As for the inner product in $\A_n^{(0,1)} \!\equiv \R^{n+1}$, we have
 \begin{eqnarray}
 2 \langle x, y \rangle = x\bar{y} + y\bar{x} \quad\quad \text{for} \;\ x,y \in \A_n^{(0,1)}. %
 \end{eqnarray}

 \subsection{The pure Clifford and full Clifford groups of $\A_n$}
 Following Chevalley \cite{chevalley1954book}, we call the multiplicative group consisting of
 all products of non-zero $1$-vectors in $\A_n$ the (pure) {\it Clifford group} of $\A_n$ and denote it by
 $\Gamma_n^{\rm pure}$; that is
 \begin{eqnarray}
 \Gamma_n^{\rm pure} = \{ a \in \A_n \mid %
                   a = x_1 x_2 \cdots x_m \neq 0, \; x_1, x_2, \cdots, x_m \in \A_n^{(1)}, \; m \ge 0 \}, %
 \end{eqnarray}
 where, as a convention, we regard the null product as $1$.
 (Note that the idea of the definition goes back to R. Lipschitz \cite{lipschitz1886}.)
 Following Ahlfors-Lounesto \cite{ahlfors-lounesto1989cvta}, we call the group consisting of all products
 of non-zero para-vectors in $\A_n$ the {\it full Clifford group} of $\A_n$ and denote it by
 $\Gamma_n^{\rm full}$ or simply $\Gamma_n$; that is, %
 \begin{eqnarray}
 \Gamma_n  = \{ a \in \A_n \mid %
         a = x_1 x_2 \cdots x_m \neq 0, \; x_1, x_2, \cdots, x_m \in \A_n^{(0,1)}, \; m \ge 0 \}. %
 \end{eqnarray}
 Notice that $\Gamma_n^{\rm pure}$ is subgroup of $\Gamma_n$.
 As simple examples, we have
 \begin{eqnarray}
 \Gamma_0 = \A_0^{\times}, \quad \Gamma_1 = \A_1^{\times}, \quad \Gamma_2 = \A_2^{\times} = \A_2 \backslash \{0\}. %
 \end{eqnarray}
% $\Gamma_0 = \A_0^{\times}$, %
% $\Gamma_1 = \A_1^{\times}$, and
% $\Gamma_2 = \A_2^{\times} = \A_2 \backslash \{0\}$. % (the latter case follows from Theorem \ref{thm:lw} bellow).
%
% \begin{eqnarray}
% \Gamma_2 = \A_2^{\times} = \A_2 \backslash \{0\}.
% \end{eqnarray}

 We also define the {\it pure even Clifford group} $\Gamma_n^{+\rm pure}$ of $\A_n$ by
 \begin{eqnarray}
 \hspace{-10pt} \Gamma_n^{+\rm pure}
 \!\!\!&=&\!\!\! \Gamma_n^{\rm pure} \cap \A_n^{+} \nonumber \\ %
 \!\!\!&=&\!\!\! \{ a \in \A_n \mid a = x_1 x_2 \cdots x_{2m} \neq 0, \; x_1, x_2, \cdots, x_{2m}
 \in \A_n^{(1)}, \; m \ge 0 \}. %
 \end{eqnarray}

 The two versions of Clifford groups are related by the following proposition. %

 \begin{proposition}
 $\Gamma_{n+1}^{+\rm pure} \cong \Gamma_n$. %
 \end{proposition}

 \begin{proof}
 Notice that $\A_n \cong \A_{n+1}^{+}$ as real algebras with an isomorphism given by %
 $e_i \leftrightarrow e_ie_{n+1}$, $i=1,2,\cdots,n$. The desired isomorphism follows by observing that %
 \begin{eqnarray*}
 & & (s_1e_1+\cdots+s_ne_n+s_{n+1}e_{n+1})(t_1e_1+\cdots+t_ne_n+t_{n+1}e_{n+1}) \\ %
 &=& (s_1e_1e_{n+1}+\cdots+s_ne_ne_{n+1}-s_{n+1})(t_1e_1e_{n+1}+\cdots+t_ne_ne_{n+1}+t_{n+1}),
 \end{eqnarray*}
 with real coefficients $s_i$ and $t_i$, $i=1,2,\cdots,n+1$.
 \end{proof}

 \begin{proposition}[Waterman \cite{waterman1993advm}]
 {\rm(i)} For $a \in \Gamma_n$, $a\bar{a} = \bar{a}a = |a|^2 \in \R^{\times}$. %and hence $a^{-1}=a^*/|a|^2$. %

 {\rm(ii)} For $a,b \in \A_n$ with $a\in\Gamma_n$ or $b\in\Gamma_n$, $|ab|=|a||b|$.
 \end{proposition}

% \begin{lemma}[Vahlen-Maass]\label{lem:wm}
% Suppose $a=b+ce_n$ where $b,c \in \A_{n-1}\backslash\{ 0 \}$.
% Then $a \in \Gamma_n$ if and only if \,$b,c \in \Gamma_{n-1}$ and \,$c^{-1}b \in \A_{n-1}^{(0,1)}$. \qed %
% \end{lemma}

% \nn {\bf Remarks.} \ The following items (a)--(e) are easily checked to be true. %
% \begin{itemize}
% \item[(a)] $bb^*, cc^*\in \A_{n-1}^{(0,1)}$.
% \item[(b)] $\bar{c}b\in\A_{n-1}^{(0,1)}$ and hence $\bar{c}b+\bar{b}c\in \R$.
% \item[(c)] $b\bar{c}\in\A_{n-1}^{(0,1)}$ and hence $b\bar{c}+c\bar{b}\in \R$.
% \item[(d)] $c(\bar{c}b+\bar{b}c)\bar{c}=c\bar{b}c\bar{c}+c\bar{c}b\bar{c}=(b\bar{c}+c\bar{b})c\bar{c}$. %
% \item[(e)] $a=b+ce_n=(c-be_n)e_n$.
% \end{itemize}

 The following deep theorem says that an element of $\Gamma_n$ with real part $1$ is
 actually determined by its $1$- and $2$-parts; for a proof, % of Theorem \ref{thm:lw},
 see \cite{lipschitz1959am} or \cite{ahlfors-lounesto1989cvta}.

 \begin{theorem}[Lipschitz-Vahlen]\label{thm:lw}
 Given arbitrary $u \in \A_n^{(1)}$ and $v \in \A_n^{(2)}$, there exists exactly one element $a \in \Gamma_n$
 with $a^{(0)}=1$, $a^{(1)}=u$ and $a^{(2)}=v$. %
 \end{theorem}

 As a consequence of Theorem \ref{thm:lw}, we conclude that the full Clifford group $\Gamma_n$
 is a Lie group of dimension $1+n+n(n-1)/2=(n^2+n+2)/2$. On the other hand, it can be shown that
 the group $\A_n^{\times}$ of invertible elements in $\A_n$ is a Lie group of dimension $2^n$.
 Hence we have $\Gamma_n \subset \A_n^{\times}$ and $\Gamma_n \neq \A_n^{\times}$ for $n \ge 3$. %

 To give an explicit example of invertible elements of $\A_n$ which are not in
 $\Gamma_n$, consider the element $a=1+te_1e_2\cdots e_n \in \A_n$ with $t\in \R\backslash\{-1,1\}$.
 Then $a$ is invertible since
 $(1+te_1e_2\cdots e_n)(1-te_1e_2\cdots e_n)=1 \pm t^2 \in \R^{\times}$. %
 That $a \notin \Gamma_n$ (with $n \ge 3$ and $t \neq 0$) follows immediately from
 Theorem \ref{thm:lw} since $1 \in \Gamma_n$ and $a \neq 1$. %
 % or from the fact that $$ a\bar{a}=(1+te_1e_2e_3)^2=(1+t^2)+2t e_1e_2e_3 \notin \R. $$ %

 \subsection{Another characterization of the full Clifford group $\Gamma_n$}

 It is easy to show (say, by induction) that if $a \in \Gamma_n$ then for all $x \in \A_n^{(0,1)}$,
 $ax(a')^{-1} \in \A_n^{(0,1)}$. %
Conversely, we have:
 \begin{proposition}[Vahlen-Maass]\label{prop:amap}
 If $a \in \A_n$ is invertible and for all $x \in \A_{n}^{(0,1)}$, $ax(a')^{-1} \in \A_{n}^{(0,1)}$, %
 then $a \in \Gamma_n$. \qed %
 \end{proposition}

 Thus we obtain the following characterization of the full Clifford group: %
 \begin{eqnarray}
 \Gamma_n \!\!\!&=&\!\!\! \{ a \in \A_n^{\times} \mid %
 ax(a')^{-1} \in \A_{n}^{(0,1)} \,\; \text{for all} \,\; x \in \A_{n}^{(0,1)} \}. %
 \end{eqnarray}
 Note that $\Gamma_n$ is a subgroup of $\Gamma_{n+1}$. In particular, we have, for $a \in \Gamma_n$, %
 $$ ae_{n+1}(a')^{-1} = e_{n+1}a'(a')^{-1} = e_{n+1}. $$ %

 \subsection{The isomorphism ${\rm SO}(n+1) \cong \Gamma_n / \R^{\times}$}

 An element $a \in \Gamma_n$ defines a linear transformation
 $\rho(a): \A_{n}^{(0,1)} \rightarrow \A_{n}^{(0,1)}$ by %
 \begin{eqnarray}
 \rho(a)x = ax(a')^{-1} \quad\quad \text{for} \;\ x \in \A_{n}^{(0,1)}. %
 \end{eqnarray}
 Since $|\rho(a)x|^2=(ax(a')^{-1})(ax(a')^{-1})'=ax(a')^{-1}a'x'a^{-1}=|x|^2$ for all $x \in \A_{n}^{(0,1)}$,
 we see that $\rho(a) \in {\rm O}(n+1)$. It can be shown that $\rho(a) \in {\rm SO}(n+1)$
 (say, first prove the conclusion for $a=1,e_1,\cdots,e_n$,
 then for $a\in\A_{n}^{(0,1)}\backslash\{0\}$ by continuity, and finally for $a\in\Gamma_n$ by induction). %
 This defines a homomorphism of groups
 $$ \rho: \Gamma_n \longrightarrow {\rm SO}(n+1). $$ %
 It can be shown that $\rho$ is an epimorphism with kernel $\R^{\times}$.
 Thus we have obtained the following isomorphisms of groups: %
 \begin{eqnarray}
 {\rm Spin}(n+1) \, \cong \, \Gamma_n / \R^{+} \quad \text{and} \quad
 {\rm SO}(n+1) \, \cong \, \Gamma_n / \R^{\times}. %
 \end{eqnarray}

 \subsection{The one-point compactification $\hat{\A}_n^{(0,1)}$ of $\A_n^{(0,1)}$}
 Following Ahlfors, we identify the Euclidean space $\R^{n+1}$ with $\A_n^{(0,1)}$,
 and hence the one-point compactification $\hat{\R}^{n+1}$ of $\R^{n+1}$ with
 \begin{eqnarray}
 \hat{\A}_n^{(0,1)}:=\A_n^{(0,1)} \cup \{ \infty \}, %
 \end{eqnarray}
 where, as usual, the symbol $\infty:=0^{-1}$ operates as follows:
 \begin{eqnarray}
 && \infty + a = a + \infty = \infty \quad\quad  \text{for} \;\ a \in \A_n; \\ %
 && \infty \cdot \infty = \infty; \\ %
 && \infty \cdot a = a \cdot \infty = \infty \quad\quad\quad \text{for} \,\ a \in \A_n^{\times}; \\ %
 && \infty^{-1} = 0,
 \end{eqnarray}
 and the operations $\infty \pm \infty$, $\infty \cdot 0$ and $0 \cdot \infty$ are forbidden.

 \subsection{An observation of Ahlfors}\label{ss:ahlfors}

 Ahlfors made the following useful observation: %

 \begin{lemma}[Ahlfors \cite{ahlfors1984aasfm}]\label{lem:obs}
 Suppose $a,b \in \Gamma_n$. Then $a^{-1}b \in \A_n^{(0,1)}$ if and only if $ab^* \in \A_n^{(0,1)}$,
 while $ab^{-1} \in \A_n^{(0,1)}$ if and only if $a^*b \in \A_n^{(0,1)}$. \qed %
 \end{lemma}

 \subsection{The group ${\rm SL}(2,\Gamma_n)$ of Vahlen matrices of dimension $n$}

 A {\it Vahlen matrix} of dimension $n$ is a matrix
 $\Big(\,\begin{matrix} a & b \\ c & d \end{matrix}\,\Big)$ such that
 \begin{itemize}
 \item[(i)] $a, b, c, d \in \Gamma_n \cup \{0\}$;
 \vskip 3pt

 \item[(ii)] $ad^*-bc^*=1$;

 \item[(iii)] $a^{-1}b, c^{-1}d, ac^{-1}, bd^{-1} \in \hat{\A}_n^{(0,1)}=\A_n^{(0,1)}\cup\{\infty\}$. %
 \end{itemize}
 By Ahlfors' observation (Lemma \ref{lem:obs}), the condition $a^{-1}b \in \hat{\A}_n^{(0,1)}$
 in (iii) is equivalent to $ab^* \in \A_n^{(0,1)}$, and similarly for the other conditions in (iii).

 On the other hand, while keeping (i) and (ii) unchanged, one can drop
 any two except the last two of the four requirements in (iii).
 Precisely, we have

 \begin{proposition}[Ahlfors \cite{ahlfors1985dgca}\cite{ahlfors1985aasfm}]\label{prop:(iii)}
 Suppose $a, b, c, d \in \Gamma_n \cup \{0\}$ and $ad^*-bc^*=1$. Then %
 \begin{itemize}
 \item[(a)] $a^{-1}b, c^{-1}d \in \hat{\A}_n^{(0,1)} \,  \Longrightarrow ac^{-1}, bd^{-1} \in \hat{\A}_n^{(0,1)}$; %
 \item[(b)] $a^{-1}b, ac^{-1} \in \hat{\A}_n^{(0,1)} \Longleftrightarrow c^{-1}d, bd^{-1} \in \hat{\A}_n^{(0,1)}$; %
 \item[(c)] $a^{-1}b, bd^{-1} \in \hat{\A}_n^{(0,1)} \Longleftrightarrow c^{-1}d, ac^{-1} \in \hat{\A}_n^{(0,1)}$. %
 \end{itemize}
 \end{proposition}

% Note that the equivalence in part (b) was first obtained by Ahlfors.
%
% \begin{proof} When at least one of $a,b,c,d$ is zero, the proof is
% simple and hence omitted. Therefore we may assume that $a,b,c,d \in \Gamma_n$. %
%
% (a) Suppose $a^{-1}b=u \in \A_n^{(0,1)}$ and $c^{-1}d=v \in \A_n^{(0,1)}$.
% Then $b=au$ and $d=cv$. By (ii), we have $1=a(cv)^*-(au)c^*=avc*-auc^*=a(v-u)c^*$ and hence %
% $$ (c^*a)^{-1}=a^{-1}(c^*)^{-1}=v-u \in \A_n^{(0,1)}. $$ %
% It follows that $a^*c=c^*a \in \A_n^{(0,1)}$. By Ahlfors' observation, $ac^{-1} \in \A_n^{(0,1)}$.
% To prove that $bd^{-1} \in \A_n^{(0,1)}$,
% let us write $x=b^{-1}a \in \A_n^{(0,1)}$ and $y=d^{-1}c \in \A_n^{(0,1)}$. %
% Then $a=bx$ and $c=dy$. By (ii), we have $1=(bx)d^*-b(dy)^*=bxd^*-byd^*=b(x-y)d^*$ and hence
% $$ (d^*b)^{-1}=b^{-1}(d^*)^{-1}=x-y \in \A_n^{(0,1)}. $$ %
% It follows that $b^*d=d^*b \in \A_n^{(0,1)}$.
% By Ahlfors' observation, $bd^{-1} \in \A_n^{(0,1)}$.
% \end{proof}

 \begin{example}\label{ex} {\rm We give a simple example to show that the
 converse implication of part (a) in Proposition \ref{prop:(iii)} is not true, %
 that is,
\begin{eqnarray}
 ac^{-1}, bd^{-1} \in \hat{\A}_n^{(0,1)} \Longrightarrow\!\!\!\!\!\!\!\!\!/ \phantom{00} a^{-1}b, c^{-1}d \in \hat{\A}_n^{(0,1)}. %
\end{eqnarray} %
 For this, let $n=2$ and
 $$ a=d=1+\textstyle\frac{\sqrt{2}}{2}e_1, \quad b=c=\big( 1-\textstyle\frac{\sqrt{2}}{2}e_1 \big) e_2 \in \Gamma_2. $$ %
 Then $ad^*-bc^*=1$,
 $a^*c=b^*d=\frac32e_2 \in \A_2^{(0,1)}$ and $ab^*=cd^*=(\frac12+\sqrt{2}e_1)e_2 \not\in \A_2^{(0,1)}$. %
 By Ahlfors' observation, $ac^{-1}, bd^{-1} \in \hat{\A}_2^{(0,1)}$ and $a^{-1}b, c^{-1}d \not\in \hat{\A}_2^{(0,1)}$.} %
 \end{example}

% (iii) above can be replaced by any one of the following conditions (iii-1)--(iii-4):
% \begin{itemize}
% \item[(iii-1)] $ac^{-1}, a^{-1}b \in \hat{\A}_n^{(0,1)}$. %
%
% \item[(iii-2)] $bd^{-1}, c^{-1}d \in \hat{\A}_n^{(0,1)}$. %
%
% \item[(iii-3)] $a^*c, \,ab^* \in \A_n^{(0,1)}$.
%
% \item[(iii-4)] $b^*d, \,cd^* \in \A_n^{(0,1)}$.
% \end{itemize}

 \begin{proposition}\label{prop:gp}
 The Vahlen matrices of dimension $n$ form a group under matrix multiplication,
 with the inverse of $A=\Big(\,\begin{matrix} a & b \\ c & d \end{matrix}\,\Big)$ given by %
 \begin{eqnarray}
 A^{-1}=\begin{pmatrix} \phantom{-}d^* & -b^*  \\ -c^* & \phantom{-}a^* \end{pmatrix}. %
 \end{eqnarray}
 \end{proposition}

 For a detailed proof of Proposition \ref{prop:gp}, see Waterman \cite{waterman1993advm}.

 \begin{notation}
 {\rm We denote by ${\rm SL}(2,\Gamma_n)$ the multiplicative group of all Vahlen matrices of dimension $n$.} %
 \end{notation}

 Since $\Gamma_n \subset \Gamma_{n+1}$,
 it follows that ${\rm SL}(2,\Gamma_n)$ is a subgroup of ${\rm SL}(2,\Gamma_{n+1})$. %

\subsection{M\"{o}bius transformations of $\hat{\A}_n^{(0,1)}$ via Vahlen matrices}

 A M\"{o}bius transformation of $\hat{\A}_n^{(0,1)}$ is defined to be a conformal
 automorphism of $\hat{\A}_n^{(0,1)}$, or equivalently, the composition of
 an even number of inversions in $n$-spheres and reflections in hyperplanes in $\hat{\A}_n^{(0,1)}$.
 In 1902, Vahlen \cite{vahlen1902ma} initiated the study of M\"{o}bius
 transformations of $\hat{\A}_n^{(0,1)}$ via Vahlen matrices.
 This study was later revived in 1949 by Maass \cite{maass1949amsuh} and
 re-initiated by Ahlfors in the 1980's (see \cite{ahlfors1984aasfm}--\cite{ahlfors-lounesto1989cvta}).

 A Vahlen matrix $A=\Big(\,\begin{matrix} a & b \\ c & d \end{matrix}\,\Big)$ of dimension $n$ gives rise to a
 M\"{o}bius transformation $T_A$ of $\hat{\A}_n^{(0,1)}$ defined by %
 \begin{eqnarray}
 && T_A(x) = (ax+b)(cx+d)^{-1} \in \hat{\A}_n^{(0,1)}, \quad\quad x \in \A_n^{(0,1)}; \\ %
 &&T_A(\infty) = ac^{-1}.
 \end{eqnarray}
 Note that $T_A(\infty)$ equals the limit of $T_A(x)$ (where $x\in\A_n^{(0,1)}$) as $x\rightarrow\infty$. %
 To see that $T_A(x) \in \hat{\A}_n^{(0,1)}$ for $x \in \A_n^{(0,1)}$,  first note that this is the case when
 $$ A = \begin{pmatrix}\,1 & y \\ 0 & 1\end{pmatrix}, \,\ %
        \begin{pmatrix}\,g & 0\phantom{aaa} \\ \,0 & {g^*}^{-1}\!\!\end{pmatrix} %
        \,\ \text{or} \,\
        \begin{pmatrix}0 & \!\!\!{-1} \\ 1 & 0\end{pmatrix} $$ %
 where $y\in\A_n^{(0,1)}$ and $g\in\Gamma_n$.
 For a general $A\in{\rm SL}(2,\Gamma_n)$, this can be seen from the following decomposition of a general Vahlen
 matrix into simpler ones.

 \begin{proposition}[Ahlfors \cite{ahlfors1985dgca}\cite{ahlfors1985aasfm}\cite{ahlfors1986cvta}]
 A Vahlen matrix $\Big(\,\begin{matrix} a & b \\ c & d \end{matrix}\,\Big)\in{\rm SL}(2,\Gamma_n)$
 can be decomposed into a product of simple ones as follows: %
 \begin{eqnarray}
&&\hspace{-20pt}\begin{pmatrix}a & b \\ c & d\end{pmatrix} %
 =\begin{pmatrix}1 & ac^{-1} \\ 0 & 1\end{pmatrix} %
  \begin{pmatrix}{c^*}^{-1} & 0 \\ 0 & c\end{pmatrix} %
  \begin{pmatrix}0 & \!\!\!{-1} \\ 1 & 0\end{pmatrix} %
  \begin{pmatrix}1 & c^{-1}d \\ 0 & 1\end{pmatrix} %
  \quad \text{if} \;\; c\neq 0; \label{eqn:cneq0} \\ %
&&\hspace{-20pt}\begin{pmatrix}a & b \\ 0 & d\end{pmatrix} %
 =\begin{pmatrix}a & 0\phantom{aaa}\! \\ 0 & {a^*}^{-1}\!\end{pmatrix} %
  \begin{pmatrix}1 & a^{-1}b \\ 0 & 1\end{pmatrix} %
  \quad \text{if} \;\; c=0 \;(\text{then} \;\; d={a^*}^{-1}). \label{eqn:c=0} %
 \end{eqnarray}
 \end{proposition}

 Notice that $-A$ gives rise to the same M\"{o}bius transformation as $A$ does.
 Furthermore, it is easy to verify that each M\"{o}bius transformation of
 $\hat{\A}_n^{(0,1)}$ is given by a Vahlen matrix and that
 $A \in {\rm SL}(2,\Gamma_n)$ gives rise to the identity transformation of $\hat{\A}_n^{(0,1)}$
 if and only if $A=\pm I$. Thus a M\"{o}bius transformation is given by exactly two
 Vahlen matrices $\pm A$ and we have the isomorphism of groups
 \begin{eqnarray}
 \text{M\"{o}b}(\hat{\A}_n^{(0,1)}) \,\cong\, {\rm PSL}(2,\Gamma_n):= {\rm SL}(2,\Gamma_n) / \{ \pm I \}. %
 \end{eqnarray}

 Observe that the inclusion of groups ${\rm SL}(2,\Gamma_n) \subset {\rm SL}(2,\Gamma_{n+1})$
 induces the Poincar\'{e} extension of M\"{o}bius transformations:
 $\text{M\"{o}b}(\hat{\A}_n^{(0,1)}) \subset \text{M\"{o}b}(\hat{\A}_{n+1}^{(0,1)})$.

\subsection{The group ${\rm Isom}^+(\H^{n+2})$ }
 Since the Euclidean space $\R^{n+2}$ is identified with $\A_{n+1}^{(0,1)}$,
 the upper half-space model of the hyperbolic $(n+2)$-space $\H^{n+2}$ is given by %
 \begin{eqnarray}
 \H^{n+2} = \Big\{ x \in \A_{n+1}^{(0,1)} \mid %
        x = x_0 + \sum_{i=1}^{n+1}x_ie_i, \, x_0, x_1, \cdots, x_{n+1}\in \R, \, x_{n+1}>0 \Big\}, %
 \end{eqnarray}
 equipped with Riemannian metric \,$ds^2 = (dx_0^2 + dx_1^2 + \cdots + dx_{n+1}^2)/x_{n+1}^2$
 of constant sectional curvature $-1$.
 Its boundary at infinity, $\partial \H^{n+2}$, is then identified with
 $\hat{\A}_{n}^{(0,1)} = \A_{n}^{(0,1)} \cup \{ \infty \}$. %
 It is well-known that a totally complete geodesic $m$-plane, $1 \le m \le n+1$,
 is the upper half of either a Euclidean $m$-plane or a Euclidean $m$-sphere,
 both orthogonal to $\A_{n}^{(0,1)} \equiv \R^{n+1}$. %
 In this model the orientation-preserving isometries of $\H^{n+2}$
 are exactly the M\"{o}bius transformations of $\hat{\A}_{n}^{(0,1)}$ extended to $\H^{n+2}$. %
 Thus we have the isomorphisms of groups
 \begin{eqnarray}
 {\rm Isom}^+(\H^{n+2}) \,\cong\, \text{M\"{o}b}(\hat{\A}_n^{(0,1)}) \,\cong\, {\rm PSL}(2,\Gamma_n). %
 \end{eqnarray}

\subsection{\bf M\"{o}bius transformations of $\hat{\A}_{n}^{(0,1)}$ fixing both $0$ and $\infty$}
% It is easy to know that

 \begin{proposition}
 A M\"{o}bius transformation of $\hat{\A}_{n}^{(0,1)}$ fixes both $0$ and $\infty$
 if and only if its Vahlen matrices $\pm A$ are of the form %
 \begin{eqnarray}
 A = \begin{pmatrix}\;a & 0\phantom{aaa} \\ \;0 & {a^*}^{-1} \end{pmatrix}, \quad a \in \Gamma_n. %
 \end{eqnarray}
 \end{proposition}
 By Poincar\'{e} extension, $A$ gives a M\"{o}bius transformation of $\hat{\A}_{n+1}^{(0,1)}$ with %
 $$ T_A (e_{n+1}) = ae_{n+1}a^* = a\bar{a}e_{n+1} = |a|^2e_{n+1}; $$ %
 in particular, $T_A (e_{n+1}) = e_{n+1}$ if and only if $|a|=1$.

\subsection{\bf M\"{o}bius transformations of $\hat{\A}_n^{(0,1)}$ fixing both $-1$ and $1$}\label{ss:1-1} %

 \begin{proposition}\label{prop:1-1}
 A M\"{o}bius transformation of $\hat{\A}_n^{(0,1)}$ fixes both $-1$ and $1$
 if and only if its Vahlen matrices $\pm A$ are of the form %
 \begin{eqnarray}
 A = \left(\begin{matrix}\,a & b\, \\ \,b & a\, \end{matrix}\right) %
 \end{eqnarray}
 where $a,b \in \Gamma_n \cup \{0\}$ satisfy $ab^*\in\A_n^{(0,1)}$ and $aa^*-bb^*=1$.
 \end{proposition}

 Note that $ab^*\in\A_n^{(0,1)}$ implies $a^*b\in\A_n^{(0,1)}$ by Proposition \ref{prop:(iii)}.
 However, $ab^*\in\A_n^{(0,1)}$ does not imply $a^*b\in\A_n^{(0,1)}$, as shown by Example \ref{ex}. %

 \subsection{\bf M\"{o}bius transformations of $\hat{\A}_n^{(0,1)}$ fixing $e_{n+1}$}\label{ss:en+1} %

 \begin{proposition}\label{prop:en+1}
 A M\"{o}bius transformation of $\hat{\A}_n^{(0,1)}$ fixes $e_{n+1}\in\A_{n+1}^{(0,1)}$
 if and only if its Vahlen matrices $\pm A$ are of the form %
 \begin{eqnarray}
 A = \left(\begin{matrix} \; a & b\, \\ -b' & a' \end{matrix}\right) %
 \end{eqnarray}
 where $a,b \in \Gamma_n \cup \{0\}$ satisfy $ab^*\in\A_n^{(0,1)}$ and $|a|^2+|b|^2=1$. %
 \end{proposition}

 \begin{proposition}\label{prop:en+1fix}
 The subgroup of {\rm M\"{o}b}$(\hat{\A}_n^{(0,1)})$ consisting of M\"{o}bius transformations
 fixing $e_{n+1}\in\A_{n+1}^{(0,1)}$ is isomorphic to ${\rm SO}(n+2)$. %
 \end{proposition}

 %%%%%%%%%%%%%%%%%%%%%%%%%%%%%%%%%%%%%%%%%%%%%%%%%%%%%%%%%%%%%%%%%%%%%%%%%%%%%%%%%%%%%%%%%%%%%%%%%%%%%%%%%%%%%%%%%%%%%%%%

 \section{\bf The Clifford algebra $\A_2$}\label{s:A2}%

 \nn As our aim is to study the geometry of  hyperbolic $4$-space, we shall be focusing on the case $n=2$
 in the rest of this paper.

 \subsection{A brief account of $\A_2$}

 Let us first briefly recall the basic facts that will be used frequently in the rest of this paper. %
 As associative algebras over the reals, the Clifford algebra
 $\A_2 = \mathsf{Cl}_{0,2} = \R + \R e_1 + \R e_2 + \R e_1 e_2$ (with $e_1^2=e_2^2=-1$ and $e_1e_2+e_2e_1=0$)
 is isomorphic to the algebra
 ${\mathbb H} = \R + \R {\bf i} + \R {\bf j} + \R {\bf k}$ of quaternions,
 with an algebra isomorphism given by
 \begin{eqnarray}
     e_1 \longleftrightarrow {\bf i}, \quad %
     e_2 \longleftrightarrow {\bf j}, \quad %
  e_1e_2 \longleftrightarrow {\bf k}.       %
 \end{eqnarray}
 However, $\A_2$ is a graded algebra, equipped with three involutions; %
 explicitly, for $a=a_0+a_1e_1+a_2e_2+a_{12}e_1e_2\in\A_2$ with $a_0,a_1,a_2,a_{12}\in\R$,
 the images of $a$ under the three involutions $()'$, $()^*$ and $\bar{()}$ of $\A_2$ are respectively given by
 \begin{eqnarray}
      a' \!\!&=&\!\! a_0-a_1e_1-a_2e_2+a_{12}e_1e_2; \\ %
     a^* \!\!&=&\!\! a_0+a_1e_1+a_2e_2-a_{12}e_1e_2; \\ %
 \bar{a} \!\!&=&\!\! a_0-a_1e_1-a_2e_2-a_{12}e_1e_2. %
 \end{eqnarray}
 An element $a \in \A_2$ is invertible if and only if $a \neq 0$; in particular, $a^{-1}=|a|^{-2}\bar{a}$.
 Hence the group of invertible elements of $\A_2$ is $\A_2^{\times} = \A_2 \backslash \{0\}$. %
 We denote the subgroup of $\A_2^{\times}$ consisting of all the unit elements of $\A_2$ by $\A_2^{\rm unit}$,
 that is, %
 \begin{eqnarray}
 \A_2^{\rm unit} = \{ a \in \A_2 \mid |a|=1 \}. %
 \end{eqnarray}
 The space $\A_2^{(0,1)}$ of para-vectors of $\A_2$ and its one-point compactification $\hat{\A}_2^{(0,1)}$ are
 \begin{eqnarray}
 \A_2^{(0,1)}=\R + \R e_1 + \R e_2, \quad\quad \hat{\A}_2^{(0,1)}=\A_2^{(0,1)}\cup\{\infty\}.
 \end{eqnarray}
 The full Clifford group $\Gamma_2$ of $\A_2$, defined as the multiplicative group consisting of the products of
 non-zero para-vectors of $\A_2$, is identical with $\A_2^{\times}$; that is, %
 \begin{eqnarray}
 \Gamma_2 = \A_2^{\times} = \A_2 \backslash \{ 0 \}. %\cong {\mathbb H}^{\times} = {\mathbb H} \backslash \{ 0 \}.
 \end{eqnarray}
 A Vahlen matrix $A \in {\rm SL}(2, \Gamma_2)$ is a $2 \times 2$ matrix
 $\Big(\,\begin{matrix} a & b \\ c & d \end{matrix}\,\Big)$ such that
 \begin{eqnarray}
 {\rm(i)}\; a, b, c, d \in \A_2; \quad\quad %
 {\rm(ii)}\; ad^*-bc^*=1; \quad\quad %
 {\rm(iii)}\; ab^*, cd^* \in \A_n^{(0,1)}. %
 \end{eqnarray}
% A Vahlen matrix $A =\Big(\,\begin{matrix} a & b \\ c & d \end{matrix}\,\Big) \in {\rm SL}(2, \Gamma_2)$
% defines a M\"{o}bius transformation $T_A$ of $\hat{\A}_2^{(0,1)}$ by
% $$ T_A(z) = (az+b)(cz+d)^{-1}, \quad\quad z \in \hat{\A}_2^{(0,1)}. $$ %

 As a special rule of commutativity in $\A_2$, we notice that
 \begin{eqnarray}
  a e_1e_2 = e_1e_2 a', \quad\quad\quad  a \in \A_2. %
 \end{eqnarray}

 \subsection{The exponential function $\exp$}
 As usual, we define the exponential function $\exp: \A_2 \rightarrow \A_2$ by its Taylor expansion:
 \begin{eqnarray}
 \exp x=\sum_{m=0}^{\infty}\frac{x^m}{m!}, \quad\quad\quad  x\in\A_{2}. %
 \end{eqnarray}
 The convergence is guaranteed as usual since we have $|x^m|=|x|^m$ for $x\in\A_{2}$ and for integers $m \ge 0$.
 Note that in fact $\exp(x) \in \A_2^{\times}$ since we have %
 \begin{eqnarray}
 \exp x\,\exp(-x)=\exp(-x)\exp x=1.
 \end{eqnarray}
 However, in general, we have, for $x,y \in \A_2$, %
 \begin{eqnarray}
 \exp x\,\exp y \neq \exp y\,\exp x \neq \exp(x+y). %
 \end{eqnarray}

 \subsection{Hyperbolic functions $\cosh$ and $\sinh$ of an $\A_{2}$-variable}%
 \label{ss:cosh-sinh} %
% For the purposes of this paper, we find it is natural and convenient to define
 For reasons that will be clear in \S \ref{ss:reason}, we  define the hyperbolic functions
 $\cosh: \A_{2} \rightarrow \A_{2}$ and $\sinh: \A_{2} \rightarrow \A_{2}$, respectively, by
 \begin{eqnarray}
 && \cosh x=\frac{\exp x +\exp(-x^*)}{2}; \\
 && \sinh x=\frac{\exp x -\exp(-x^*)}{2}
 \end{eqnarray}
 (note the reverse or star involution in the expressions). In general,
 \begin{eqnarray}
 \cosh x &\!\!\neq\!\!& \sum_{m=0}^{\infty}\frac{x^{2m}}{(2m)!}, \\ %
 \sinh x &\!\!\neq\!\!& \sum_{m=0}^{\infty}\frac{x^{2m+1}}{(2m+1)!}. %
 \end{eqnarray}
 However, the following two identities hold:
 \begin{eqnarray}
 && \cosh x +\sinh x \,=\,\exp x; \\
 && \cosh x -\sinh x \,=\,\exp(-x^*).
 \end{eqnarray}
 Since $\exp(x^*)=\,(\exp x)^*$, $\cosh(x^*)=\,(\cosh x)^*$ and
 $\sinh(x^*)=(\sinh x)^*$, we may abuse notation and write them as
 \,$\exp x^*$, $\cosh x^*$ and $\sinh x^*$ respectively. It is easy to verify that
 \begin{eqnarray}
 & & \cosh(-x) = \cosh x^*, \\ %
 & & \sinh(-x) = -\sinh x^*, \\ %
 & & \cosh x \,\cosh x^* - \sinh x \,\sinh x^* = 1.
 \end{eqnarray}
% In particular, for $x=\theta u$ with $\theta\in\R$, $u \in \A_2$ such that $u^{(0)}=0$ and $|u|=1$, %
% \begin{eqnarray}
% \exp(\theta u)&\!\!\!=\!\!&\cos\theta+u\sin\theta, \\
% \cosh(\theta u)&\!\!\!=\!\!&\cos\theta+u^{(2)}\sin\theta, \\
% \sinh(\theta u)&\!\!\!=\!\!&u^{(1)}\sin\theta.
% \end{eqnarray}

 \subsection{The polar decomposition of non-real elements in $\A_{2}\backslash\R$}\label{ss:polar} %

 It is easy to verify that every non-real element $a \in \A_{2}\backslash\R$ can be written as %
 \begin{eqnarray}\label{eqn:polar}
 a = |a|(\cos\theta + u\sin\theta), %
 \end{eqnarray}
 where $\theta \in (-\pi,0) \cup (0,\pi)$ and $u \in \A_2^{(1,2)} \cap \A_2^{\rm unit}$ are
 determined by $a$ up to sign. %
 Indeed, $a$ can be written as above in exactly two ways, the other one being %
 \begin{eqnarray}
 a = |a|(\cos(-\theta) + (-u)\sin(-\theta)) %
 \end{eqnarray}
 with the same $\theta$ and $u$ as in (\ref{eqn:polar}).
 It is useful to observe that $(-\theta)(-u)=\theta u$.

 We remark that the $u$ appearing in (\ref{eqn:polar}) is a square root of $-1$ in $\A_2$; indeed, %
 \begin{eqnarray}
 \A_2^{(1,2)} \cap \A_2^{\rm unit} \!\!&=&\!\! \{u \in \A_2 \mid u^{(0)}=0, |u| = 1 \} \nonumber \\ %
                                   \!\!&=&\!\! \{ u \in \A_2 \mid u^2 = -1 \}. %
 \end{eqnarray}

 Notice that we can also write every $a \in \R^{\times}$ as in (\ref{eqn:polar}), %
 by setting $\theta = 0$ or $\pm\pi$ according as $a>0$ or $a<0$
 and $u \in \A_2^{(1,2)} \cap \A_2^{\rm unit}$ arbitrary.

 \begin{definition}
 {\rm For $a \in \A_2\backslash \{0\}$, we define the set of periods of $a$ by %
 $$ {\rm Period}(a) := \{ 2m\pi u \mid m \in \Z \},  $$
 where $u \in \A_2^{(1,2)} \cap \A_2^{\rm unit}$ is the same as in (\ref{eqn:polar}). %
 Recall that the pair $\pm u$ is unique for non-real $a \in \A_{2}\backslash\R$,
 while $u \in \A_2^{(1,2)} \cap \A_2^{\rm unit}$ is arbitrary for $a \in \R^{\times}$.} %
 \end{definition}

 \begin{notation}
 {\rm  For $a,b \in \A_2$, we write \,$a \equiv b$ mod(period)\, if $b-a \in {\rm Period}(a)$. } %
 \end{notation}

 \subsection{A multi-valued logarithmic function of an $\A_2$-variable} %

 We define a multi-valued logarithmic function $\log$ of an $\A_2$-variable as
 the set-valued inverse function of the exponential function $\exp: \A_2 \rightarrow \A_2^{\times}$. %
 Explicitly, we have

 (a) if $a >0$, then
 \begin{eqnarray}\label{eqn:log+}
 \log a = \{ \log_{\R} a + 2m\pi u \mid m \in \Z, \, u \in \A_2^{(1,2)} \cap \A_2^{\rm unit} \}, %
 \end{eqnarray}
 where $\log_{\R}:\R^{+} \rightarrow \R$ is the usual real-valued logarithmic function; %

 (b) if $a <0$, then
 \begin{eqnarray}\label{eqn:log-}
 \log a = \{ \log_{\R}|a| + (2m+1)\pi u \mid m \in \Z, \, u \in \A_2^{(1,2)} \cap \A_2^{\rm unit} \}; %
 \end{eqnarray}

 (c) if $a \in \A_2 \backslash \R$, then
 \begin{eqnarray}\label{eqn:lognonreal}
 \log a = \{ \log_{\R}|a| + (\theta+2m\pi)u \mid m \in \Z \}, %
 \end{eqnarray}
 where $\theta$ and $u$ are the same  as in the polar decomposition (\ref{eqn:polar}) of $a$ %
 (recall that $\theta u$ is well defined, as already observed in \S \ref{ss:polar}). %

 \subsection{A single valued logarithmic function} %

 We may also define a single valued function
 ${\rm Log}:\A_2\backslash\R_{\le 0} \rightarrow \A_2$ as follows: %

 (a) if $a>0$, then
 \begin{eqnarray}
 {\rm Log}\,a = \log_{\R} a \in \R; %
 \end{eqnarray}

 (c) if $a \in \A_2 \backslash \R$, then %
 \begin{eqnarray}\label{eqn:Log}
 {\rm Log}\,a \,=\, \log_{\R}|a| + \theta u \in \A_2, %
 \end{eqnarray}
 where $\theta$ and $u$ are the same as in the polar decomposition (\ref{eqn:polar}) of $a$; %
 in particular, $\theta u$ is well defined. %

 It can be easily verified that the function ${\rm Log}$ is continuous in $\A_2 \backslash \R_{\le 0}$. %

 \subsection{Operations $\oplus$ and $\ominus$ in $\A_2$}
 For $x,y \in \A_2$, we define, as a subset of $\A_2$,
 \begin{eqnarray}
 x \oplus y := \log\,(\exp(x)\exp(y)). %
 \end{eqnarray}
 Note that, in general, $x \oplus y \neq y \oplus x$.
 Observe that $\cosh(x \oplus y)$ and $\sinh(x \oplus y)$ are well-defined elements in $\A_2$ and
 the following identities are easily verified:
 \begin{eqnarray}
 \cosh(x \oplus y)&\!\!\!=\!\!&\cosh x \cosh y + \,\sinh x \sinh y, \label{eq:chx+y} \\ %
 \sinh(x \oplus y)&\!\!\!=\!\!&\sinh x \cosh y + \,\cosh x \sinh y. \label{eq:shx+y} %
 \end{eqnarray}

 For convenience, we also write
 \begin{eqnarray}
 x \ominus y := x \oplus (-y) \,=\, \log\,(\exp(x)\exp(-y)).
 \end{eqnarray}
 Since $\cosh(-x)=\cosh x^*$ and $\sinh(-x)=-\sinh x^*$, we obtain from (\ref{eq:chx+y}) and (\ref{eq:shx+y}):
 \begin{eqnarray}
 \cosh(x \ominus y)&\!\!\!=\!\!&\cosh x \cosh y^* - \sinh x \sinh y^*, \\ %
 \sinh(x \ominus y)&\!\!\!=\!\!&\sinh x \cosh y^* - \cosh x \sinh y^*. %
 \end{eqnarray}

 \subsection{\bf M\"{o}bius transformations of $\hat{\A}_{2}^{(0,1)}$ fixing both $-1$ and $1$}\label{ss:reason}%

 The following proposition explains why we choose to define
 the hyperbolic functions $\cosh$ and $\sinh$ as we did in \S \ref{ss:cosh-sinh}. %
 \begin{proposition}\label{prop:1-1v2}
 A M\"{o}bius transformation of $\hat{\A}_{2}^{(0,1)}$ fixes both $-1$ and $1$
 if and only if its Vahlen matrices $\pm A$ are of the form %
 \begin{eqnarray}
 A = \begin{pmatrix}\cosh x & \sinh x \\ \sinh x & \cosh x \end{pmatrix}, %
 \quad\quad  x \in \A_2. %
 \end{eqnarray}
 \end{proposition}

 \begin{proof}
 By Proposition \ref{prop:1-1} (with $n=2$), a M\"{o}bius transformation of $\hat{\A}_{2}^{(0,1)}$
 fixes both $-1$ and $1$ if and only if its Vahlen matrices $\pm A$ are of the form %
 \begin{eqnarray*}
 A = \left(\begin{matrix}\,a & b\, \\ \,b & a\, \end{matrix}\right) %
 \end{eqnarray*}
 where $a,b \in \A_2$ satisfy $ab^*\in\A_2^{(0,1)}$ and $aa^*-bb^*=1$.
 Since $ab^*=(ab^*)^*=ba^*$, we have
 $$ (a+b)(a-b)^*=aa^*-ab^*+ba^*-bb^*=1. $$ %
 Now let $x \in \log(a+b)$. Then $a+b=\exp( x)$.
 Therefore \,$(a-b)^*=\exp(- x)$ and \,$a-b=(\exp(- x))^*=\exp(- x^*)$. %
 It follows that $a=\cosh x$ and $b=\sinh x$. %
 \end{proof}

 \subsection{\bf M\"{o}bius transformations of $\hat{\A}_{2}^{(0,1)}$ fixing both $0$ and $\infty$}

 The following two propositions will be used in \S \ref{s:halfdist}. They can be proved by easy calculations. %
 \begin{proposition}\label{prop:e2fixed}
 If a M\"{o}bius transformation of $\hat{\A}_{2}^{(0,1)}$ fixes each of $-e_2, e_2, 0$ and $\infty$
 then its Vahlen matrices are of the form
 $\pm \Big(\small\begin{matrix}\,a & 0\phantom{bb}\! \\ \,0 & a^{-1}\! \end{matrix}\normalsize\Big)$
 with $a \in \R + \R e_1$, $|a|=1$. %
 \end{proposition}
%%
% \begin{proposition}\label{prop:4fixed}
% If a M\"{o}bius transformation of $\hat{\A}_{2}^{(0,1)}$ fixes all of $-e_1, e_1, 0$ and $\infty$
% then its Vahlen matrices are of the form
% $\pm \Big(\small\begin{matrix}\,a & 0\phantom{bb}\! \\ \,0 & a^{-1}\! \end{matrix}\normalsize\Big)$
% with $0 \neq a \in \R + \R e_2$. %
% \end{proposition}
%
 \begin{proposition}\label{prop:allfixed}
 If a M\"{o}bius transformation of $\hat{\A}_{2}^{(0,1)}$ fixes each of $-1, 1, -e_1, e_1, 0$ and $\infty$
 then it is the identity transformation, that is, its Vahlen matrices are $\pm I$. %
 \end{proposition}

%%%%%%%%%%%%%%%%%%%%%%%%%%%%%%%%%%%%%%%%%%%%%%%%%%%%%%%%%%%%%%%%%%%%%%%%%%%%%%%%%%%%%%%%%%%%%%%%%%%%%%%%%%%%%%%%%%%%%

 \section{\bf Geometry of ${\rm SO}(3)$ via the Clifford algebra $\A_{2}$}\label{s:SO3} %

 \nn In this section we study the geometry of ${\rm SO}(3)$ via the Clifford algebra $\A_{2}$ or quaternions.
 As explained earlier, the unit elements in $\A_{2}$ form the spinor group ${\rm Spin}(3)$,
 the double covering of ${\rm SO}(3)$.
 More precisely, every unit element $a$ in $\A_{2}$ acts as a rotation $\rho_a$ on the Euclidean $3$-space $\A_2^{(0,1)}$
 fixing the origin by $\rho_a(x) = ax(a')^{-1}$, and it is clear that $\rho_a = \rho_{-a}$. %
 Conversely, every rotation $\rho$ of $\A_2^{(0,1)}$ fixing the origin corresponds to exactly one pair of unit elements
 $\pm a$ in $\A_{2}$ such that $\rho = \rho_a = \rho_{-a}$. %
 We shall first determine the axis and the rotation angle of $\rho_a$ in terms of $a$, and then study
 the Euler decomposition of $a$ and determine the Euler angles associated to $\rho_a$.
 This will be used in \S \ref{s:halfdist} to define the quaternion half lengths
 between two oriented line-plane flags in $\H^4$ with a common perpendicular line. %

 \subsection{\bf Axis and rotation angle of an element in ${\rm SO}(3)$}
 It is a well-known fact that every non-identity element $\phi\in {\rm SO}(3)$ has $1$ as a simple eigenvalue
 and hence fixes pointwise a unique straight line in $\A_2^{(0,1)}$, called its axis.
 Thus $\phi$ acts on the Euclidean space $\A_2^{(0,1)}$ as a rotation about its axis.

 Let $v \in \A_2^{(0,1)}$ be a unit eigenvector of $\phi$ with eigenvalue $1$
 and let $v^{\perp}$ denote the orthogonal complement plane of $v$ in $\A_2^{(0,1)}$. %
 We orient the Euclidean space $\A_2^{(0,1)}$ by the ordered basis $(1, e_1, e_2)$.
 Then $v\in\A_2^{(0,1)}$ naturally induces an orientation of the plane $v^{\perp}$
 and $\phi |_{v^{\perp}}: v^{\perp} \rightarrow v^{\perp}$ is a rotation of a certain angle
 $2\theta \in (0,2\pi)$ about the origin. %
 We then say that $\phi$ is the rotation of angle $2\theta$ about $v$
 (precisely, about the oriented axis $\R v$ with the orientation given by $v$).

 Observe that the rotation angle of $\phi$ about $-v$ is $2\pi-2\theta$.

 \subsection{\bf The unit elements in $\A_{2}$ as rotations of $\A_2^{(0,1)}$}

 Recall that $\Gamma_2=\A_2\backslash\{0\}$ and ${\rm SO}(3)\cong\Gamma_2/\R^{\times}\cong\A_2^{\rm unit}/\{\pm 1\}$, %
 where a pair $\pm a \in \A_2^{\rm unit}$ correspond to
 $\rho_{a}=\rho_{-a}\in{\rm SO}(3): \A_2^{(0,1)} \rightarrow \A_2^{(0,1)}$ defined by $\rho_{a}(x)=ax(a')^{-1}=axa^*$. %

 As a special case of the polar decomposition (\ref{eqn:polar}), we obtain %
 \begin{proposition}\label{prop:a}
 An element $a \in \A_{2}^{\rm unit} \backslash \{ \pm 1 \}$ can be written uniquely in the form
 \begin{eqnarray}\label{eqn:au}
   a \,=\, \cos\theta + ve_1e_2\sin\theta \,=\, \exp(\theta ve_1e_2), %
 \end{eqnarray}
 where $\theta \in (0,\pi)$ and $v\in\A_2^{(0,1)}\cap\A_2^{\rm unit}$ (that is, $v$ is a unit para-vector). \qed %
 %that is, $v\in\A_2^{(0,1)}$ with $|v|=1$. \qed %
 \end{proposition}

 Observe that, with the same $\theta$ and $v$ as in (\ref{eqn:au}), we have %
 \begin{eqnarray}\label{eqn:-a}
  -a = \cos(\pi-\theta) + (-v) e_1e_2\sin(\pi-\theta). %
 \end{eqnarray}

 \begin{example} {\rm Let $\theta \in (0,\pi)$. Then we have %
 \begin{itemize}
 \item[(a)] $v=e_2$ if $a=\exp(\theta e_1)$, since $e_1=e_2e_1e_2$; %
 \item[(b)] $v=-e_1$ if $a=\exp(\theta e_2)$, since $e_2=(-e_1)e_1e_2$; %
 \item[(c)] $v=1$ if $a=\exp(\theta e_1e_2)$.
 \end{itemize}}
 \end{example}

 The axis and rotation angle of $\rho_a$ are related to $\theta$ and $v$ in (\ref{eqn:au}) as follows. %

 \begin{proposition}\label{prop:axis}
 If $a \in \A_{2}^{\rm unit} \backslash \{ \pm 1 \}$ is written as in $(\ref{eqn:au})$, that is, %
 $$ a \,=\, \cos\theta + ve_1e_2\sin\theta \,=\, \exp(\theta ve_1e_2) $$
 with $\theta\in(0,\pi)$ and $v\in\A_{2}^{(0,1)}\cap\A_{2}^{\rm unit}$,
 then $v$ is an eigenvector of $\rho_{a} \in {\rm SO}(3)$ with eigenvalue $1$
 and the rotation angle of $\rho_{a}$ about $v$ is $2\theta$. %
 \end{proposition}

 \begin{proof}
% To show that $v$ is an eigenvector of $\rho_{a}$,
 Noticing that $vv'=|v|^2=1$, we have
 \begin{eqnarray*}
 \rho_{a}(v)
 \!\!\!&=&\!\!\! ava^*=\,(\cos\theta+ve_1e_2\sin\theta)v(\cos\theta+e_2e_1v\sin\theta) \\ %
 \!\!\!&=&\!\!\! (\cos^2\theta)v+(\sin^2\theta)ve_1e_2ve_2e_1v+(\cos\theta\sin\theta)v(e_2e_1+e_1e_2)v \\ %
 \!\!\!&=&\!\!\! (\cos^2\theta)v+(\sin^2\theta)vv'e_1e_2e_2e_1v = v. %
 \end{eqnarray*}
 Thus $v$ is an eigenvector of $\rho_{a}$ with eigenvalue $1$.

 To determine the rotation angle of $\rho_{a}$ about $v$, let us first
 consider the special case where $v=e_2$. In this case, $a=a^*=\cos\theta+e_1\sin\theta$ and
 $\rho_{a}$ restricted to the plane $(e_2)^{\perp}=\R+\R e_1$ oriented by $e_2$ is a rotation
 about the origin of angle $2\theta \in (0,2\pi)$ since we have
 $$ \rho_{a}(1)=a1a^*=a^2=\cos(2\theta)+e_1\sin(2\theta). $$ %
 In the general case, we choose $c\in\A_2\backslash\{0\}$ so that $\rho_c(v)=cv(c')^{-1}=e_2$. %
 Then
 $$ cac^{-1}=\cos\theta+(\sin\theta)cve_1e_2c^{-1}=\cos\theta+(\sin\theta)cv(c')^{-1}e_1e_2=\cos\theta+e_1\sin\theta. $$ %
 By the special case just considered, $\rho_{cac^{-1}}=\rho_c\rho_a(\rho_c)^{-1}$ is the rotation of
 angle $2\theta$ about $e_2$. Hence $\rho_a$ is the rotation of angle $2\theta$ about $v$.
 \end{proof}

 \begin{remark} {\rm It follows from (\ref{eqn:-a}) and Proposition \ref{prop:axis} that the rotation angle of
 $\rho_{-a}=\rho_{a}$ about $-v$ is $2(\pi-\theta)$.} %
 \end{remark}

 \subsection{\bf Euler decomposition of units in $\A_2$}

 Given $a\in \A_2^{\rm unit}$, suppose %
 $$ a \,=\, a_0+a_1e_1+a_2e_2+a_{12}e_1e_2, \quad\quad a_0,a_1,a_2,a_{12} \in \R. $$ %
 We pursue the following Euler decomposition:
 \begin{eqnarray}\label{eqn:a=e}
  a = \exp(\alpha e_1) \exp(\beta e_1e_2) \exp(\gamma e_1), %
 \end{eqnarray}
 with $\alpha, \beta, \gamma \in \R/2\pi\Z$.
 (A geometric interpretation of this decomposition will be given in the next subsection.)
 Expanding the right side of (\ref{eqn:a=e}), we obtain
 \begin{eqnarray}\label{eqn:a=ee}
 a &\!\!\!=\!\!& \cos\beta\cos(\gamma+\alpha) +  e_1  \cos\beta\sin(\gamma+\alpha)
                                              +  e_2  \sin\beta\sin(\gamma-\alpha) \nonumber \\ %
   && \hspace{151pt}+\;  e_1e_2  \sin\beta\cos(\gamma-\alpha). %
 \end{eqnarray}
 Thus the decomposition (\ref{eqn:a=e}) is equivalent to the following system of equations:
 \begin{eqnarray}
    a_0 &\!\!\!=\!\!& \cos\beta\cos (\gamma+\alpha), \label{eqn:a_0=1} \\
    a_1 &\!\!\!=\!\!& \cos\beta\sin (\gamma+\alpha), \label{eqn:a_0=2} \\
    a_2 &\!\!\!=\!\!& \sin\beta\cos (\gamma-\alpha), \label{eqn:a_0=3} \\
 a_{12} &\!\!\!=\!\!& \sin\beta\sin (\gamma-\alpha). \label{eqn:a_0=4} %
 \end{eqnarray}
 It follows that
 $\cos^2\beta = \,a_0^2 + \,a_1^2$ and $\sin^2\beta = \,a_2^2 + \,a_{12}^2$.
 If $\sin (2\beta) \neq 0$, then a solution of $\beta\in\R/2\pi\Z$ determines
 $\gamma+\alpha, \gamma-\alpha \in \R/2\pi\Z$ %
 and hence we obtain two solutions, $(\alpha, \beta, \gamma)$ and $(\alpha+\pi, \beta, \gamma+\pi)$, %
 to the above system of equations in $(\R/2\pi\Z)^3$. %
 Since there are exactly four solutions of $\beta\in\R/2\pi\Z$,
 we obtain in total eight solutions to the system of equations in $(\R/2\pi\Z)^3$ as follows: %
 \begin{eqnarray*}
 && (\alpha, \beta, \gamma), \ (\alpha+\pi, \beta, \gamma+\pi); \\ %
 && (\alpha+\pi, \beta+\pi, \gamma), \ (\alpha, \beta+\pi, \gamma+\pi); \\ %
 && (\alpha+\textstyle\frac12\pi, -\beta, \gamma-\textstyle\frac12\pi), \
    (\alpha-\textstyle\frac12\pi, -\beta, \gamma+\textstyle\frac12\pi); \\ %
 && (\alpha-\textstyle\frac12\pi, -\beta+\pi, \gamma-\textstyle\frac12\pi), \
    (\alpha+\textstyle\frac12\pi, -\beta+\pi, \gamma+\textstyle\frac12\pi). %
 \end{eqnarray*}
 It is easy to verify that the condition $\sin 2\beta =0$ is equivalent to that %
 $$ a \in (\R + \R e_1)\cup(\R e_2 + \R e_1e_2)=\A_1 \cup (\A_1e_2). $$ %
 In this case, $\sin\beta=0$ or $\cos\beta=0$
 and only $\alpha+\gamma \in \R/2\pi\Z$ or only $\alpha-\gamma \in \R/2\pi\Z$ is determined. %
 As a result, there are infinitely many solutions in $(\R/2\pi\Z)^3$.

 \begin{definition}
 {\rm
 We call a triple $(\alpha,\beta,\gamma)\in(\R/2\pi\Z)^3$ or $(\alpha,\beta,\gamma)\in(\R/\pi\Z)^3$ {\it regular} if
 and only if $\sin 2\beta\neq 0$. For convenience, we introduce the following notations: %
 \begin{eqnarray}
 (\R/2\pi\Z)_{\rm reg}^3 \!\!\!&:=&\!\!\! \{(\alpha,\beta,\gamma)\in(\R/2\pi\Z)^3\mid\sin 2\beta\neq 0\}; \\ %
  (\R/\pi\Z)_{\rm reg}^3 \!\!\!&:=&\!\!\! \{(\alpha,\beta,\gamma)\in(\R/\pi\Z)^3 \mid\sin 2\beta\neq 0\}.    %
 \end{eqnarray}
 }
 \end{definition}

 Then we have proved the following propositions. %
 \begin{proposition}
 The map $\Phi: (\R / 2\pi\Z)_{\rm reg}^3 \rightarrow \A_2^{\rm unit}\backslash(\A_1e_2\cup\A_1)$ defined by %
 $$ (\alpha, \beta, \gamma) \longmapsto a = \exp(\alpha e_1) \exp(\beta e_1e_2) \exp(\gamma e_1) $$ %
 is an eight-fold covering map. %
 \end{proposition}

% Explicitly, we have
% \begin{eqnarray}
% && \cos 2\beta = a_0^2+a_1^2-a_2^2-a_{12}^2, \\ %
% && \sin 2\beta = \pm \,2\sqrt{(a_0^2+a_1^2)(a_2^2+a_{12}^2)}; %
% \end{eqnarray}
% \vskip -18pt
% \begin{eqnarray}
% && \cos 2\gamma = \pm \frac{a_0a_2-a_1a_{12}}{\sqrt{(a_0^2+a_1^2)(a_2^2+a_{12}^2)}}, \\ %
% && \sin 2\gamma = \pm \frac{a_1a_2+a_0a_{12}}{\sqrt{(a_0^2+a_1^2)(a_2^2+a_{12}^2)}}; %
% \end{eqnarray}
% \vskip -12pt
% \begin{eqnarray}
% && \cos 2\alpha = \pm \frac{a_0a_2+a_1a_{12}}{\sqrt{(a_0^2+a_1^2)(a_2^2+a_{12}^2)}}, \\ %
% && \sin 2\alpha = \pm \frac{a_1a_2-a_0a_{12}}{\sqrt{(a_0^2+a_1^2)(a_2^2+a_{12}^2)}}. %
% \end{eqnarray}
%
% \nn {\bf Remark.} Note that there hold the following Euler's identities:
% \begin{eqnarray}
% (a_0^2+a_1^2)(a_2^2+a_{12}^2) %
% &\!\!=\!\!&(a_0a_2 + a_1a_{12})^2+(a_1a_2 - a_0a_{12})^2 \\ %
% &\!\!=\!\!&(a_0a_2 - a_1a_{12})^2+(a_1a_2 + a_0a_{12})^2. %
% \end{eqnarray}

 \begin{proposition}
 The map $\Psi: (\R/\pi\Z)_{\rm reg}^3 \rightarrow (\A_2^{\rm unit}\backslash(\A_1e_2\cup\A_1))/\{\pm 1\}$ %
 defined by
 $$ (\alpha, \beta, \gamma) \longmapsto \pm a \,=\, \exp(\alpha e_1) \exp(\beta e_1e_2) \exp(\gamma e_1) $$ %
 is a two-fold covering map. Explicitly, we have
 $$ \Psi^{-1}(\Psi(\alpha, \beta, \gamma))
  = \{ (\alpha, \beta, \gamma),\,(\alpha+\tst\frac12\pi, -\beta, \gamma+\tst\frac12\pi) \}.$$  % \qed %
 \end{proposition}

 \subsection{\bf Geometric interpretation of the Euler decomposition (\ref{eqn:a=e})}\label{ss:geoeuler} %

 Given an element $a = \exp(\alpha e_1)\exp(\beta e_1e_2)\exp(\gamma e_1) \in \A_2^{\rm unit}$
 with $\alpha, \beta, \gamma \in \R/2\pi\Z$, the orthogonal transformation
 $\rho_a: \A_2^{(0,1)} \rightarrow \A_2^{(0,1)}$ defined by $\rho_a(x)=ax(a')^{-1}$ can be %
 decomposed as follows:
 \begin{eqnarray}\label{eqn:rhoa}
 \rho_a=\rho_{\exp(\alpha e_1)}\rho_{\exp(\beta e_1e_2)}\rho_{\exp(\gamma e_1)}=\eta_3\eta_2\eta_1, %
 \end{eqnarray}
 where $\eta_1=\rho_{\exp(\alpha e_1)}$, %
 $\eta_2=\eta_1\rho_{\exp(\beta e_1e_2)}\eta_1^{-1}$ and %
 $\eta_3=(\eta_2\eta_1)\rho_{\exp(\gamma e_1)}(\eta_2\eta_1)^{-1}$.  %

 By Proposition \ref{prop:axis}, the orthogonal transformation $\eta_1=\rho_{\exp(\alpha e_1)}$
 is the rotation of angle $2\alpha$ about $e_2$, %
 the orthogonal transformation $\eta_2=\eta_1\rho_{\exp(\beta e_1e_2)}\eta_1^{-1}$ is the rotation of angle $2\beta$
 about $\eta_1(1)=\exp(2\alpha e_1)$, %
 and the orthogonal transformation $\eta_3=(\eta_2\eta_1)\rho_{\exp(\gamma e_1)}(\eta_2\eta_1)^{-1}$
 is the rotation of angle $2\gamma$ about $\eta_2\eta_1(e_2)=\eta_2(e_2)$. %

 The transformation $\rho_a=\eta_3\eta_2\eta_1\in {\rm SO}(3)$ can be better understood
 if we consider its action on the unit sphere ${\mathbb S}^2$ in the Euclidean space $\A_2^{(0,1)}$:
 $$ {\mathbb S}^2=\A_2^{(0,1)}\cap\A_2^{\rm unit}=\{x_0+x_1e_1+x_2e_2\mid x_0, x_1, x_2 \in \R,\,x_0^2+x_1^2+x_2^2=1\}. $$ %
 Note that ${\rm SO}(3)$ acts freely and transitively on the space of ordered orthogonal frames of $\A_2^{(0,1)}$
% (which is ${\rm SO}(3)$)
 and hence on the unit tangent bundle of ${\mathbb S}^2$:
 $$ \mathsf{T}^1{\mathbb S}^2 = \{ (x,u)\in {\mathbb S}^2 \times {\mathbb S}^2 \mid \langle x,u \rangle = 0 \}. $$ %
 Thus an element $\phi \in {\rm SO}(3)$ is determined by the image of a chosen unit tangent vector,
 say $(1, e_1)$, of ${\mathbb S}^2$ under $\phi$.
 Alternatively, we may think of a unit tangent vector $(x,u)$ of ${\mathbb S}^2$
 as a pointed, oriented great circle $C_{x,u}$ in ${\mathbb S}^2$ since the unit vector $u\in {\mathbb S}^2$ is tangent
 at $x\in{\mathbb S}^2$ to a unique great circle $C_{x,u}$
 passing through $x$ and gives $C_{x,u}$ an orientation.

 Given $a = \exp(\alpha e_1)\exp(\beta e_1e_2)\exp(\gamma e_1) \in \A_2^{\rm unit}$,
 recall that $\rho_a=\eta_3\eta_2\eta_1$ as in (\ref{eqn:rhoa}).
 Now we explain how to obtain geometrically the unit tangent vector
 $$ \rho_a(1, e_1)=(\rho_a(1), \rho_a(e_1))\in \mathsf{T}^1{\mathbb S}^2. $$ %
 For this, we start from $(1, e_1)\in \mathsf{T}^1{\mathbb S}^2$, traverse along the oriented great circle $C_{1, e_1}$ %
 in its orientation by distance $2\alpha \in [0,2\pi)$ and arrive at
 $$ \eta_1(1, e_1)=(\eta_1(1), \eta_1(e_1)) \in\mathsf{T}^1{\mathbb S}^2. $$
 %=(\exp(2\alpha e_1), \exp[(2\alpha+\pi)e_1])\in\mathsf{T}^1{\mathbb S}^2. $$ %
 Then we perform the rotation about $\eta_1(e_1)$ by angle $2\beta\in[0,2\pi)$ and arrive at
 $$ \eta_2\eta_1(1, e_1)=(\eta_2\eta_1(1), \eta_2\eta_1(e_1))\in\mathsf{T}^1{\mathbb S}^2. $$ %
 Finally, we traverse along the the oriented great circle $C_{\eta_2\eta_1(1), \eta_2\eta_1(e_1)}$
 in its orientation by distance $2\gamma \in [0,2\pi)$ and arrive at
 $(\rho_a(1), \rho_a(e_1))\in \mathsf{T}^1{\mathbb S}^2$.
% See Fig. 1 for an illustration of the procedure described above. %

%%%%%%%%%%%%%%%%%%%%%%%%%%%%%%%%%%%%%%%%%%%%%%%%%%%%%%%%%%%%%%%%%%%%%%%%%%%%%%%%%%%%%%%%%%%%%%%%%%%%%%%%%%%%%%%%%%%%
\begin{figure}[h]
\begin{pspicture}(0,-3)(0,3.5)
\psset{xunit=8pt} \psset{yunit=8pt} \psset{runit=8pt}
\psarc[linewidth=0.3pt,linecolor=black](0,0){10}{0}{209}
\psarc[linewidth=0.3pt,linecolor=black](0,0){10}{211}{360}
\psellipse*[linewidth=0pt,fillcolor=lightgray,linecolor=lightgray](0,0)(10,4)
%
%Upper inclined plane \gamma_3
\pscustom[linewidth=0.7pt,linecolor=blue,fillstyle=solid,fillcolor=pink]{\pscurve(7.541335276,
-2.626884464)(7.377490208, -2.153503048)(7.184529551,
-1.671622721)(6.963214841, -1.183145271)(6.714419500,
-.689998490)(6.439125416, -.194128596)(6.138419044,
.302507432)(5.813487138, .797949602)(5.465612054,
1.290242630)(5.096166690, 1.777443670)(4.706609097,
2.257629933)(4.298476660, 2.728906370)(3.873380100,
3.189413060)(3.432997088, 3.637332588)(2.979065575,
4.070897255)(2.513377080, 4.488395920)(2.037769425,
4.888180945)(1.554119626, 5.268674554)(1.064336426,
5.628375110)(.570352778, 5.965863042)(.074118208,
6.279806436)(-.422408875, 6.568966305)(-.917268891,
6.832201461)(-1.408508894, 7.068473050)(-1.894190130,
7.276848594)(-2.372395870, 7.456505746)(-2.841238854,
7.606735476)(-3.298868762, 7.726944898)(-3.743479535,
7.816659595)(-4.173316510, 7.875525510)(-4.586683318,
7.903310322)(-4.981948569, 7.899904381)(-5.357552365,
7.865321125)(-5.712012330, 7.799697040)(-6.043929597,
7.703291113)(-6.351994240, 7.576483820)(-6.634990458,
7.419775602)(-6.891801401, 7.233784919)(-7.121413554,
7.019245796)(-7.322920744, 6.777004916)(-7.495527704,
6.508018296)(-7.638553251, 6.213347489)(-7.751432912,
5.894155448)(-7.833721208, 5.551701872)(-7.885093386,
5.187338254)(-7.905346703, 4.802502587)(-7.894401226,
4.398713634)(-7.852300157, 3.977564973)(-7.779209644,
3.540718676)(-7.675418149, 3.089898791)(-7.541335277,
2.626884463)\psline(-7.541335284,
2.626884455)(7.541335284,-2.626884455)} %
%Lower inclined plane \gamma_3
\pscustom[linewidth=0.7pt,linecolor=blue,fillstyle=solid,fillcolor=pink]{\pscurve(7.541335343,
-2.626884428)(7.382979280, -2.697898944)(7.220444776,
-2.767386557)(7.053823792, -2.835307953)(6.883210648,
-2.901624681)(6.708701906, -2.966299209)(6.530396315,
-3.029294940)(6.348394804, -3.090576216)(6.162800353,
-3.150108362)(5.973718027, -3.207857676)(5.781254828,
-3.263791479)(5.585519683, -3.317878114)(5.386623363,
-3.370086972)(5.184678457, -3.420388500)(4.979799244,
-3.468754232)(4.772101666, -3.515156798)(4.561703280,
-3.559569933)(4.348723162, -3.601968503)(4.133281852,
-3.642328511)(3.915501269, -3.680627116)(3.695504697,
-3.716842639)(3.473416607, -3.750954590)(3.249362719,
-3.782943656)(3.023469838, -3.812791736)(2.795865802,
-3.840481937)(2.566679414, -3.865998588)(2.336040385,
-3.889327249)(2.104079275, -3.910454713)(1.870927325,
-3.929369027)(1.636716520, -3.946059483)(1.401579392,
-3.960516637)(1.165649049, -3.972732305)(.9290589983,
-3.982699574)(.6919431295, -3.990412805)(.4544356501,
-3.995867630)(.2166709794, -3.999060964)(-0.2121631793,
-3.999990997)(-.2590916177, -3.998657205)(-.4968202626,
-3.995060342)(-.7342677482, -3.989202444)(-.9712996593,
-3.981086826)(-1.207781846, -3.970718082)(-1.443580479,
-3.958102078)(-1.678562127, -3.943245956)(-1.912593769,
-3.926158124)(-2.145542964, -3.906848252)(-2.377277873,
-3.885327269)(-2.607667343, -3.861607355)(-2.836580983,
-3.835701935)(-3.063889238, -3.807625669)(-3.289463453,
-3.777394450) \pscurve(-3.289463472,-3.777394445)(-3.101958074,
-3.956198090)(-2.753369954, -4.276563902)(-2.398577314,
-4.587292790)(-2.038379656, -4.887684548)(-1.673588661,
-5.177062269)(-1.305026364, -5.454773852)(-.933523268,
-5.720193512)(-.559916560, -5.972723124)(-.185048116,
-6.211793648)(.190237324, -6.436866348)(.565094073,
-6.647434043)(.938677435, -6.843022235)(1.310145545,
-7.023190175)(1.678661342, -7.187531866)(2.043394398,
-7.335676982)(2.403522813, -7.467291689)(2.758235062,
-7.582079398)(3.106731834, -7.679781446)(3.448227816,
-7.760177674)(3.781953457, -7.823086903)(4.107156751,
-7.868367379)(4.423104878, -7.895917072)(4.729085852,
-7.905673888)(5.024410172, -7.897615848)(5.308412359,
-7.871761111)(5.580452426, -7.828167934)(5.839917348,
-7.766934552)(6.086222447, -7.688198953)(6.318812690,
-7.592138560)(6.537163950, -7.478969840)(6.740784191,
-7.348947809)(6.929214568, -7.202365462)(7.102030468,
-7.039553112)(7.258842464, -6.860877646)(7.399297187,
-6.666741693)(7.523078132, -6.457582728)(7.629906376,
-6.233872064)(7.719541186, -5.996113834)(7.791780572,
-5.744843798)(7.846461750, -5.480628170)(7.883461501,
-5.204062359)(7.902696450, -4.915769570)(7.904123253,
-4.616399457)(7.887738689, -4.306626621)(7.853579687,
-3.987149123)(7.801723212, -3.658686868)(7.732286132,
-3.321980038)(7.645424909, -2.977787371)(7.541335276, -2.626884464)}
%
%axes
\psline[linewidth=0.6pt,linecolor=black](0,0)(0,6.1)
\psline[linewidth=0.6pt,linecolor=black]{->}(0,6.5)(0,13)
\psline[linewidth=0.6pt,linecolor=black]{->}(-5.692795234,
-3.288576164)(-11.258,-6.5)
\psline[linewidth=0.6pt,linecolor=black](0,0)(6.1,0)
\psline[linewidth=0.6pt,linecolor=black]{->}(6.5,0)(13,0)
\psline[linewidth=0.4pt,linecolor=black](0,0)(-5.692795234,
-3.288576164)
\psline[linewidth=0.4pt,linecolor=black,linestyle=dashed](-7.541335284,
2.626884455)(0,0)
\psline[linewidth=0.4pt,linecolor=black](0,0)(7.541335284,
-2.626884455)
\psline[linewidth=0.4pt,linecolor=black](0,0)(2.385846903,
4.598164327)
%
%Angles
\psarc[linewidth=0.2pt,linecolor=black]{->}(0,0){0.65}{210}{341}
\psarc[linewidth=0.2pt,linecolor=black]{->}(0,0){1}{-19.20483549}{62.57658132}
\psarc[linewidth=0.2pt,linecolor=black]{->}(7.541335284,
-2.626884455){1}{24.67086101}{107.6492263}
%
%Labels
\uput{7pt}[270](0,0){\mbox{$2\alpha$}}
\uput{10pt}[70](7.541335284, -2.626884455){\mbox{$2\beta$}} %
\uput{10pt}[20](0,0){\mbox{$2\gamma$}}
%\uput{0pt}[0](-9,-1){\mbox{\scriptsize $F_1$}}
%\uput{0pt}[0](-5,7){\mbox{\scriptsize $F_3$}}
%\uput{0pt}[200](-11.258,-6.5){\mbox{\scriptsize $\vec{l}_1$}}
%\uput{0pt}[70](3.385846903,6.525431464){\mbox{\scriptsize
%$\vec{l}_3$}}
\uput{3pt}[95](-5.692795234,-3.288576164){\mbox{$1$}}
\uput{3pt}[50](10,0){\mbox{$e_1$}}
\uput{3pt}[50](0,10){\mbox{$e_2$}}
%
%Elliptical Curves
\pscurve[linewidth=0.7pt,linecolor=blue](-10., 0.)(-9.980267284,
-.2511620782)(-9.921147013, -.5013329344)(-9.822872507,
-.7495252584)(-9.685831611, -.9947595488)(-9.510565163,
-1.236067978)(-9.297764859, -1.472498211)(-9.048270524,
-1.703117167)(-8.763066800, -1.927014696)(-8.443279255,
-2.143307180)(-8.090169943, -2.351141010)(-7.705132427,
-2.549695959)(-7.289686273, -2.738188424)(-6.845471059,
-2.915874510)(-6.374239897, -3.082052971)(-5.877852522,
-3.236067978)(-5.358267951, -3.377311702)(-4.817536743,
-3.505226720)(-4.257792918, -3.619308209)(-3.681245521,
-3.719105944)(-3.090169938, -3.804226066)(-2.486898867,
-3.874332645)(-1.873813142, -3.929149003)(-1.253332332,
-3.968458806)(-.6279051925, -3.992106914)(0., -4.)(.6279051925,
-3.992106914)(1.253332332, -3.968458806)(1.873813142,
-3.929149003)(2.486898867, -3.874332645)(3.090169938,
-3.804226066)(3.681245521, -3.719105944)(4.257792918,
-3.619308209)(4.817536743, -3.505226720)(5.358267951,
-3.377311702)(5.877852522, -3.236067978)(6.374239897,
-3.082052971)(6.845471059, -2.915874510)(7.289686273,
-2.738188424)(7.705132427, -2.549695959)(8.090169943,
-2.351141010)(8.443279255, -2.143307180)(8.763066800,
-1.927014696)(9.048270524, -1.703117167)(9.297764859,
-1.472498211)(9.510565163, -1.236067978)(9.685831611,
-.9947595488)(9.822872507, -.7495252584)(9.921147013,
-.5013329344)(9.980267284, -.2511620782)(10., 0.)
\pscurve[linewidth=0.7pt,linecolor=blue](10., 0.)(9.996946698,
.09883863624)(9.987788655, .1976169157)(9.972531465,
.2962745182)(9.951184443, .3947511977)(9.923760627,
.4929868180)(9.890276762, .5909213908)(9.850753296,
.6884951116)(9.805214365, .7856483952)(9.753687776,
.8823219148)(9.696204996, .9784566352)(9.632801127,
1.073993850)(9.563514887, 1.168875220)(9.488388586,
1.263042804)(9.407468101, 1.356439098)(9.320802848,
1.449007068)(9.228445749, 1.540690186)(9.130453201,
1.631432467)(9.026885049, 1.721178495)(8.917804533,
1.809873468)(8.803278266, 1.897463223)(8.683376187,
1.983894271)(8.558171512, 2.069113834)(8.427740699,
2.153069872)(8.292163400, 2.235711114)(8.151522403,
2.316987096)(8.005903592, 2.396848186)(7.855395894,
2.475245615)(7.700091215, 2.552131510)(7.540084399,
2.627458917)(7.375473148, 2.701181840)(7.206357986,
2.773255259)(7.032842186, 2.843635160)(6.855031708,
2.912278566)(6.673035132, 2.979143558)(6.486963603,
3.044189303)(6.296930737, 3.107376084)(6.103052583,
3.168665313)(5.905447537, 3.228019563)(5.704236268,
3.285402590)(5.499541646, 3.340779350)(5.291488682,
3.394116027)(5.080204399, 3.445380055)(4.865817845,
3.494540123)(4.648459928, 3.541566214)(4.428263372,
3.586429612)(4.205362659, 3.629102918)(3.979893897,
3.669560073)(3.751994771, 3.707776374)(3.521804442,
3.743728483)(3.289463494, 3.777394444)
\pscurve[linewidth=0.7pt,linecolor=blue](7.071067810,
-7.071067810)(7.279112686, -6.835116664)(7.458430227,
-6.572190433)(7.608312751, -6.283326769)(7.728168738,
-5.969665682)(7.817525181, -5.632445059)(7.876029417,
-5.272995743)(7.903450567, -4.892736323)(7.899680404,
-4.493167506)(7.864733816, -4.075866214)(7.798748714,
-3.642479336)(7.701985518, -3.194717262)(7.574826098,
-2.734347092)(7.417772303, -2.263185707)(7.231443950,
-1.783092560)(7.016576388, -1.295962364)(6.774017610,
-.803717594)(6.504724876, -.308300920)(6.209760966,
.188332478)(5.890289960, .684222620)(5.547572678,
1.177412442)(5.182961660, 1.665955556)(4.797895859,
2.147923901)(4.393894956, 2.621415374)(3.972553357,
3.084561319)(3.535533905, 3.535533905)(3.084561319,
3.972553357)(2.621415374, 4.393894956)(2.147923901,
4.797895859)(1.665955556, 5.182961660)(1.177412442,
5.547572678)(.684222620, 5.890289960)(.188332478,
6.209760966)(-.308300920, 6.504724876)(-.803717594,
6.774017610)(-1.295962364, 7.016576388)(-1.783092560,
7.231443950)(-2.263185707, 7.417772303)(-2.734347092,
7.574826098)(-3.194717262, 7.701985518)(-3.642479336,
7.798748714)(-4.075866214, 7.864733816)(-4.493167506,
7.899680404)(-4.892736323, 7.903450567)(-5.272995743,
7.876029417)(-5.632445059, 7.817525181)(-5.969665682,
7.728168738)(-6.283326769, 7.608312751)(-6.572190433,
7.458430227)(-6.835116664, 7.279112686)(-7.071067810, 7.071067810)
\pscurve[linestyle=dashed,dash=3pt 2pt,
linewidth=0.7pt,linecolor=blue](3.239972649,
3.784232070)(3.016002733, 3.813737852)(2.790355808,
3.841123054)(2.563157331, 3.866372450)(2.334533653,
3.889471998)(2.104611877, 3.910408855)(1.873519869,
3.929171379)(1.641386114, 3.945749138)(1.408339678,
3.960132914)(1.174510162, 3.972314708)(.9400275752,
3.982287748)(.7050222879, 3.990046488)(.4696249919,
3.995586612)(.2339665670, 3.998905043)(-.001821962041,
3.999999934)(-.2376094680, 3.998870675)(-.4732648543,
3.995517896)(-.7086570977, 3.989943461)(-.9436552914,
3.982150468)(-1.178128778, 3.972143252)(-1.411947190,
3.959927376)(-1.644980498, 3.945509633)(-1.877099145,
3.928898040)(-2.108174046, 3.910101834)(-2.338076723,
3.889131466)(-2.566679352, 3.865998595)(-2.793854803,
3.840716086)(-3.019476766, 3.813297996)(-3.243419798,
3.783759570)(-3.465559357, 3.752117234)(-3.685771938,
3.718388580)(-3.903935101, 3.682592364)(-4.119927521,
3.644748490)(-4.333629109, 3.604878001)(-4.544921047,
3.563003063)(-4.753685830, 3.519146966)(-4.959807397,
3.473334089)(-5.163171117, 3.425589912)(-5.363663922,
3.375940979)(-5.561174339, 3.324414895)(-5.755592527,
3.271040316)(-5.946810392, 3.215846916)(-6.134721618,
3.158865383)(-6.319221702, 3.100127406)(-6.500208064,
3.039665643)(-6.677580075, 2.977513710)(-6.851239097,
2.913706172)(-7.021088575, 2.848278504)(-7.187034073,
2.781267085)(-7.348983306, 2.712709180)(-7.506846232, 2.642642909)
\pscurve[linewidth=0.7pt,linecolor=blue](-7.541335288,
2.626884454)(-7.634661083, 2.583391574)(-7.726419298,
2.539368262)(-7.816591095, 2.494823558)(-7.905157959,
2.449766606)(-7.992101704, 2.404206658)(-8.077404480,
2.358153070)(-8.161048770, 2.311615296)(-8.243017402,
2.264602892)(-8.323293546, 2.217125510)(-8.401860717,
2.169192900)(-8.478702786, 2.120814903)(-8.553803973,
2.072001451)(-8.627148860, 2.022762568)(-8.698722387,
1.973108364)(-8.768509857, 1.923049032)(-8.836496943,
1.872594853)(-8.902669684, 1.821756186)(-8.967014494,
1.770543468)(-9.029518161, 1.718967215)(-9.090167851,
1.667038017)(-9.148951112, 1.614766536)(-9.205855875,
1.562163505)(-9.260870455, 1.509239725)(-9.313983556,
1.456006061)(-9.365184274, 1.402473444)(-9.414462092,
1.348652870)(-9.461806901, 1.294555378)(-9.507208969,
1.240192090)(-9.550658982, 1.185574156)(-9.592148016,
1.130712797)(-9.631667553, 1.075619274)(-9.669209477,
1.020304902)(-9.704766082, .9647810364)(-9.738330065,
.9090590780)(-9.769894536, .8531504680)(-9.799453014,
.7970666860)(-9.826999429, .7408192468)(-9.852528126,
.6844197000)(-9.876033862, .6278796252)(-9.897511813,
.5712106316)(-9.916957567, .5144243552)(-9.934367133,
.4575324548)(-9.949736935, .4005466120)(-9.963063817,
.3434785277)(-9.974345045, .2863399189)(-9.983578300,
.2291425176)(-9.990761687, .1718980678)(-9.995893732,
.1146183233)(-9.998973380, .05731504492)(-10., 0.)
%
%Frame of L_3
%\psline[linewidth=1pt,linecolor=black]{->}(2.385846903,
%4.598164327)(3.385846903, 6.525431464)
\psline[linewidth=1pt,linecolor=black]{->}(2.385846903,
4.598164327)(.807282341, 5.945055549)
%
%Frame of L_2
\psline[linewidth=1pt,linecolor=black]{->}(7.541335284,
-2.626884455)(9.146007679, -1.889806122)
\psline[linewidth=1pt,linecolor=black]{->}(7.541335284,
-2.626884455)(6.961649616, -.804916391)
%
%Frame of L_1
\psline[linewidth=1pt,linecolor=black]{->}(-5.692795234,
-3.288576164)(-3.683918371, -3.844981079)
%
%Points
\qdisk(0,10){1pt} \qdisk(10,0){1pt} \qdisk(-5.692795234,
-3.288576164){2pt} \qdisk(2.385846903, 4.598164327){2pt}
\qdisk(-7.541335284, 2.626884455){1pt} \qdisk(7.541335284,
-2.626884455){2pt}
%Second points of intersection of the two ellipses
%\qdisk(3.289463472, 3.777394445){2pt}
%\qdisk(-3.289463472,-3.777394445){2pt}
%
\end{pspicture}
\caption{The Euler angles $2\alpha$, $2\beta$, 2$\gamma$}\label{fig:angles}%
\end{figure}
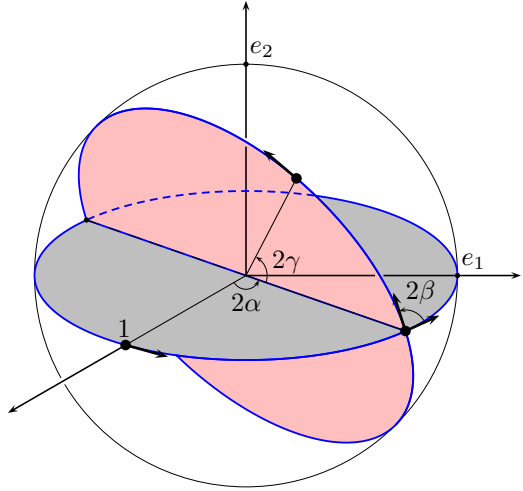
%%%%%%%%%%%%%%%%%%%%%%%%%%%%%%%%%%%%%%%%%%%%%%%%%%%%%%%%%%%%%%%%%%%%%%%%%%%%%%%%%%%%%%%%%%%%%%%%%%%%%%%%%%%%%%%%%%%%

 Conversely, given a general unit tangent vector $(x,u) \in \mathsf{T}^1{\mathbb S}^2$,
 we may read off a pair of elements $\pm a \in \A_2^{\rm unit}$ as follows so that
 the transformation $\rho_{a}=\rho_{-a}$ sends $(1, e_1)$ to $(x,u)$. %
 For this, consider the pointed, oriented great circles $C_{1, e_1}$ and $C_{x,u}$. %

 First suppose these two great circles are distinct. Let $z$ be one of the intersection points of them
 (the other one is $-z$).
% Let $2\alpha \in [0,2\pi)$ be the spherical distance from $1$ to $z$
% traversed along the orientation of $C_{1, e_1}$.
 Suppose that, starting from $(1,e_1) \in \mathsf{T}^1{\mathbb S}^2$, %
 we traverse along the oriented great circle $C_{1, e_1}$ in its orientation by distance $2\alpha \in [0,2\pi)$
 to arrive $(z,w) \in \mathsf{T}^1{\mathbb S}^2$. %
 Suppose further that we need to perform a rotation by angle $2\beta \in [0,2\pi)$ about $z$
 to arrive at $(z,v)\in \mathsf{T}^1{\mathbb S}^2$ so that %
 the great circles $C_{z,v}$ and $C_{x,u}$ coincide and are orientated in the same way. %
 Finally, let $2\gamma \in [0,2\pi)$ be the angle traversed from $z$ to $x$
 along the oriented great circle $C_{z,v}=C_{x,u}$ in its orientation.
% Thus we arrive at $(x,u) \in \mathsf{T}^1{\mathbb S}^2$
% by traversing along the oriented great circle $C_{z,v}$ by angle $2\gamma \in [0,2\pi)$. %
 In this way, we have obtained the desired decomposition (\ref{eqn:a=e}), that is,
 $a = \exp(\alpha e_1)\exp(\beta e_1e_2)\exp(\gamma e_1)$. %
 See Figure \ref{fig:angles} for an illustration of the procedure described above. %
 Note that if we choose $-z$ instead of $z$, then we obtain
 the Euler decomposition (\ref{eqn:a=e}) for $-a$ instead of $a$, that is,
 \begin{eqnarray}
 -a=\exp((\alpha+\textstyle\frac12\pi)e_1)\exp((\pi-\beta)e_1e_2)\exp((\gamma+\textstyle\frac12\pi)e_1). %
 \end{eqnarray} %

 \begin{remark}
 {\rm The two triples of angles $(2\alpha, 2\beta, 2\gamma)$ and $(2\alpha+\pi, 2\pi-2\beta, 2\gamma+\pi)$ %
 in $(\R/2\pi \Z)^3$ are (the two choices of) the Euler angles associated to the orthogonal transformation
 $\rho_a=\rho_{-a} \in {\rm SO}(3)$. Note that Euler angles have been widely used by physicists and astronomers
 (see for example \cite[pp.289]{hestenes} or \cite[pp.100]{landau}). }
 \end{remark}

 The case when the great circles $C_{1, e_1}$ and $C_{x,u}$ coincide is
 simpler. In this case, we have either $a = \exp(\gamma e_1)$ or
 $a = \exp(\textstyle\frac{\pi}{2} e_1e_2)\exp(\gamma e_1)$ for some $\gamma \in [0,\pi)$. %

 \subsection{\bf An identity of Arnold on quaternions} %

 As a special case of (\ref{eqn:a=ee}), we obtain identity (\ref{eqn:arnold}) below
 which was used in an essential way by Arnold in \cite{arnold1995}.

 \begin{proposition}[Arnold]
 For $s,t \in \R$, the following identity holds: %
 \begin{eqnarray}\label{eqn:arnold}
 \exp(s e_1)\exp(t e_1e_2)\exp(-s e_1) = \,\exp(t \exp(2s e_1)e_1e_2). %
 \end{eqnarray}
 \end{proposition}

% As a direct corollary of (\ref{eqn:arnold}), we obtain: %

 \begin{corollary}
 For $\alpha, \beta, \gamma \in \R$, there hold
 \begin{eqnarray}
 \hspace{-20pt}\exp(\alpha e_1)\exp(\beta e_1e_2)\exp(\gamma e_1) %
 \!\!\!&=&\!\!\! \exp(2\beta \exp(2\alpha e_1)e_1e_2)\exp((\alpha+\gamma) e_1) \\ %
 \!\!\!&=&\!\!\! \exp((\alpha+\gamma) e_1)\exp(2\beta \exp(-2\alpha e_1)e_1e_2). %
 \end{eqnarray}
 \end{corollary}

 %%%%%%%%%%%%%%%%%%%%%%%%%%%%%%%%%%%%%%%%%%%%%%%%%%%%%%%%%%%%%%%%%%%%%%%%%%%%%%%%%%%%%%%%%%%%%%%%%%%%%%%%%%%%%%%%%%%%%%%%
 %%%%%%%%%%%%%%%%%%%%%%%%%%%%%%%%%%%%%%%%%%%%%%%%%%%%%%%%%%%%%%%%%%%%%%%%%%%%%%%%%%%%%%%%%%%%%%%%%%%%%%%%%%%%%%%%%%%%%%%%

 \section{\bf Half-distances between oriented lines or line-plane flags in $\H^4$}\label{s:halfdist} %

 \nn In this section we define the $\{e_1, e_2\}$-quaternion half distances from one oriented line-plane flag
 (an oriented line and an oriented plane containing it) in $\H^4$ to another
 along an oriented line orthogonal to the two flags (when such a line exists),
 and the $e_1$- and $e_2$-complex half distances from one oriented line in $\H^4$ to another
 along an oriented line-plane flag orthogonal to the two lines.

 \subsection{Lines and planes in $\H^4$}

 Recall that we set the upper half-space model of the hyperbolic $4$-space as %
 $$ \H^4 = \{ x_0 + x_1e_1 + x_2e_2 + x_3e_3 \mid x_0,x_1,x_2,x_3\in \R, \, x_3>0 \} \subset \A_3^{(0,1)}. $$ %
 Its conformal boundary at infinity is
 $\partial \H^4 = \hat{\A}_2^{(0,1)} = \A_2^{(0,1)} \cup \{\infty\}$. %
 In this model, an orientation-preserving isometry of $\H^4$, $\eta \in {\rm Isom}^+(\H^4)$,
 is exactly a M\"{o}bius transformation of $\hat{\A}_2^{(0,1)}$ %
 given by a pair of Vahlen matrices $\pm A \in {\rm SL}(2,\Gamma_2)$.
 It is a well-known fact that the group ${\rm Isom}^+(\H^4)$ of orientation-preserving isometries of $\H^4$ %
 acts on the set of ordered orthogonal frames of $\H^4$ freely and transitively. %

% \begin{notation} {\rm
% For $x\in\H^4$ and $y\in\partial\H^4$, we shall use $R_{[x,y]}$ to denote the unique geodesic ray
% connecting $x$ to $y$, with the orientation from $x$ to $y$, when oriented. } %
% \end{notation}

 \begin{convention} {\rm
 We shall always orient $\H^4$ by choosing the orientation of $\H^4$ given by the ordered orthogonal frame at $e_3$ %
 \begin{eqnarray}\label{eqn:Frame}
 {\rm Frame}_{e_3}(1, e_1,e_2, \infty):=(\vL_{[-1,1]}, \vL_{[-e_1,e_1]}, \vL_{[-e_2,e_2]}, \vL_{[0,\infty]}). %
 \end{eqnarray} } %
 \end{convention}

% \begin{terminology} {\rm
% For convenience, we shall simply use the terms {\it lines} and {\it planes} to mean complete geodesic lines %
% and totally geodesic $2$-planes in $\H^4$. }
% \end{terminology}

 We introduce notations for oriented lines and planes in $\H^4$ as follows. %

 \begin{notation}
 {\rm For distinct points $u,v \in \partial \H^4$, we use $\vL_{[u,v]}$ to denote the unique line %
 in $\H^4$ connecting them, oriented from $u$ to $v$. } %
 \end{notation}

 \begin{notation}[The horizontal and vertical lines in $\H^4$]
 {\rm We denote respectively by $\vL_{\rm h}$ and $\vL_{\rm v}$ the horizontal and vertical oriented lines in $\H^4$: %
 \begin{eqnarray}\label{eqn:LhLv}
 \vL_{\rm h} = \vL_{[-1,1]}, \quad \vL_{\rm v} = \vL_{[0,\infty]}. %
 \end{eqnarray} } %
 \end{notation}

 \begin{notation}
 {\rm For intersecting (distinct) oriented lines $\vL_1$ and $\vL_2$ in $\H^4$, we use $\vL_1 \vee \vL_2$ to denote
 the oriented plane in $\H^4$ which contains both $L_1$ and $L_2$, with the orientation determined by
 the ordered frame $(\vL_1,\vL_2)$.
 In particular, when

 \vskip 3pt

 \centerline{$\vL_1=\vL_{[u,v]}$  \quad  and  \quad  $\vL_2=\vL_{[x,y]}$} %

 \vskip 3pt

 \nn with distinct $u,v,x,y \in \partial \H^4$, we shall simply write $\vL_1 \vee \vL_2$ as $\vPi_{[u,v]\vee[x,y]}$. %
 Note that, with this notation, %
 $\vPi_{[u,v]\vee[x,y]}=\vPi_{[y,x]\vee[u,v]}=\vPi_{[x,y]\vee[v,u]}=\vPi_{[v,u]\vee[y,x]}$. } %
 \end{notation}

 \begin{notation}[The horizontal and vertical planes in $\H^4$]
 {\rm We denote respectively by $\vPi_{\rm h}$ and $\vPi_{\rm v}$ the horizontal and vertical oriented planes in $\H^4$: %
 \begin{eqnarray}\label{eqn:PIhPIv}
 \vPi_{\rm h} = \vPi_{[-1,1]\vee[-e_1,e_1]}, \quad  \vPi_{\rm v} = \vPi_{[-e_2,e_2]\vee[0,\infty]}. %
 \end{eqnarray} } %
 \end{notation}

% \begin{notation}[The horizontal and vertical flags in $\H^4$]
% {\rm We denote respectively by $\vF_{\rm h}$ and $\vF_{\rm v}$ the horizontal and vertical oriented flags in $\H^4$: %
% \begin{eqnarray}\label{eqn:FhFv}
% \vF_{\rm h} = (\vL_{\rm h}, \vPi_{\rm h}), \quad \vF_{\rm v} = (\vL_{\rm v}, \vPi_{\rm v}). %
% \end{eqnarray} } %
% \end{notation}

 \subsection{Orientation-preserving isometries of $\H^4$ acting along $L_{[0,\infty]}$}\label{ss:0infty} %

 We say that an isometry $\eta \in {\rm Isom}^+(\H^4)$ acts along the line $L_{[0,\infty]}$
 if $\eta$ fixes both $0$ and $\infty$ on $\partial \H^4$.
 In this case, $\eta$ can be decomposed as the composite, in either order, of a pure hyperbolic translation
 along $\vL_{[0,\infty]}$ and a rotation about $\vL_{[0,\infty]}$. %

 Indeed, the Vahlen matrices $\pm A$ of $\eta$ are of the form
 $$ A = \begin{pmatrix}\;a & 0\phantom{aaa} \\ \;0 & {a^*}^{-1} \end{pmatrix}, \quad\quad
    a \in \Gamma_2 = \A_2^{\times} = \A_2 \backslash \{0\}. $$ %

 For $\lambda >0$, let $\tau_\lambda \in {\rm Isom}^+(\H^4)$ be the isometry
 whose Vahlen matrices are
 $\pm \Big(\small\begin{matrix}\;\lambda & 0\phantom{aa} \\ \;0 & \lambda^{-1}\end{matrix}\normalsize\Big)$. %
 Then $\tau_\lambda$ is a pure hyperbolic translation along the oriented line $\vL_{[0,\infty]}$ %
 by the signed translation distance $2\log_{\R} \lambda \in \R$.

 By Proposition \ref{prop:axis}, the element $\rho_{a/|a|} \in {\rm SO}(3)$, when regarded as in ${\rm Isom}^+(\H^4)$, is
 a rotation about a plane in $\H^4$ which contains the line $L_{[0,\infty]}$.

 It is easy to verify that
 $\eta \,=\, \tau_{|a|} \,\rho_{a/|a|} = \rho_{a/|a|} \,\tau_{|a|}$. %

 \subsection{Line-plane flags in $\H^4$}\label{ss:flag}

 We shall consider the geometric configuration in $\H^4$ which is a plane with a line in it singled out.

 \begin{definition}
 {\rm A {\it flag} or {\it line-plane flag} $F$ in $\H^4$ is an ordered pair $(L,\Pi)$
 where $L$ is a line and $\Pi$ a plane in $\H^4$ such that $L$ is contained in $\Pi$.
 An {\it oriented flag} $\vF=(\vL,\vPi)$ is obtained from a flag $F=(L,\Pi)$ by orienting each of $L$ and $\Pi$.} %
 \end{definition}

 \begin{notation}[The horizontal and vertical flags in $\H^4$]
 {\rm We denote respectively by $\vF_{\rm h}$ and $\vF_{\rm v}$ the horizontal and vertical oriented flags in $\H^4$: %
 \begin{eqnarray}\label{eqn:FhFv}
 \vF_{\rm h} = (\vL_{\rm h}, \vPi_{\rm h}), \quad \vF_{\rm v} = (\vL_{\rm v}, \vPi_{\rm v}). %
 \end{eqnarray} } %
 \end{notation}

 \begin{definition}
 {\rm (i) We say that two lines or a line and a plane in $\H^4$ are {\it orthogonal} to each other
 if they intersect and are perpendicular to each other.
 (ii) We say that a line $L'$ and a flag $(L,\Pi)$ in $\H^4$ are {\it orthogonal}
 to each other if $L'$ is orthogonal to both $L$ and $\Pi$, that is,
 $L'$ intersects $L$ and $L'$ is perpendicular to $\Pi$.} %
 \end{definition}

 \begin{definition}
 {\rm By a {\it flag-line cross}, $(F, L')$, in $\H^4$ we mean a flag $F$ and a line $L'$ in $\H^4$
 which are orthogonal to each other. %
 We say that a flag-line cross in $\H^4$ is oriented if both the flag and the line are oriented.} %
 \end{definition}

 \subsection{Correspondence between flag-line crosses and orthogonal frames in $\H^4$}\label{ss:flc-oof}

 We set up a one-to-one correspondence between the set of oriented flag-line crosses in $\H^4$
 and the set of ordered orthogonal frames in $\H^4$ as follows.

 Given an ordered orthogonal frame in $\H^4$, that is, an ordered quadruple
 of oriented lines $(\vL_1, \vL_2, \vL_3, \vL_4)$ in $\H^4$ meeting at a point,
 we obtain an oriented flag-line cross in $\H^4$ such that
 the oriented flag is $(\vL_1, \vL_1 \vee \vL_2)$ and the oriented line is $\vL_4$. %

 Conversely, given an oriented flag-line cross $(\vF, \vL')$ in $\H^4$, where $\vF=(\vL,\vPi)$ and $L \cap L' = x$, %
 we obtain an ordered orthogonal frame $(\vL, \vL_2, \vL_3, \vL')$ in $\H^4$ as follows:
 $\vL_2$ is the oriented line in $\Pi$ passing through $x$ and orthogonal to $L$
 such that the ordered pair $(\vL,\vL_2)$ determines the positive orientation of $\vPi$, and
 $\vL_3$ is the oriented line in $\H^4$ passing through $x$ and orthogonal to all of $L, L_2$ and $L'$
 such that the ordered quadruple $(\vL, \vL_2, \vL_3, \vL')$ determines the positive orientation of $\H^4$. %

 It follows from the one-to-one correspondence just established that, %
 given two oriented flag-line crosses in $\H^4$, %
 there exists a unique orientation-preserving isometry of $\H^4$ %
 which sends the first oriented flag-line cross to the second. %

 \subsection{A specific isometry of $\H^4$} %

 We shall need in \S \ref{ss:e1e2} the isometry $\iota \in {\rm Isom}^+(\H^4)$ which sends %
 the oriented flag-line cross $(\vF_{\rm v},\vL_{\rm h})$ to $(\vF_{\rm h},\vL_{\rm v})$, or equivalently, %
 \begin{eqnarray}
 \vL_{[-1,1]} \stackrel{\iota}{\longleftrightarrow} \vL_{[0,\infty]}, \quad %
 \vL_{[e_2,-e_2]} \stackrel{\iota}{\longleftrightarrow} \vL_{[-e_1,e_1]}. %
 \end{eqnarray}
 In terms of orthogonal frames of $\H^4$, this is equivalent to the requirement that
 \begin{eqnarray}
 \iota\big({\rm Frame}_{e_3}(1, e_1, e_2, \infty)\big) = {\rm Frame}_{e_3}(\infty, -e_2, -e_1, 1). %
 \end{eqnarray}
 The isometry $\iota \in {\rm Isom}^+(\H^4)$ is unique and has Vahlen matrices $\pm K$, where %
 \begin{eqnarray}\label{eqn:K0}
 K = \frac{e_1+e_2}{2}\begin{pmatrix}\,1\!\! & \phantom{-}1\, \\ \,1\!\! & -1\,\end{pmatrix}. %
 \end{eqnarray}
 It is interesting to note that $\iota$ is an involution ($\iota^2={\rm id}$), with $K^{-1} = -K$.
%
% It is easy to check that $K^2=\Big(\small\begin{matrix}\phantom{-}0 & 1\, \\ -1 & 0\,\end{matrix}\normalsize\Big)$ and $K^4=-I$. %
% The corresponding M\"{o}bius transformation of $\A_2^{(0,1)}$ % $T_{K}=T_{-K}:\A_2^{(0,1)}\rightarrow\A_2^{(0,1)}$ %
% defines an isometry $\eta_{K}=\eta_{-K} \in {\rm Isom}^+(\H^4)$.
% It is easy to check that $\eta(e_3)=e_3$ and
% $$ \eta: \ {\rm Frame}_{e_3}(1, e_1, e_2, \infty) \longmapsto {\rm Frame}_{e_3}(\infty, -e_2, -e_1, 1). $$ %
% Therefore $\eta_{K}$ maps the oriented lines
% $$ L_{[-1,1]}, \ L_{[-e_1,e_1]}, \ L_{[-e_2,e_2]}, \ L_{[0,\infty]} $$ %
% respectively onto
% $$ \phantom{+}L_{[0,\infty]}, \ L_{[e_2,-e_2]}, \ L_{[e_1,-e_1]}, \ L_{[-1,1]}. $$ %
% In particular, $\eta_{K}(L_{[-1,1]})=L_{[0,\infty]}$ and
% $\eta_{K}(\Pi_{[-1,1]\vee[-e_1,e_1]}) = \Pi_{[-1,1]\vee[-e_1,e_1]}$. %
% It follows that $\eta_{K}$ fixes the oriented plane $\Pi_{[-e_1,e_1]\vee[-e_2,e_2]}$ pointwise and
% restricts to the coherently oriented, orthogonal complement plane $\Pi_{[-1,1]\vee[0,\infty]}$
% as the rotation of angle $\pi/2$ about the point $e_3$. %

% The conjugation relation (\ref{eqn:conjK}) below among Vahlen matrices is easy to verify. %
%
% \begin{proposition}
% For $\sigma \in \A_2$ and $K$ given in (\ref{eqn:K}), we have %
%%
% \begin{eqnarray}\label{eqn:conjK}
% \begin{pmatrix} \exp(\sigma) & 0 \\ 0 & \exp(-\sigma^*)\end{pmatrix} %
% \,=\, K \begin{pmatrix} \cosh\sigma & \sinh\sigma \\ \sinh\sigma & \cosh\sigma \end{pmatrix} K^{-1}. %
% \end{eqnarray}
% \end{proposition}

 \subsection{The quaternion half distances between oriented flags in $\H^4$}\label{ss:qhdist} %

 We define two quaternion half distances from one oriented flag in $\H^4$ to another
 along an oriented common orthogonal line of the two flags (if such a line exists).
 Recall that we have identified the algebra of quaternions ${\mathbb H}$ with the Clifford algebra $\A_2$. %

 \begin{definition}\label{def:clfd}
 {\rm
 Suppose $\vF_1$ and $\vF_2$ are two oriented flags both orthogonal to an oriented line $\vL$ in $\H^4$. %
 Let $\iota$ be the unique orientation-preserving isometry of $\H^4$ such that %
 $\iota(\vL)=\vL_{\rm v}$ and $\iota(\vF_1)=\vF_{\rm h}$.
 % :=\vL_{[0,\infty]} :=(\vL_{[-1,1]},\vPi_{[-1,1]\vee[-e_1,e_1]})
 Then there exists a unique orientation-preserving isometry $\eta$ of $\H^4$ such that %
 $\eta(\iota(\vL))=\iota(\vL)$ and $\eta(\iota(\vF_1))=\iota(\vF_2)$. %
 Since $\eta(\vL_{\rm v})=\vL_{\rm v}$, $\eta$ has Vahlen matrices %
 $\pm \bigg(\begin{matrix}\,a & 0\phantom{bbb}\! \\ \,0 & {a^*}^{-1}\! \end{matrix}\bigg)$ %
 for some pair $\pm a \in \A_2\setminus\{0\}$. %
 We then define the two {\it quaternion half distances from $\vF_1$ to $\vF_2$ along $\vL$} as
 $$ \delta_{\vL}(\vF_1,\vF_2) \,=\, \log (\pm a) \quad {\rm modulo \ period}. $$ } %
 \end{definition}

 \begin{remark} {\rm
 The discussions in \S \ref{ss:geoeuler} allow us to read off the pair $\pm a$ (or precisely, $\pm a/|a|$) %
 in the above definition geometrically from the configuration $(\vF_1, \vL, \vF_2)$. }
 \end{remark}

 The following proposition recaptures Definition \ref{def:clfd} and will be used later .

 \begin{proposition}\label{prop:quatdist}
 Suppose $\vF_1$ and $\vF_2$ are two oriented flags orthogonal to an oriented line $\vL$ in $\H^4$. %
 Let $\tau \in {\rm Isom}^+(\H^4)$ be such that $\tau(\vL)=\vL$ and $\tau(\vF_1)=\vF_2$, %
 and let $\iota \in {\rm Isom}^+(\H^4)$ be such that $\iota(\vL)=\vL_{\rm v}$ and $\iota(\vF_1)=\vF_{\rm h}$. %
 Then the isometry $\iota\tau\iota^{-1}$ has Vahlen matrices %
 $\pm \Big(\,\small\begin{matrix} \exp\delta & 0 \\ \!\!0 & \!\!\exp(-\delta) \end{matrix}\normalsize\,\Big)$
 where $\delta=\delta_{\vL}(\vF_1,\vF_2)$. \qed %
 \end{proposition}

 \subsection{The $e_1$- and $e_2$-complex half distances between oriented lines in $\H^4$}\label{ss:e1e2} %

 \begin{definition}\label{def:e1cmplx}
 {\rm
 Suppose $\vL_1$ and $\vL_2$ are two oriented lines orthogonal to an oriented flag $\vF$ in $\H^4$. %
 Let $\iota$ be the unique orientation-preserving isometry of $\H^4$ such that %
 $\iota(\vF)=\vF_{\rm v}$ and $\iota(\vL_1)=\vL_{\rm h}$.
 Let $\eta$ be the unique orientation-preserving isometry of $\H^4$ such that %
 $\eta(\iota(\vF))=\iota(\vF)$ and $\eta(\iota(\vL_1))=\iota(\vL_2)$. %
 It follows from \S \ref{ss:0infty} and Proposition \ref{prop:e2fixed} that
 $\eta$ has Vahlen matrices %
 $\pm \Big(\small\begin{matrix}\,a & 0\phantom{bb}\! \\ \,0 & a^{-1}\! \end{matrix}\normalsize\Big)$
 for a unique pair $\pm a \in (\R + \R e_1)\backslash \{0\}$. %
 We define the two {\it $e_1$-complex half distances from $\vL_1$ to $\vL_2$ along $\vF$} as %
 $$ \delta_{\vF}^{(e_1)}(\vL_1,\vL_2) = \log (\pm a) \in (\R + \R e_1)/2\pi e_1\Z. $$ %
  }
 \end{definition}
% We remark that $\iota(L_2)=L_{[-a^2,a^2]}$. %

 \begin{definition}\label{def:e2cmplx}
 {\rm
 Suppose $\vL_1$ and $\vL_2$ are two oriented lines both orthogonal to an oriented flag $\vF$ in $\H^4$. %
 Let $\iota$ be the unique orientation-preserving isometry of $\H^4$ such that %
 $\iota(\vF)=\vF_{\rm h}$ and $\iota(\vL_1)=\vL_{\rm v}$.
 Let $\eta$ be the unique orientation-preserving isometry of $\H^4$ such that %
 $\eta(\iota(\vF))=\iota(\vF)$ and $\eta(\iota(\vL_1))=\iota(\vL_2)$. %
 By Proposition \ref{prop:1-1}, $\eta$ has Vahlen matrices %
 $\pm \Big(\,\small\begin{matrix} a & \!\!b \\ b & \!\!a \end{matrix}\normalsize\,\Big)
 =\pm \Big(\,\small\begin{matrix} \cosh\delta & \!\!\sinh\delta \\
                                  \sinh\delta & \!\!\cosh\delta \end{matrix}\normalsize\,\Big)$ %
 (it will be shown in Proposition \ref{prop:e2=chie1} below that $\delta \in \R + \R e_2$) %
 for a unique pair $\delta, \delta + \pi e_2 \in \R + \R e_2 \!\!\mod 2\pi e_2$. %
 We define the two {\it $e_2$-complex half distances from $\vL_1$ to $\vL_2$ along $\vF$} as %
 \begin{eqnarray}
 \delta_{\vF}^{(e_2)}(\vL_1,\vL_2) = \log (\pm(a+b)) = \delta, \delta + \pi e_2 \in (\R + \R e_2)/2\pi e_2\Z. %
 \end{eqnarray}
  }
 \end{definition}

 \begin{remark} {\rm The quaternion, $e_1$- and $e_2$-complex half distances defined above are
 easily seen to be invariant under orientation-preserving isometries of $\H^4$.} %
 \end{remark}

 The $e_1$- and $e_2$-complex half distances are simply related by the $\R$-linear map %
 $\chi: \R + \R e_1 \rightarrow \R + \R e_2$ determined by $\chi(1)=1$ and $\chi(e_1)=e_2$. %

 \begin{proposition}\label{prop:e2=chie1}
 Suppose $\vL_1$ and $\vL_2$ are two oriented lines orthogonal to an oriented flag $\vF$ in $\H^4$. %
 Then
 \begin{eqnarray}
 \delta_{\vF}^{(e_2)}(\vL_1,\vL_2) = \chi(\delta_{\vF}^{(e_1)}(\vL_1,\vL_2)). %
 \end{eqnarray}
 \end{proposition}

 \begin{proof}
 With no loss of generality, we may assume that $\vF=\vF_{\rm v}$ and $\vL_1=\vL_{\rm h}$. %
 Let $\tau \in {\rm Isom}^+(\H^4)$ be such that $\tau(\vF)=\vF$ and $\tau(\vL_1)=\vL_2$.
 Then $\tau$ has Vahlen matrices
 $\pm \Big(\small\begin{matrix}\,\exp \delta & 0 \\ \,0 & \exp(-\delta^*) \end{matrix}\normalsize\Big)$
 where $\delta=\delta_{\vF}^{(e_1)}(\vL_1,\vL_2)$.

 Let $K$ be the Vahlen matrix defined as in (\ref{eqn:K0}) and let
 $\iota \in {\rm Isom}^+(\H^4)$ be the isometry corresponding to $\pm K$.
 Then $\iota(\vF)=\vF_{\rm h}$ and $\iota(\vL_1)=\vL_{\rm v}$. %

 It is easy to check that the isometry $\eta:=\iota \tau \iota^{-1}$
 satisfies $\eta(\iota(\vF))=\iota(\vF)$ and $\eta(\iota(\vL_1))=\iota(\vL_2)$, %
 with Vahlen matrices (write $e=e_1+e_2$ temporarily) %
 \begin{eqnarray}
 \hspace{-20pt} \pm \, K\left(\begin{matrix}\,\exp \delta & 0 \\ \,0 & \exp(-\delta^*) \end{matrix}\right) \, K^{-1} %
 &\!\!=\!\!& \pm \left(\begin{matrix} e (\cosh\delta) e^{-1} & e (\sinh\delta) e^{-1} \\ %
                                      e (\sinh\delta) e^{-1} & e (\cosh\delta) e^{-1} \end{matrix}\right) \nonumber \\ %
 &\!\!=\!\!& \pm \left(\begin{matrix}\cosh\chi(\delta) & \sinh\chi(\delta) \\ %
                                     \sinh\chi(\delta) & \cosh\chi(\delta) \end{matrix}\right). %
 \end{eqnarray}
 By definition, we have $\delta_{\vF}^{(e_2)}(\vL_1,\vL_2)=\chi(\delta)=\chi(\delta_{\vF}^{(e_1)}(\vL_1,\vL_2))$.
 \end{proof}

 The following proposition recaptures Definition \ref{def:e2cmplx} and will be used later .

 \begin{proposition}\label{prop:e2dist}
 Suppose $\vL_1$ and $\vL_2$ are two oriented lines both orthogonal to an oriented flag $\vF$ in $\H^4$. %
 Let $\tau \in {\rm Isom}^+(\H^4)$ be such that $\tau(\vF)=\vF$ and $\tau(\vL_1)=\vL_2$, %
 and let $\iota \in {\rm Isom}^+(\H^4)$ be such that $\iota(\vF)=\vF_{\rm h}$ and $\iota(\vL_1)=\vL_{\rm v}$. %
 Then the isometry $\iota\tau\iota^{-1}$ has Vahlen matrices %
 $\pm \Big(\,\small\begin{matrix} \cosh\delta & \!\!\sinh\delta \\
                                  \cosh\delta & \!\!\sinh\delta \end{matrix}\normalsize\,\Big)$
 where $\delta=\delta_{\vF}^{(e_2)}(\vL_1,\vL_2)$. \qed %
 \end{proposition}

 \subsection{The quaternion and the $e_2$-complex half distances under change of orientations}

 For an oriented line $\vL$ (resp. oriented plane $\vPi$) in $\H^4$, let $\vL^-$ denote %
 the same line with the opposite orientation; similarly, $\vPi^-$, for an oriented plena $\vPi$. %
 For an oriented flag $\vF=(\vL,\vPi)$ in $\H^4$, let $\vF^{-+}$, $\vF^{+-}$ and $\vF^{--}$denote %
 the oriented flags $(\vL^-,\vPi)$, $(\vL,\vPi^-)$ and $(\vL^-,\vPi^-)$, respectively. %
 Thus $\vF^{--}=(\vF^{-+})^{+-}$.

 The following two propositions record the changes of the $e_2$-complex half distances between two oriented lines %
 and the quaternion half distances between two oriented flags in $\H^4$ under a change of orientations of
 the oriented lines and flags involved
 and we leave it to the reader to verify the relations contained therein. %

 \begin{proposition}
 Suppose $\vL_1$ and $\vL_2$ are two oriented lines in $\H^4$ orthogonal to an oriented flag $\vF$. %
 Let $\delta_{\vF}(\vL_1,\vL_2)=\delta^{(e_2)}_{\vF}(\vL_1,\vL_2)$ be the $e_2$-complex half distances %
 from $\vL_1$ to $\vL_2$ along $F$. Then we have the following relations: %
 \begin{eqnarray*}
 \delta_{\vF}(\vL_2,\vL_1) \!\!&=&\!\! - \delta_{\vF}(\vL_1,\vL_2) \mod 2\pi e_2; \\ %
 \delta_{\vF}(\vL_1^{-},\vL_2) \!\!&=&\!\! \delta_{\vF}(\vL_1,\vL_2) + \frac{\pi}{2} e_2 \mod 2\pi e_2; \\ %
 \delta_{\vF}(\vL_1,\vL_2^{-}) \!\!&=&\!\! \delta_{\vF}(\vL_1,\vL_2) + \frac{\pi}{2} e_2 \mod 2\pi e_2; \\ %
 \delta_{\vF^{+-}}(\vL_1,\vL_2) \!\!&=&\!\! \overline{\delta_{\vF}(\vL_1,\vL_2)} \mod 2\pi e_2; \\ %
 \delta_{\vF^{-+}}(\vL_1,\vL_2) \!\!&=&\!\! -\overline{\delta_{\vF}(\vL_1,\vL_2)} \mod 2\pi e_2; \\ %
 \delta_{\vF^{--}}(\vL_1,\vL_2) \!\!&=&\!\! -\delta_{\vF}(\vL_1,\vL_2) \mod 2\pi e_2. %
 \end{eqnarray*}
 \end{proposition}

 \begin{proposition}
 Suppose $\vF_1$ and $\vF_2$ are two oriented flags in $\H^4$ orthogonal to an oriented line $\vL$. %
 Let $\delta_{\vL}(\vF_1,\vF_2)$ be the quaternion half distances from $\vF_1$ to $\vF_2$ along $\vL$. %
 Then we have the following relations: %
 \begin{eqnarray*}
 \delta_{\vL}(\vF_2,\vF_1) \!\!&=&\!\! -\delta_{\vL}(\vF_1,\vF_2) \mod(\rm{period}); \\ %
 \delta_{\vL^{-}}(\vF_1,\vF_2) \!\!&=&\!\! -\delta_{\vL}(\vF_1,\vF_2) \mod(\rm{period}); \\ %
 \delta_{\vL}(\vF_1,\vF_2^{+-}) \!\!&=&\!\! \delta_{\vL}(\vF_1,\vF_2)\oplus\frac{\pi}{2}e_1e_2 \mod(\rm{period}); \\ %
 \delta_{\vL}(\vF_1,\vF_2^{-+}) \!\!&=&\!\! \delta_{\vL}(\vF_1,\vF_2)\oplus\frac{\pi}{2}e_1 \mod(\rm{period}); \\ %
 \delta_{\vL}(\vF_1^{+-},\vF_2) \!\!&=&\!\! \frac{\pi}{2}e_1e_2\oplus\delta_{\vL}(\vF_1,\vF_2) \mod(\rm{period}); \\ %
 \delta_{\vL}(\vF_1^{-+},\vF_2) \!\!&=&\!\! \frac{\pi}{2}e_1\oplus\delta_{\vL}(\vF_1,\vF_2) \mod(\rm{period}). %
 \end{eqnarray*}
 \end{proposition}

%% Recall that the definition of the relation \,$a = b \;{\rm mod}\,({\rm period})$ %
%% for two elements $a,b$ in $\A_2$ is given earlier in \S \ref{ss:polar}. %

%%%%%%%%%%%%%%%%%%%%%%%%%%%%%%%%%%%%%%%%%%%%%%%%%%%%%%%%%%%%%%%%%%%%%%%%%%%%%%%%%%%%%%%%%%%%%%%%%%%%%%%%%%%%%%%%%%%
%%%%%%%%%%%%%%%%%%%%%%%%%%%%%%%%%%%%%%%%%%%%%%%%%%%%%%%%%%%%%%%%%%%%%%%%%%%%%%%%%%%%%%%%%%%%%%%%%%%%%%%%%%%%%%%%%%%

 \section{\bf Generalized Delambre-Gauss formulas for oriented,
 augmented right-angled hexagons in $\H^4$}\label{s:gauss4arah} %

 \nn In this section we restate and prove our generalized Delambre-Gauss formulas for oriented,
 augmented right-angled hexagons in $\H^4$. The proof we shall give is in principle the same as that for generalized
 Delambre-Gauss formulas for oriented right-angled hexagons in $\H^3$ which we give in \S \ref{s:gauss4rah}, the appendix. %

 \vskip 6pt

 Similar to that in the case of $\H^3$, a right-angled hexagon in $\H^4$ is defined to be a cyclic six-tuple
 $\{L_n\}_{n=1}^{6}$ of lines in $\H^4$, with indices modulo $6$, such that for all $n = 1, \cdots, 6$,
 the adjacent lines $L_n$ and $L_{n+1}$ are orthogonal to each other. %

 \begin{definition}{\rm
 An {\it augmented right-angled hexagon} in $\H^4$ is by definition a cyclic six-tuple $\{S_n\}_{n=1}^{6}$,
 with indices modulo $6$, such that $S_1, S_3, S_5$ are lines and $S_2, S_4, S_6$ are flags
 in $\H^4$, or $S_1, S_3, S_5$ are flags and $S_2, S_4, S_6$ are lines in $\H^4$,
 and such that for all $n = 1, \cdots, 6$, $S_n$ and $S_{n+1}$ are orthogonal to each other.} %
 \end{definition}

 \begin{definition}{\rm
 An {\it oriented, augmented right-angled hexagon} $\{\vS_n\}_{n=1}^{6}$ in $\H^4$ is obtained from %
 an augmented right-angled hexagon $\{S_n\}_{n=1}^{6}$ in $\H^4$ by arbitrarily orienting each $S_n$ as $\vS_n$,
 $n = 1, \cdots, 6$.} %
 \end{definition}

 \begin{definition}{\rm
 We call any choice of one of the two quaternion or $e_2$-complex half distances
 $\delta_{\vS_n}(\vS_{n-1}, \vS_{n+1})$ as defined in \S \ref{ss:qhdist} and \S \ref{ss:e1e2} %
 %(where we use the $e_2$-complex half distance from \S \ref{ss:e1e2})
 {\it a half side-length} of the oriented, augmented right-angled hexagon $\{\vS_n\}_{n=1}^{6}$
 along $\vS_n$, $n = 1, 2, \cdots, 6$.} %
 \end{definition}

 \begin{remark}
 {\rm It follows that if, say, $S_1, S_3, S_5$ are lines and $S_2, S_4, S_6$ are flags in $\H^4$, %
 then the half side-lengths along $\vS_1$, $\vS_3$, $\vS_5$ are $\{e_1,e_2\}$-quaternions
 (that is, elements in $\A_2$) modulo period, and those along $\vS_2$, $\vS_4$, $\vS_6$ %
 are $e_2$-complex numbers (that is, elements in $\R + \R e_2$) modulo $2\pi e_2$. } %
 \end{remark}

 We obtain the following generalized Delambre-Gauss formulas for the quaternion and %
 the $e_2$-complex half side-lengths of an oriented, augmented right-angled hexagon in $\H^4$ %
 which is merely a restatement of Theorem \ref{thm:intro-gauss}. %

 \begin{theorem}\label{thm:gauss}
 Let $\{\vS_i\}_{n=1}^{6}$ be an oriented, augmented right-angled hexagon in $\H^4$ with %
 chosen quaternion and $e_2$-complex half side-lengths $\delta_n$ along $\vS_n$, $n=1,\cdots,6$. %
 Then the following generalized Delambre-Gauss formulas hold: %
 \begin{eqnarray}
 & & \hspace{-50pt} \sinh\delta_1 \cosh\delta_2 \sinh\delta_3 + \cosh\delta_1 \cosh\delta_2 \cosh\delta_3 \nonumber \\ %
 &=& \varepsilon(\sinh\delta_4 \cosh\delta_5 \sinh\delta_6 + \cosh\delta_4 \cosh\delta_5 \cosh\delta_6)^*;  \label{eqn:1} \\ %
 & & \hspace{-50pt} \sinh\delta_1 \cosh\delta_2 \cosh\delta_3 + \cosh\delta_1 \cosh\delta_2 \sinh\delta_3 \nonumber \\ %
 &=& \varepsilon(\sinh\delta_4 \sinh\delta_5 \sinh\delta_6 - \cosh\delta_4 \sinh\delta_5 \cosh\delta_6)^*;  \label{eqn:2} \\ %
 & & \hspace{-50pt} \sinh\delta_1 \sinh\delta_2 \sinh\delta_3 - \cosh\delta_1 \sinh\delta_2 \cosh\delta_3 \nonumber \\ %
 &=& \varepsilon(\sinh\delta_4 \cosh\delta_5 \cosh\delta_6 + \cosh\delta_4 \cosh\delta_5 \sinh\delta_6)^*;  \label{eqn:3} \\ %
 & & \hspace{-50pt} \sinh\delta_1 \sinh\delta_2 \cosh\delta_3 - \cosh\delta_1 \sinh\delta_2 \sinh\delta_3 \nonumber \\ %
 &=& \varepsilon(\sinh\delta_4 \sinh\delta_5 \cosh\delta_6 - \cosh\delta_4 \sinh\delta_5 \sinh\delta_6)^*,  \label{eqn:4} %
 \end{eqnarray}
 with $\varepsilon = 1$ or $-1$, depending on the choice of the half side-lengths $\{\delta_n\}_{n=1}^{6}$. %
 \end{theorem}

 The main idea in our proof of Theorem \ref{thm:gauss} is to use an identity
 involving six isometries of $\H^4$, namely (\ref{eqn:arah6isom1}),
 rewritten as a different identity (\ref{eqn:arah6isom2}) in Lemma \ref{lem:arah6isom} below,
 where the entries of the Vahlen matrices of the isometries in (\ref{eqn:arah6isom2}) are
 given by appropriate functions of the quaternion and the $e_2$-complex half side-lengths %
 associated to the given oriented, augmented, right-angled hexagon in $\H^4$. %

%%%%%%%%%%%%%%%%%%%%%%%%%%%%%%%%%%%%%%%%%%%%%%%%%%%%%%%%%%%%%%%%%%%%%%%%%%%%%%%%%%%%%%%%%%%%%%%%%%%%%%%%%%%%%%%%%%%%%%%
 \begin{lemma}\label{lem:arah6isom}
 Given an oriented, augmented, right-angled hexagon $\{\vS_n\}_{n=1}^{6}$ in $\H^4$,
 let $\tau_n, \iota_n, \eta_n \in {\rm Isom}^{+}(\H^4)$, $n=1,\cdots,6$ be determined and defined by %
 \begin{eqnarray}
 & \tau_n(\vS_n)=\vS_n, \quad \tau_n(\vS_{n-1})=\vS_{n+1}; \\ %
 & \iota_n := (\tau_n\cdots\tau_2\tau_1)^{-1}; \\
 & \hspace{-16pt} \eta_n := \iota_{n-1} \tau_n \iota_{n-1}^{-1}, %
 \end{eqnarray}%
 with indices modulo $6$. Then the isometries satisfy %
 \begin{eqnarray}
 & \tau_{6} \cdots \tau_{2} \tau_{1} = {\rm id}; \label{eqn:arah6isom1} \\ %
 & \eta_{1} \eta_{2} \cdots \eta_{6} = {\rm id}; \label{eqn:arah6isom2} \\ %
 & \iota_n(\vS_n)=\vS_1, \quad \iota_n(\vS_{n+1})=\vS_6, \quad n=1,3,5; \label{eqn:arah6isom3} \\ %
 & \iota_n(\vS_{n})=\vS_6, \quad \iota_n(\vS_{n+1})=\vS_1, \quad n=2,4,6. \label{eqn:arah6isom4} %
 \end{eqnarray}
 \end{lemma}

 \begin{proof}
 To prove (\ref{eqn:arah6isom1}), write $\tau=\tau_{6} \tau_{5} \tau_{4} \tau_{3} \tau_{2} \tau_{1}$. %
 Then $\tau \in {\rm Isom}^{+}(\H^4)$ and it can be checked that $\tau(\vS_6)=\vS_6$ and %
 $\tau(\vS_1)=\vS_1$ by going through the following table: %
 \begin{eqnarray*}
 && \vS_1 \stackrel{\tau_1}{\longrightarrow} \vS_1 \stackrel{\tau_2}{\longrightarrow} \vS_3 %
          \stackrel{\tau_3}{\longrightarrow} \vS_3 \stackrel{\tau_4}{\longrightarrow} \vS_5
          \stackrel{\tau_5}{\longrightarrow} \vS_5 \stackrel{\tau_6}{\longrightarrow} \vS_1, \\ %
 && \vS_6 \stackrel{\tau_1}{\longrightarrow} \vS_2 \stackrel{\tau_2}{\longrightarrow} \vS_2 %
          \stackrel{\tau_3}{\longrightarrow} \vS_4 \stackrel{\tau_4}{\longrightarrow} \vS_4
          \stackrel{\tau_5}{\longrightarrow} \vS_6 \stackrel{\tau_6}{\longrightarrow} \vS_6. %
 \end{eqnarray*}
 Since $(\vS_6,\vS_1)$ or $(\vS_1,\vS_6)$ is an oriented flag-line cross and $\tau$ leaves it invariant,
 it follows that $\tau$ must be the identity isometry. %

 To prove (\ref{eqn:arah6isom2}), first note that $\iota_0=\iota_6={\rm id}$ by (\ref{eqn:arah6isom1}). Then %
 \begin{eqnarray*}
 \eta_1\cdots\eta_6
 \!\!&=&\!\! \tau_1 (\tau_1^{-1}\tau_2\tau_1)((\tau_2\tau_1)^{-1}\tau_3(\tau_2\tau_1))%
     \cdots ((\tau_5\tau_4\tau_3\tau_2\tau_1)^{-1}\tau_6(\tau_5\tau_4\tau_3\tau_2\tau_1)) \\ %
 \!\!&=&\!\! \tau_6\tau_5\tau_4\tau_3\tau_2\tau_1 \;=\; {\rm id}. %
 \end{eqnarray*}

 To verify (\ref{eqn:arah6isom3}) and (\ref{eqn:arah6isom4}), %
 one evaluates $\iota_n^{-1}(\vS_6)$ and $\iota_n^{-1}(\vS_1)$ as follows: %
 \begin{eqnarray}
 & \iota_n^{-1}(\vS_1)=\vS_n, \quad \iota_n^{-1}(\vS_6)=\vS_{n+1}, \quad n=1,3,5; \label{eqn:iotaninv135} \\ %
 & \iota_n^{-1}(\vS_6)=\vS_n, \quad \iota_n^{-1}(\vS_1)=\vS_{n+1}, \quad n=2,4,6. \label{eqn:iotaninv246} %
 \end{eqnarray}
 This completes the proof of Lemma \ref{lem:arah6isom}.
 \end{proof}

%%%%%%%%%%%%%%%%%%%%%%%%%%%%%%%%%%%%%%%%%%%%%%%%%%%%%%%%%%%%%%%%%%%%%%%%%%%%%%%%%%%%%%%%%%%%%%%%%%%%%%%%%%%%%%%%%

 \begin{proof}[Proof of Theorem \ref{thm:gauss}]
 With no loss of generality, we may assume that $S_1, S_3, S_5$ are lines and $S_2, S_4, S_6$ are flags in $\H^4$.
 Then each $\delta_n$, $n=1,3,5$ is one of the two quaternion half side-lengths of $\{\vS_n\}_{n=1}^{6}$
 along $\vS_n$, and each $\delta_n$, $n=2,4,6$ is one of the two $e_2$-complex half side-lengths of
 $\{\vS_n\}_{n=1}^{6}$ along $\vS_n$.

 By applying an orientation-preserving isometry of $\H^4$, we may assume that %
 the oriented flag-line cross $(\vS_6, \vS_1)$ is $(\vF_{\rm h}, \vL_{\rm v})$, that is, %
 \begin{equation}
 \vS_6 = \vF_{\rm h}, \quad\quad  \vS_1 = \vL_{\rm v}. %
 \end{equation}
 Let $\tau_n \in {\rm Isom}^{+}(\H^4)$, $n=1,\cdots,6$, with indices modulo $6$, be determined by %
 $\tau_n(\vS_n)=\vS_n$ and $\tau_n(\vS_{n-1})=\vS_{n+1}$.
 Set $\iota_n=(\tau_n\cdots\tau_2\tau_1)^{-1}$ and $\eta_n = \iota_{n-1} \tau_n \iota_{n-1}^{-1}$. %
 By Lemma \ref{lem:arah6isom} we have (\ref{eqn:arah6isom2}) through (\ref{eqn:arah6isom4}). %

 By (\ref{eqn:arah6isom3}) and (\ref{eqn:arah6isom4}) we have, for $n=1,3,5$, %
 \begin{eqnarray}
 \iota_{n-1}(\vS_n)=\vS_1=\vL_{\rm v}, \quad \iota_{n-1}(\vS_{n-1})=\vS_6=\vF_{\rm h}. %
 \end{eqnarray}
 Then Proposition \ref{prop:quatdist} applies and concludes that
 $\eta_n \in {\rm Isom}^{+}(\H^4)$, $n=1,3,5$ has Vahlen matrices $\pm A_n$ where %
 \begin{eqnarray}
 A_n=\Big(\,\small\begin{matrix}\exp \delta_n & 0 \\ %
             0 & \exp(-\delta_n^*)\end{matrix}\normalsize\,\Big), \quad n = 1,3,5. %
 \end{eqnarray}
 By (\ref{eqn:arah6isom3}) and (\ref{eqn:arah6isom4}) again, we have, for $n=2,4,6$, %
 \begin{eqnarray}
 \iota_{n-1}(\vS_n)=\vS_6=\vF_{\rm h}, \quad \iota_{n-1}(\vS_{n-1})=\vS_1=\vL_{\rm v}. %
 \end{eqnarray}
 Now Proposition \ref{prop:e2dist} applies and concludes that
 $\eta_n \in {\rm Isom}^{+}(\H^4)$, $n=2,4,6$ has Vahlen matrices $\pm A_n$ where %
 \begin{eqnarray}
 A_n=\Big(\,\small\begin{matrix}\cosh\delta_n & \sinh\delta_n \\ %
             \sinh\delta_n & \cosh\delta_n\end{matrix}\normalsize\,\Big), \quad\;\ n = 2,4,6. %
 \end{eqnarray}
 By (\ref{eqn:arah6isom2}) we have $\eta_1 \eta_2 \cdots \eta_6 = {\rm id}$. %
 Hence there exists $\varepsilon \in \{-1,1\}$ such that %
 \begin{eqnarray}\label{eqn:6...21}
 A_1 A_2 \cdots A_6 = \varepsilon I, %
 \end{eqnarray}
 where $I$ is the identity $2 \times 2$ matrix, or equivalently,
 \begin{eqnarray}\label{eqn:123}
 A_1 A_2 A_3 = \varepsilon\, A_6^{-1} A_5^{-1} A_4^{-1}. %
 \end{eqnarray}
 Working out the products of matrices on both sides of (\ref{eqn:123}) and %
 equating the corresponding $(1,1)$-, $(1,2)$-, $(2,1)$- and $(2,2)$-entries, we obtain %
 \begin{eqnarray}
 & & \hspace{-60pt}\phantom{-}\exp\delta_1\cosh\delta_2\exp\delta_3 \nonumber \\ %
 &=& \varepsilon (\cosh\delta_4\exp(-\delta_5^*)\cosh\delta_6+\sinh\delta_4\exp\delta_5\sinh\delta_6)^*,\label{eqn:11} \\ %
 & & \hspace{-60pt}-\exp\delta_1\sinh\delta_2\exp(-\delta_3^*) \nonumber \\ %
 &=& \varepsilon (\sinh\delta_4\exp(-\delta_5^*)\cosh\delta_6+\cosh\delta_4\exp\delta_5\sinh\delta_6)^*,\label{eqn:12} \\ %
 & & \hspace{-60pt}-\exp(-\delta_1^*)\sinh\delta_2\exp\delta_3 \nonumber \\ %
 &=& \varepsilon (\cosh\delta_4\exp(-\delta_5^*)\sinh\delta_6+\sinh\delta_4\exp\delta_5\cosh\delta_6)^*,\label{eqn:21} \\ %
 & & \hspace{-60pt}\phantom{-}\exp(-\delta_1^*)\cosh\delta_2\exp(-\delta_3^*) \nonumber \\ %
 &=& \varepsilon (\sinh\delta_4\exp(-\delta_5^*)\sinh\delta_6+\cosh\delta_4\exp\delta_5\cosh\delta_6)^*.\label{eqn:22}    %
 \end{eqnarray}
 Now the desired formulas (\ref{eqn:1})--(\ref{eqn:4}) follow from formulas (\ref{eqn:11})--(\ref{eqn:22})
 above by performing operations (\ref{eqn:11})\,$+$\,(\ref{eqn:22}), (\ref{eqn:11})\,$-$\,(\ref{eqn:22}), %
 (\ref{eqn:21})\,$+$\,(\ref{eqn:12}), and (\ref{eqn:21})\,$-$\,(\ref{eqn:12}), respectively.
 %% This finishes the proof of Theorem \ref{thm:gauss}.
 \end{proof}

%%%%%%%%%%%%%%%%%%%%%%%%%%%%%%%%%%%%%%%%%%%%%%%%%%%%%%%%%%%%%%%%%%%%%%%%%%%%%%%%%%%%%%%%%%%%%%%%%%%%%%%%%%%%%%%%%%%%%%%%

 \begin{remark} {\rm It is not difficult to check that the formulas (\ref{eqn:1}) through (\ref{eqn:4}) %
 can be rewritten as the following formulas (\ref{eqn:c1}) through (\ref{eqn:c4}), respectively:
 \begin{eqnarray}
 & & \hspace{-64pt} \cosh \,(\delta_1 \ominus \delta_2^* \oplus \delta_3) %
                   + \cosh \,(\delta_1 \oplus  \delta_2   \oplus \delta_3) \nonumber \\ %
 &=& \varepsilon\,(\cosh \,(\delta_4 \ominus \delta_5^* \oplus \delta_6) %
                 + \cosh \,(\delta_4 \oplus \delta_5 \oplus \delta_6))^*;  \label{eqn:c1} \\ %
 & & \hspace{-64pt} \sinh \,(\delta_1 \ominus \delta_2^* \oplus \delta_3) %
                   + \sinh \,(\delta_1 \oplus  \delta_2   \oplus \delta_3) \nonumber \\ %
 &=& \varepsilon\,(\sinh \,(\delta_4 \ominus \delta_5^* \ominus \delta_6^*) %
                 - \sinh \,(\delta_4 \oplus \delta_5 \ominus \delta_6^*))^*;    \label{eqn:c2} \\ %
 & & \hspace{-64pt} \sinh \,(\delta_1 \ominus \delta_2^* \ominus \delta_3^*) %
                   - \sinh \,(\delta_1 \oplus  \delta_2 \ominus   \delta_3^*) \nonumber \\ %
 &=& \varepsilon\,(\sinh \,(\delta_4 \ominus \delta_5^* \oplus \delta_6) %
                 + \sinh \,(\delta_4 \oplus \delta_5 \oplus \delta_6))^*;  \label{eqn:c3} \\ %
 & & \hspace{-64pt} \cosh \,(\delta_1 \oplus \delta_2 \ominus \delta_3^*) %
                    -\cosh \,(\delta_1 \ominus \delta_2^* \ominus \delta_3^*) \nonumber \\ %
 &=& \varepsilon\,(\cosh \,(\delta_4 \oplus \delta_5 \ominus \delta_6^*) %
                  -\cosh \,(\delta_4 \ominus \delta_5^* \ominus \delta_6^*))^*.  \label{eqn:c4}%
 \end{eqnarray} }
 \end{remark}

 \begin{remark} {\rm Taking the reverse involution $\delta \mapsto \delta^*$,
 we see that the four identities are left invariant,
 with a simple shift of indices: $123456 \rightarrow 456123$.}
 \end{remark}

%
% \section{\bf A special case of the Delambre-Gauss formulas}\label{s:case} %
%
% To be written up ...

%%%%%%%%%%%%%%%%%%%%%%%%%%%%%%%%%%%%%%%%%%%%%%%%%%%%%%%%%%%%%%%%%%%%%%%%%%%%%%%%%%%%%%%%%%%%%%%%%%%%%%%%%%%%%%%%%%%%%
%%%%%%%%%%%%%%%%%%%%%%%%%%%%%%%%%%%%%%%%%%%%%%%%%%%%%%%%%%%%%%%%%%%%%%%%%%%%%%%%%%%%%%%%%%%%%%%%%%%%%%%%%%%%%%%%%%%%%

 \section{\bf Appendix: Generalized Delambre-Gauss formulas for oriented right-angled hexagons in $\H^3$}\label{s:gauss4rah} %

 \subsection{Delambre-Gauss formulas for spherical triangles}

 In spherical trigonometry there are for spherical triangles the important Delambre's analogies or Gauss formulas, %
 (\ref{eq:gaussspher1})--(\ref{eq:gaussspher4}) below, which we call the Delambre-Gauss formulas in this paper.
 These formulas were discovered by Delambre in 1807, published in 1809, and were subsequently discovered independently
 by Gauss. They play an important role in deriving many other important formulas in spherical trigonometry. %
 As a few examples, the Napier's analogies and the law of tangents for spherical triangles follow directly, %
 and one can derive from them the beautiful L'Huillier Theorem (see \cite{casey}, page 44). %
 We leave it to the reader to derive from  the Delambre-Gauss formulas the law of cosines and the law of sines
 for spherical triangles. %

 In the following theorem and its two corollaries we consider a spherical triangle in the unit sphere having
 side-lengths $a,b,c\in (0,\pi)$ and corresponding opposite interior angles $\alpha,\beta,\gamma\in (0,\pi)$. %

 \begin{theorem}[Delambre-Gauss formulas]\label{thm:gaussspher}
 In a spherical triangle we have
 \begin{eqnarray}\label{eq:gaussspher}
    \cos\textstyle\frac12(a+b)\,\sin\textstyle\frac12\gamma
&=& \cos\textstyle\frac12(\alpha+\beta)\,\cos\textstyle\frac12 c, \label{eq:gaussspher1} \\ %
    \sin\textstyle\frac12(a+b)\,\sin\textstyle\frac12\gamma
&=& \cos\textstyle\frac12(\alpha-\beta)\,\sin\textstyle\frac12 c, \label{eq:gaussspher2} \\ %
    \cos\textstyle\frac12(a-b)\,\cos\textstyle\frac12\gamma
&=& \sin\textstyle\frac12(\alpha+\beta)\,\cos\textstyle\frac12 c, \label{eq:gaussspher3} \\ %
    \sin\textstyle\frac12(a-b)\,\cos\textstyle\frac12\gamma
&=& \sin\textstyle\frac12(\alpha-\beta)\;\sin\textstyle\frac12 c. \label{eq:gaussspher4}    %
 \end{eqnarray}
 \end{theorem}

 \begin{remark}
 {\rm It is easy to show that $a + b > \pi$ if and only if $\alpha + \beta > \pi$.}
 \end{remark}
% And it is easy to see that the Napier's analogies and the spherical law of tangents for spherical triangles
% follow directly from the Delambre-Gauss formulas.

 \begin{corollary}[Napier's analogies]\label{thm:napierspher}
 In a spherical triangle we have
 \begin{eqnarray}\label{eq:napierspher}
    \sin\textstyle\frac12(\alpha-\beta) \,/ \sin\textstyle\frac12(\alpha+\beta)
&=& \tan\textstyle\frac12(a-b) \,/ \tan\textstyle\frac12 c,   \label{eq:napierspher1} \\ %
    \cos\textstyle\frac12(\alpha-\beta) \,/ \cos\textstyle\frac12(\alpha+\beta)
&=& \tan\textstyle\frac12(a+b) \,/ \tan\textstyle\frac12 c,   \label{eq:napierspher2} \\ %
    \sin\textstyle\frac12(a-b) \,/ \sin\textstyle\frac12(a+b)
&=& \tan\textstyle\frac12(\alpha-\beta) \,/ \cot\textstyle\frac12 \gamma, \label{eq:napierspher3} \\ %
    \cos\textstyle\frac12(a-b) \,/ \cos\textstyle\frac12(a+b)
&=& \tan\textstyle\frac12(\alpha+\beta) \,/ \cot\textstyle\frac12 \gamma. \label{eq:napierspher4}    %
 \end{eqnarray}
 \end{corollary}

 \begin{corollary}[Law of tangents]\label{thm:lotspher}
 In a spherical triangle we have
 \begin{eqnarray}\label{eq:lotspher}
      \tan\textstyle\frac12(a-b) \,/ \tan\textstyle\frac12(a+b)
\;=\; \tan\textstyle\frac12(\alpha-\beta) \,/ \tan\textstyle\frac12(\alpha+\beta) %
 \end{eqnarray}
 and two other similar formulas. % (the Law of Tangents for the given triangle).
 \end{corollary}

 \subsection{Delambre-Gauss formulas for hyperbolic triangles and convex planar right-angled hexagons}

 We may obtain the Delambre-Gauss formulas for hyperbolic triangles (and consequently, the Napier's analogies
 as well as the law of tangents) by simply changing all the appearance of trigonometric functions of side-lengths
 in the identities for spherical triangles over to the corresponding hyperbolic trigonometric functions of
 the same side-lengths of hyperbolic triangles.

 \begin{theorem}[Delambre-Gauss formulas for hyperbolic triangles]\label{thm:gausshyper}
 For a triangle in the hyperbolic plane having side-lengths $a,b,c >0$ and corresponding
 opposite interior angles $\alpha,\beta,\gamma\in (0,\pi)$, the following formulas hold: %
 \begin{eqnarray}\label{eq:gausshyper}
    \cosh\textstyle\frac12(a+b)\,\sin\textstyle\frac12 \gamma
&=& \cos\textstyle\frac12(\alpha+\beta)\,\cosh\textstyle\frac12 c,  \label{eq:gausshyper1} \\ %
    \sinh\textstyle\frac12(a+b)\,\sin\textstyle\frac12 \gamma
&=& \cos\textstyle\frac12(\alpha-\beta)\,\sinh\textstyle\frac12 c,  \label{eq:gausshyper2} \\ %
    \cosh\textstyle\frac12(a-b)\,\cos\textstyle\frac12 \gamma
&=& \sin\textstyle\frac12(\alpha+\beta)\,\cosh\textstyle\frac12 c,  \label{eq:gausshyper3} \\ %
    \sinh\textstyle\frac12(a-b)\,\cos\textstyle\frac12 \gamma
&=& \sin\textstyle\frac12(\alpha-\beta)\;\sinh\textstyle\frac12 c.  \label{eq:gausshyper4}    %
 \end{eqnarray}
 \end{theorem}

 We may also obtain Delambre-Gauss formulas for convex right-angled hexagons in the hyperbolic plane %
 as follows. %
 \begin{theorem}[Delambre-Gauss formulas for convex right-angled hexagons in the hyperbolic plane]\label{thm:gaussprah} %
 For a convex right-angled hexagon in the hyperbolic plane with side-lengths $\{ l_n > 0 \}_{n=1}^{6}$, %
 the following formulas hold: %
 \begin{eqnarray}\label{eq:gauss4prah}
     \cosh\textstyle\frac12(l_1+l_3)\,\sinh\textstyle\frac12 l_2
 &=& \cosh\textstyle\frac12(l_4+l_6)\,\cosh\textstyle\frac12 l_5,  \label{eq:prah1} \\ %
     \sinh\textstyle\frac12(l_1+l_3)\,\sinh\textstyle\frac12 l_2
 &=& \cosh\textstyle\frac12(l_4-l_6)\,\sinh\textstyle\frac12 l_5,  \label{eq:prah2} \\ %
     \cosh\textstyle\frac12(l_1-l_3)\,\cosh\textstyle\frac12 l_2
 &=& \sinh\textstyle\frac12(l_4+l_6)\,\cosh\textstyle\frac12 l_5,  \label{eq:prah3} \\ %
     \sinh\textstyle\frac12(l_1-l_3)\,\cosh\textstyle\frac12 l_2
 &=& \sinh\textstyle\frac12(l_4-l_6)\;\sinh\textstyle\frac12 l_5.  \label{eq:prah4}    %
 \end{eqnarray}
 \end{theorem}

 \begin{remark}
 {\rm It is easy to show that $l_1 < l_3$ if and only if $l_4 < l_6$.}
 \end{remark}

 \subsection{Delambre-Gauss formulas for oriented right-angled hexagons in $\H^3$}\label{ss:dg4rah} %

 By definition, a right-angled hexagon in $\H^3$ is a six-tuple $\{L_n\}_{n=1}^{6}$ of lines in $\H^3$ such that, for
 each $n$ modulo $6$, the two lines $L_n$ and $L_{n+1}$ intersect perpendicularly. %
 An oriented right-angled hexagon $\{\vL_n\}_{n=1}^{6}$ in $\H^3$ is obtained from a right-angled hexagon
 $\{L_n\}_{n=1}^{6}$ in $\H^3$ by orienting each side-line $L_n$ as $\vL_n$ for $n=1,\cdots,6$. %

 \begin{definition}\label{defn:side-length}
 {\rm For an oriented right-angled hexagon $\{\vL_n\}_{n=1}^{6}$ in $\H^3$,
 %% $(\vL_{n-1}, \vL_n, \vL_{n+1})$, (which Fenchel \cite{fenchel1989book} calls a double cross),
 we define its complex side-length $\sigma_n=\sigma_{\vL_n}(\vL_{n-1},\vL_{n+1})$ along $\vL_n$ as follows.
 First, there exist a unique orientation-preserving isometry $\eta$ of $\H^3$ so that
 $$ \eta(\vL_n)=\vL_{[0,\infty]} \;\;\ \text{and} \;\;\ \eta(\vL_{n-1})=\vL_{[-1,1]}.$$ %
 Since the lines $\eta(L_n)$ and $\eta(L_{n+1})$ intersect perpendicularly, $\eta(\vL_{n+1})$
 is of the form $\vL_{[-z,z]}$ for some $z \in \C \backslash \{0\}$. Then we define the complex side-length by
 \begin{eqnarray}
 \sigma_n \,=\, \sigma_{\vL_n}(\vL_{n-1},\vL_{n+1}) \,=\, \log z \in \C / 2\pi i\Z.
 \end{eqnarray}}
 \end{definition}
 It is clear that the complex side-lengths $\{\sigma_n\}_{n=1}^{6}$ are invariant under orientation-preserving
 isometries of $\H^3$. Furthermore, dividing $\sigma_n \in \C / 2\pi i\Z$ by $2$, we obtain two half side-lengths,
 $\delta_n$ and $\delta_n+\pi i$, in $\C / 2\pi i\Z$.

 We obtain Delambre-Gauss formulas for oriented right-angled hexagons in $\H^3$.

 \begin{theorem}[Delambre-Gauss formulas for oriented right-angled hexagons in $\H^3$]\label{thm:gauss4rah}
 Given an oriented right-angled hexagon $\{\vL_n\}_{n=1}^{6}$ in $\H^3$, let $\delta_n \in \C / 2\pi i\Z$ be one of %
 the two complex half side-lengths of $\{\vL_n\}_{n=1}^{6}$ along $\vL_n$. Then %
 \begin{eqnarray}\label{eq:gauss4rah}
   \cosh(\delta_1+\delta_3) \cosh\delta_2 &=& \varepsilon\,\cosh(\delta_4+\delta_6) \cosh\delta_5, \label{eq:rah1} \\ %
  -\sinh(\delta_1+\delta_3) \cosh\delta_2 &=& \varepsilon\,\cosh(\delta_4-\delta_6) \sinh\delta_5, \label{eq:rah2} \\ %
  -\cosh(\delta_1-\delta_3) \sinh\delta_2 &=& \varepsilon\,\sinh(\delta_4+\delta_6) \cosh\delta_5, \label{eq:rah3} \\ %
   \sinh(\delta_1-\delta_3) \sinh\delta_2 &=& \varepsilon\,\sinh(\delta_4-\delta_6) \sinh\delta_5, \label{eq:rah4}    %
 \end{eqnarray}
 with $\varepsilon = 1$ or $-1$, depending on the choices of the six half side-lengths $\{\delta_n\}_{n=1}^{6}$.
 \end{theorem}

 Our proof of Theorem \ref{thm:gauss4rah} uses an identity involving six isometries of $\H^3$, %
 (\ref{eqn:rah6isom2}) in Lemma \ref{lem:rah6isom} below, associated to an oriented right-angled hexagon in $\H^3$. %

%%%%%%%%%%%%%%%%%%%%%%%%%%%%%%%%%%%%%%%%%%%%%%%%%%%%%%%%%%%%%%%%%%%%%%%%%%%%%%%%%%%%%%%%%%%%%%%%%%%%%%%%%%%%%%%%%%%%%%%

 \begin{lemma}\label{lem:rah6isom}
 Given an oriented right-angled hexagon $\{\vL_n\}_{n=1}^{6}$ in $\H^3$,
 let $\tau_n \in {\rm Isom}^{+}(\H^3)$, $n=1,\cdots,6$ be determined by %
 $\tau_n(\vL_n)=\vL_n$ and $\tau_n(\vL_{n-1})=\vL_{n+1}$, with indices modulo $6$, %
 and set $\eta_n = (\tau_{n-1}\cdots\tau_1)^{-1}\tau_n(\tau_{n-1}\cdots\tau_1)$. %
 Then %
 \begin{eqnarray}
 \tau_{6} \cdots \tau_{1} &\!\!\!=\!\!\!& {\rm id}, \label{eqn:rah6isom1} \\ %
 \eta_{1} \cdots \eta_{6} &\!\!\!=\!\!\!& {\rm id}, \label{eqn:rah6isom2} %
 \end{eqnarray}
 and furthermore, $\eta_n(\vL_1)=\vL_1$ for $n=1,3,5$, and $\eta_n(\vL_{6})=\vL_{6}$ for $n=2,4,6$. %
 \end{lemma}

 \begin{proof}
 To prove (\ref{eqn:rah6isom1}), write
 $\tau=\tau_{6} \tau_{5} \tau_{4} \tau_{3} \tau_{2} \tau_{1}$. %
 It can be checked that $\tau(\vL_6)=\vL_6$ and %
 $\tau(\vL_1)=\vL_1$ by going through the following table: %
 \begin{eqnarray*}
 && \vL_1 \stackrel{\tau_1}{\longrightarrow} \vL_1 \stackrel{\tau_2}{\longrightarrow} \vL_3 %
          \stackrel{\tau_3}{\longrightarrow} \vL_3 \stackrel{\tau_4}{\longrightarrow} \vL_5
          \stackrel{\tau_5}{\longrightarrow} \vL_5 \stackrel{\tau_6}{\longrightarrow} \vL_1, \\ %
 && \vL_6 \stackrel{\tau_1}{\longrightarrow} \vL_2 \stackrel{\tau_2}{\longrightarrow} \vL_2 %
          \stackrel{\tau_3}{\longrightarrow} \vL_4 \stackrel{\tau_4}{\longrightarrow} \vL_4
          \stackrel{\tau_5}{\longrightarrow} \vL_6 \stackrel{\tau_6}{\longrightarrow} \vL_6. %
 \end{eqnarray*}
 Since oriented lines $\vL_6$ and $\vL_1$ intersect perpendicularly and $\tau \in {\rm Isom}^{+}(\H^3)$ %
 leaves each of them invariant, it follows that $\tau$ must be the identity isometry. %

 One verifies (\ref{eqn:rah6isom2}) by direct calculation: %
 \begin{eqnarray*}
 \eta_1\cdots\eta_6
 \!\!&=&\!\! \tau_1 (\tau_1^{-1}\tau_2\tau_1)((\tau_2\tau_1)^{-1}\tau_3(\tau_2\tau_1))%
     \cdots ((\tau_5\tau_4\tau_3\tau_2\tau_1)^{-1}\tau_6(\tau_5\tau_4\tau_3\tau_2\tau_1)) \\ %
 \!\!&=&\!\! \tau_6\tau_5\tau_4\tau_3\tau_2\tau_1 \;=\; {\rm id}. %
 \end{eqnarray*}

 The verification of the ``furthermore'' part of the lemma is easy. %
%% This completes the proof of Lemma \ref{lem:rah6isom}.
 \end{proof}

%%%%%%%%%%%%%%%%%%%%%%%%%%%%%%%%%%%%%%%%%%%%%%%%%%%%%%%%%%%%%%%%%%%%%%%%%%%%%%%%%%%%%%%%%%%%%%%%%%%%%%%%%%%%%%%%%%%%%%%

 \begin{proof}[Proof of Theorem \ref{thm:gauss4rah}]

 Since the $\delta_n$'s are invariant under orientation-preserving isometries of $\H^3$, %
 by applying such an isometry, we may assume that %
 \begin{equation}
 \vL_6 = \vL_{[-1,1]}, \quad \vL_1 = \vL_{[0,\infty]}. %
 \end{equation}
 Then the isometry $\eta_n \in {\rm Isom}^{+}(\H^3)$, $n=1,\cdots,6$ defined in Lemma \ref{lem:rah6isom} %
 is given by matrices $\pm A_n \in {\rm SL}(2,\C)$ where %
 \begin{eqnarray*}
 && A_n=\Big(\,\small\begin{matrix}\exp \delta_n & 0 \\ %
             0 & \exp(-\delta_n)\end{matrix}\normalsize\,\Big), \quad n = 1,3,5; \\%
 && A_n=\Big(\,\small\begin{matrix}\cosh\delta_n & \sinh\delta_n \\ %
             \sinh\delta_n & \cosh\delta_n\end{matrix}\normalsize\,\Big), \quad\;\ n = 2,4,6. %
 \end{eqnarray*}
 By (\ref{eqn:rah6isom2}) there holds $\eta_{1} \cdots \eta_{6}={\rm id}$. %
 Hence there exists $\varepsilon \in \{-1,1\}$ such that
 \begin{eqnarray}
 A_1 A_2 A_3 A_4 A_5 A_6 = \varepsilon I, %
 \end{eqnarray}
 where $I$ is the identity $2 \times 2$ matrix, or equivalently, %
 \begin{eqnarray}\label{eq3:123}
 A_1 A_2 A_3 \,=\, \varepsilon\, A_6^{-1} A_5^{-1} A_4^{-1}. %
 \end{eqnarray}
 Working out the products of matrices on both sides of (\ref{eq3:123}) and equating %
 the corresponding $(1,1)$-, $(1,2)$-, $(2,1)$- and $(2,2)$-entries, we obtain %
 \begin{eqnarray}
 & & \hspace{-60pt} \phantom{-}\exp\delta_1\cosh\delta_2\exp\delta_3 \nonumber \\ %
 &=& \varepsilon (\cosh\delta_4\exp(-\delta_5)\cosh\delta_6+\sinh\delta_4\exp\delta_5\sinh\delta_6),\label{eq3:11} \\ %
 & & \hspace{-60pt} -\exp\delta_1\sinh\delta_2\exp(-\delta_3) \nonumber \\ %
 &=& \varepsilon (\sinh\delta_4\exp(-\delta_5)\cosh\delta_6+\cosh\delta_4\exp\delta_5\sinh\delta_6),\label{eq3:12} \\ %
 & & \hspace{-60pt} -\exp(-\delta_1)\sinh\delta_2\exp\delta_3 \nonumber \\ %
 &=& \varepsilon (\cosh\delta_4\exp(-\delta_5)\sinh\delta_6+\sinh\delta_4\exp\delta_5\cosh\delta_6),\label{eq3:21} \\ %
 & & \hspace{-60pt} \phantom{-}\exp(-\delta_1)\cosh\delta_2\exp(-\delta_3) \nonumber \\ %
 &=& \varepsilon (\sinh\delta_4\exp(-\delta_5)\sinh\delta_6+\cosh\delta_4\exp\delta_5\cosh\delta_6).\label{eq3:22} %
 \end{eqnarray}
 Now the desired formulas (\ref{eq:rah1})--(\ref{eq:rah4}) follow from formulas (\ref{eq3:11})--(\ref{eq3:22}) %
 above by performing operations (\ref{eq3:22})\,$+$\,(\ref{eq3:11}), (\ref{eq3:22})\,$-$\,(\ref{eq3:11}), %
 (\ref{eq3:21})\,$+$\,(\ref{eq3:12}) and (\ref{eq3:21})\,$-$\,(\ref{eq3:12}), respectively.
 \end{proof}

 %%%%%%%%%%%%%%%%%%%%%%%%%%%%

 \begin{remark}
 {\rm It is interesting to observe that, by suitably changing orientations of the side-lines,
 one may obtain from the single formula (\ref{eq:rah1}) the other three formulas (\ref{eq:rah2})--(\ref{eq:rah4}). %
 The rule is that, if we change the orientation for only one $\vL_n$ to obtain a new oriented right-angled hexagon
 $\{\vL'_1, \cdots, \vL'_6\}$ and choose the corresponding new half side-lengths $\delta'_1, \cdots, \delta'_6$
 as follows (with indices modulo 6):
 \begin{eqnarray}
 && \delta'_{n-1} = \delta_{n-1} - \textstyle\frac{\pi}{2}i,
    \quad \delta'_n = -\delta_n,
    \quad\, \delta'_{n+1} = \delta_{n+1} + \textstyle\frac{\pi}{2}i, \nonumber \\ %
 && \delta'_{n+2} = \delta_{n+2}, \quad\quad\; \delta'_{n+3} = \delta_{n+3}, \quad \delta'_{n+4} = \delta_{n+4}, %
 \end{eqnarray}
 then it is easily checked that $\varepsilon'=\varepsilon$.
 Following this rule, one obtains formulas (\ref{eq:rah2}), (\ref{eq:rah3}) and (\ref{eq:rah4}) from (\ref{eq:rah1}) %
 by simply changing the orientation of $\vL_6$, that of $\vL_3$, and those of $\vL_6$ and $\vL_3$, respectively. } %
% (with the new $\varepsilon$ being $-\varepsilon$, $-\varepsilon$ and $\varepsilon$, accordingly).} %
 \end{remark}

%%%%%%%%%%%%%%%%%%%%%%%%%%%%%%%%%%%%%%%%%%%%%%%%%%%%%%%%%%%%%%%%%%%%%%%%%%%%%%%%%%%%%%%%%%%%%%%%%%%%%%%%%%%%%%%%%%%%%
%%%%%%%%%%%%%%%%%%%%%%%%%%%%%%%%%%%%%%%%%%%%%%%%%%%%%%%%%%%%%%%%%%%%%%%%%%%%%%%%%%%%%%%%%%%%%%%%%%%%%%%%%%%%%%%%%%%%%

\end{document}